\pdfoutput=1
\documentclass[a4paper,english]{jnsao}
\usepackage[utf8]{inputenc}
\usepackage{enumitem}

\renewcommand{\phi}{\varphi}
\newcommand{\N}{\mathbb{N}}

\newcommand{\R}{\mathbb{R}}
\newcommand{\calF}{\mathcal{F}}
\newcommand{\calG}{\mathcal{G}}
\newcommand{\calE}{\mathcal{E}}

\newcommand{\calS}{\mathcal{S}}

\newcommand{\calM}{{\mathcal{M}}}
\newcommand{\calO}{{\mathcal{O}}}
\newcommand{\scalprod}[1]{\langle #1 \rangle}
\newcommand{\abs}[1]{| #1 |}
\newcommand{\norm}[1]{\| #1 \|}
\newcommand{\set}[2]{\left\{#1:#2\right\}}

\DeclareMathOperator{\sign}{\mathrm{sign}}
\DeclareMathOperator{\dom}{\mathrm{dom}}
\DeclareMathOperator{\ran}{\mathrm{ran}}
\DeclareMathOperator{\Id}{\mathrm{Id}}
\DeclareMathOperator{\co}{\mathrm{co}}
\DeclareMathOperator*{\argmin}{\mathrm{arg\,min}}

\newcommand{\prox}{\mathrm{prox}}

\DeclareMathOperator{\divergence}{div}
\DeclareMathOperator{\grad}{grad}

\DeclareMathOperator{\extR}{\overline{\R}}
\newcommand{\defgl}{\mathrel{\mathop:}=}

\newcommand{\dd}{{\mathrm d}}
\renewcommand{\u}{{\bar u}}

\usepackage[centercolon]{mathtools}

\usepackage{siunitx}
\usepackage{booktabs}

\usepackage{graphicx}
\usepackage{subcaption}
\usepackage{siunitx}
\usepackage{tikz,pgfplots}
\pgfplotsset{compat=newest}
\pgfplotsset{plot coordinates/math parser=false}
\usetikzlibrary{arrows.meta}
\usetikzlibrary{plotmarks}
\usepgfplotslibrary{polar}

\usepackage[nameinlink,capitalize]{cleveref}

\manuscripteprinttype{arxiv}
\manuscripteprint{2108.10077}
\date{2021-08-23}
\manuscriptlicense{}
\manuscriptcopyright{}

\title{Convex relaxation of discrete vector-valued optimization problems}
\author{Christian Clason\thanks{Faculty of Mathematics, Universit\"at Duisburg-Essen, 45117 Essen, Germany; \emph{current address:} Institute of Mathematics and Scientific Computing, University of Graz, Heinrichstrasse 36, 8010 Graz, Austria (\email{c.clason@uni-graz.at}, \orcid{0000-0002-9948-8426})}
    \and
    Carla Tameling\thanks{Institute for Mathematical Stochastics, Universit\"at G\"ottingen, Goldschmidtstr.~7,
    37077 G\"ottingen, Germany (\email{carla.tameling@mathematik.uni-goettingen.de})}
    \and
    Benedikt Wirth\thanks{Applied Mathematics, Universit\"at M\"unster, Einsteinstr.~62,
    48149 M\"unster, Germany (\email{benedikt.wirth@uni-muenster.de}, \orcid{0000-0003-0393-1938})}
}
\shortauthor{Clason, Tameling, Wirth}

\AtBeginDocument{
    \hypersetup{
        pdftitle={Convex relaxation of discrete vector-valued optimization problems},
        pdfauthor={Christian Clason, Carla Tameling, Benedikt Wirth},
        pdfkeywords={multibang control, convex relaxation, Bloch equation, linear elasticity, branched transport, semismooth Newton method}
    }
}

\acknowledgments{{CC} was supported by the German Science Fund (DFG) under grant CL 487/1-1. CT gratefully acknowledges support by the DFG Research Training Group 2088 Project A1. BW is supported by the German Science Fund (DFG) under the priority program SPP 1962, grant WI 4654/1-1, and under Germany's Excellence Strategy EXC 2044 -- 390685587, Mathematics M\"unster: Dynamics--Geometry--Structure.}

\begin{document}

\maketitle

\begin{abstract}
    We consider a class of infinite-dimensional optimization problems in which a distributed vector-valued variable should pointwise almost everywhere take values from a given finite set $\mathcal{M}\subset\mathbb{R}^m$. Such hybrid discrete--continuous problems occur in, e.g., topology optimization or medical imaging and are challenging due to their lack of weak lower semicontinuity.
    To circumvent this difficulty, we introduce as a regularization term a convex integral functional with an integrand that has a polyhedral epigraph with vertices corresponding to the values of $\mathcal{M}$; similar to the $L^1$ norm in sparse regularization, this ``vector multibang penalty'' promotes solutions with the desired structure while allowing the use of tools from convex optimization for the analysis as well as the numerical solution of the resulting problem.

    We show well-posedness of the regularized problem and analyze stability properties of its solution in a general setting. We then illustrate the approach for three specific model optimization problems of broader interest:
    optimal control of the Bloch equation, optimal control of an elastic deformation, and a multimaterial branched transport problem. In the first two cases, we derive explicit characterizations of the penalty and its generalized derivatives for a concrete class of sets $\mathcal{M}$. For the third case, we discuss the algorithmic computation of these derivatives for general sets. These derivatives are then used in a superlinearly convergent semismooth Newton method applied to a sequence of regularized optimization problems.

    We illustrate the behavior of this approach for the three model problems with numerical examples.
\end{abstract}

\section{Introduction}\label{sec:introduction}

Many optimization problems involve minimizing the distance of a quantity $S(u)$ to some given $z$, where $u$ is the optimization variable and $S(u)$ denotes the output of some model depending on $u$.
This may arise either from an optimal control problem, in which we try to choose the control $u$ such that the state $y=S(u)$ -- commonly the solution to a differential equation -- comes close to a desired state $z$,
or from an inverse problem, in which a measurement $z$ has been obtained via a forward operator $S$ from a physical configuration $u$, which we try to recover.
Typically $u$ is from an infinite-dimensional function space.
To make the problem well-posed, a regularization usually has to be incorporated, which encodes some a priori knowledge or requirement of $u$.
Such a priori knowledge could for instance be the fact that $u$ pointwise almost everywhere takes values only from a prescribed finite set $\calM\subset\R^m$ for some $m\in\N$,
which is the situation we will focus on.
Examples include topology optimization, where the spatial material composition of a (mechanical) structure is optimized
and in which $\calM$ comprises the material parameters of the available material components,
or inverse problems in which the spatial distribution of a few known materials (or, in medical imaging, tissues with known properties) has to be identified.
Our goal is to achieve this using a \emph{convex} regularization so that we can apply elegant and powerful tools from convex optimization for its analysis and numerical solution.

Specifically, in this work we consider the optimization problem
\begin{equation}\label{eq:optControlPb}
    \min_{u \in U} \frac12\norm{S(u) - z}_Y^2 + \int_{\Omega} g(u(x))\,\dd x,
\end{equation}
where $\Omega \subset \R^n$ is an open bounded domain, $U=L^2(\Omega;\R^m)$ for some $m\geq 2$, $Y$ is a Hilbert space, $z\in Y$, $S\colon U \to Y$ is a compact and Fr\'{e}chet differentiable (possibly nonlinear) operator, and the pointwise \emph{vector multibang penalty} $g \colon \R^m \to \R \cup \left\{ \infty\right\}$ has a convex polyhedral epigraph and superlinear growth at infinity. This extends the class of scalar problems considered in \cite{CK:2013,CK:2015} to the vector-valued case, and our main interest in this article is the behavior and influence of this vector multibang penalty on the solution, which we will study by way of examples for three different operators $S$ (the solution operators of the Bloch equation and of linearized elasticity as well as the graph divergence for a branched transport model) and specific costs $g$ (whose graph is given by a polyhedral cone, a square frustum, and a more general polyhedron in $\R^{m+1}$, respectively).
The basic underlying intuition for our specific choice of the term $\int_{\Omega} g(u(x))\,\dd x$ is that this regularization in combination with a quadratic discrepancy term increasingly promotes values of $u$ on lower-dimensional facets and, in particular, at the vertices of the graph of $g$, since the linear growth away from a vertex will lead to a comparatively greater increase in the penalty than the corresponding decrease in the discrepancy term. The same mechanism is responsible for the sparsity-promoting property (i.e., the preference for $u=0$) of $L^1$ regularization; it is also related to the fact that in linear optimization, minima are always found at a vertex of the polyhedral feasible set.
The central idea of our approach is to design the penalty $g$ such that these vertices correspond precisely to the elements of the set $\calM$, which we will make more precise in the following.

\paragraph{Motivation}
Formulating the original optimal control or inverse problem directly over the set of discrete-valued desired solutions leads to the minimization of the energy
\begin{equation*}
    \calE^\calM(u)= \frac12\norm{S(u) - z}_Y^2 + \int_{\Omega} \delta_\calM(u(x))\,\dd x
    \qquad\text{with }
    \delta_\calM(v)=\begin{cases}0&\text{if }v\in\calM,\\\infty&\text{otherwise.}\end{cases}
\end{equation*}
Unfortunately, $\calE^\calM$ is not weakly lower semicontinuous \cite[Cor.~2.14]{Braides:2002}
so that the problem is ill-posed (unless the inverse operator $S^{-1}$ is compact into $L^1(\Omega;\R^m)$,
in which case the energy is strongly coercive in $L^1(\Omega;\R^m)$ and one would only require strong lower semicontinuity):
generically there are no minimizers, and controls $u$ with small energy $\calE^\calM(u)$ will rapidly oscillate between different values in $\calM$.
There are (at least) two possible ways out:
\begin{enumerate}[label=(\roman*)]
    \item
        The first approach adds a penalty of variations of $u$, for instance the total variation seminorm $\norm{u}_{TV}=\int_\Omega\dd|\nabla u|$ or a Mumford--Shah-type regularization functional,
        which has the effect of preventing oscillations and penalizing the interfaces between regions of different values of $u$.
        A disadvantage of this approach is that it quite explicitly regularizes the geometry of the material distribution, which is the sought quantity.
        For instance, such a regularization will lead to rounded-off interfaces that cannot have corners.
    \item
        The second approach considers instead the relaxation (i.e., the lower semi-continuous envelope) of $\calE^\calM$,
        thereby admitting also mixed control values $u(x)\notin\calM$ that represent mixtures of values in $\calM$.
        This is an obvious disadvantage; however, it might be alleviated by adding a convex (to ensure weak lower semicontinuity) cost $\int_\Omega c(u(x))\,\dd x$ that may, for instance, encode a known preference for a certain material.
        If this is done before relaxation, then mixed control values will no longer have equal costs to pure control values so that the relaxation may again lead to pure control values $u(x)\in\calM$. This has for instance been observed in \cite{CK:2013}.
\end{enumerate}
The additional cost regularization of the latter approach acts on the material amounts rather than the geometry of their distribution
and therefore is worthwhile studying as an alternative to the standard regularization via penalization of interfaces.
Specifically, the relaxation of $\int_\Omega \delta_\calM(u(x))\,\dd x+\alpha \int_\Omega c(u(x))\,\dd x$ for some $\alpha >0$ and $c:\R^m\to\R$ nonnegative, strictly convex, and lower semicontinuous, is given by $\int_\Omega g(u(x))\,\dd x$ with
\begin{equation}\label{eq:reg}
    g=g_\infty^{**}\qquad\text{for }g_\infty:=\alpha c+\delta_\calM,
\end{equation}
where the double asterisk denotes the biconjugate or convex envelope.
The functions $g$ are precisely those with a convex polyhedral epigraph (since this epigraph is the convex hull of the finitely many points $\{(v,\alpha c(v))):v\in\calM\}$, and any function $g$ with convex polyhedral epigraph can be obtained via $c=g/\alpha$ and an appropriate choice of $\calM$), which motivates our problem formulation \eqref{eq:optControlPb}.
While our theoretical results will hold for any such choice of $c$, the explicit computation of $g$ and the numerical solution will be carried out as in \cite{CIK:2014,CK:2013,CK:2015} mostly for the two specific choices
\begin{equation*}
    c(v)=\frac12|v|_2^2 \qquad\text{and}\qquad c(v) = |v|_2.
\end{equation*}

In the case of a scalar function $u$ (i.e., for $m=1$) and the first choice of $c$, this optimization problem reduces to the one considered in \cite{CK:2013};
the difference in the vector-valued case is that now several (or even all) values in $\calM$ can be assigned the same control cost,
therefore allowing for multiple equally preferred discrete values.
Providing explicit and numerically implementable characterizations of the required generalized derivatives is one of the main contributions of this work. Furthermore, we provide an extended analysis of the stability and multibang properties of the optimal controls in the general case.

\paragraph{Model problems}
To illustrate the broad applicability of the proposed approach, we consider as specific examples three different forward operators $S$ and admissible sets $\calM$
(the analysis in \crefrange{sec:existence}{sec:penalty} will be independent of these models, though, beyond some general assumptions).

The first example follows \cite{Spincontrol:15}, where the authors try to drive a collection of spin systems using external electromagnetic fields to a desired spin state in the context of NMR spectroscopy or tomography.
The hardware here only allows a discrete set of control values (the radiofrequency pulse phases and amplitudes).
The underlying model is given by the Bloch equation in a rotating reference frame without relaxation (see \cite{Epstein:2006} for an introduction), which relates the magnetization vector $\mathbf M:[0,T]\to\R^3$ and the applied magnetic field $\mathbf B:[0,T]\to \R^3$ via the bilinear differential equation
\begin{equation*}
    \frac{\dd}{\dd t}{\mathbf{M}}(t) = \mathbf{M}(t) \times \mathbf B(t),
    \qquad\mathbf{M}(0)=\mathbf{M}_0.
\end{equation*}
The goal is to shift the magnetization vector from the initial state $\mathbf{M}_0$ (e.g., aligned to a strong external field) to a desired state $\mathbf{M}_d$ (e.g., orthogonal to the external field) at time $T$.
The control $u\in L^2((0,T);\R^2)$ enters the equation as $\mathbf{B}(t) = (u_1(t),u_2(t),\omega)$, where $\omega$ is a fixed resonance frequency (which coincides with the rotation frequency of the domain),
and thus the (nonlinear) operator $S$ maps the control $u$ onto the magnetization vector $\mathbf M(T)$ at time $T$.
For details, see \cref{sec:BlochStateEq}.

The second example deals with linearized elasticity as the most basic model of coupled PDEs as state equations, i.e., we consider $S$ to be the solution operator of the elliptic problem
\begin{equation*}
    \left\{\begin{aligned}
            -2\mu \divergence \epsilon(y) - \lambda \grad \divergence y &= u \text{ in } \Omega, \\
            y &= 0 \text{ on } \Gamma, \\
            (2\mu \epsilon(y) + \lambda \divergence y) n &= 0 \text{ on } \partial\Omega\setminus\Gamma,
    \end{aligned}\right.
\end{equation*}
with distributed control $u$;
see \cref{sec:elasticityStateEq} for details.

The third example illustrates applications to more general variational problems and concerns multimaterial branched transport as introduced in \cite{MaMaTi17}.
Here, different materials have to be transported through a street or pipe network, where for each material $i$ the amount $m_i$ has to be routed from its source $x_i$ to its sink $y_i$.
The flux along a street or pipe is described by the vector-valued function $u$ defined on the network,
where $u_i$ describes the flux of material $i$ with its sign indicating the flow direction.
To avoid an uneconomic splitting of each material, $\calM$ should only contain vectors corresponding to combinations of (positive or negative, depending on direction) fluxes $m_i$ of different materials.
The cost $c$ may in addition favor certain combinations over others (e.g., if joint transport of two materials is particularly economic).
The operator $S$ here describes the divergence of the flux,
and the deviation of $Su$ from $z=\sum_im_ie_i(\delta_{x_i}-\delta_{y_i})$
(with $e_i$ the $i$th standard unit vector and $\delta_x$ the Dirac measure at $x$) is penalized to avoid material loss.

\bigskip

Regarding the admissible set $\calM$, we consider for the case of the Bloch equation -- again following \cite{Spincontrol:15} -- radially distributed control values together with the origin, i.e.,
\begin{equation*}
    \calM = \left\{ \begin{pmatrix} 0 \\ 0 \end{pmatrix}, \begin{pmatrix} \omega_0 \cos\theta_1 \\ \omega_0 \sin \theta_1 \end{pmatrix}, \dots, \begin{pmatrix} \omega_0 \cos\theta_M \\ \omega_0 \sin \theta_M \end{pmatrix}\right\}
\end{equation*}
for a fixed amplitude $\omega_0>0$ and $M>2$ equidistributed phases
\begin{equation*}
    0\leq\theta_1<\cdots<\theta_M<2\pi.
\end{equation*}
In this example, all admissible control values apart from $0$ have the same magnitude; it also provides a link to classical sparsity promotion and allows a closed-form treatment of an arbitrary number of such states.

For the case of linearized elasticity, we consider in addition an admissible set containing control values of different magnitudes but not the origin. As an example, we make the concrete choice
\begin{equation*}
    \calM = \left\{ \begin{pmatrix} 1 \\ 1 \end{pmatrix}, \begin{pmatrix} 1 \\ -1 \end{pmatrix}, \begin{pmatrix} -1 \\ 1 \end{pmatrix}, \begin{pmatrix} -1 \\ -1 \end{pmatrix}, \begin{pmatrix} 2 \\ 2 \end{pmatrix}, \begin{pmatrix} 2 \\ -2 \end{pmatrix}, \begin{pmatrix} -2 \\ 2 \end{pmatrix}, \begin{pmatrix} -2 \\ -2 \end{pmatrix}\right\} .
\end{equation*}

For multimaterial branched transport, the admissible control values are
\begin{equation*}
    \calM = \set{ u\in\R^m}{u_i\in\{0,m_i\}\text{ for }i=1,\ldots,m\text{ or }u_i\in\{0,-m_i\}\text{ for }i=1,\ldots,m},
\end{equation*}
with $m_1,\ldots,m_m>0$ fixed material amounts.
Note that $\calM$ only contains vectors with either all nonnegative or all nonpositive components;
components with opposite sign would represent different materials flowing in opposite directions,
for which there is no economic preference.

Beyond illustrating the general procedure, these examples are meant to be useful prototypes that should be directly applicable.

\paragraph{Related work}
Convex relaxation of problems lacking weak lower semicontinuity has a long history; here we only mention the monograph \cite{Ekeland:1999a}. In the context of optimal control of partial differential equations, convex relaxation of discrete control constraints was discussed in \cite{CK:2013,CK:2015,CKK:2017,CKT:2019}; the latter two treating controls in the principal coefficient -- leading to challenges related to homogenization theory -- in combination with total variation regularization required to overcome these challenges. In the context of inverse problems, the use of the scalar multibang penalty as a regularization term was investigated in \cite{CD:2017,TramDo}; multibang regularization was applied to different imaging problems in \cite{HR:2020,WL:2021}. Error estimates for the finite element approximation of such problems can be found in \cite{CDP:2018}. A similar convex relaxation approach was applied to optimal controls with switching structure in \cite{CIK:2014}.
Special cases were treated much earlier for scalar controls. In particular, if $\calM$ contains only two points, problem \eqref{eq:optControlPb} coincides with a (regularized) bang-bang control problem; see, e.g., \cite{bergounioux_optimality_1996,tiba_optimal_1990,troltzsch_minimum_1979}. For $\calM=\{0\}$, the relaxation reduces to the well-known $L^1$ norm used to promote sparse controls; see, e.g., \cite{stadler_elliptic_2007,casas_optimality_2012,ito_optimal_2014}.

There is a vast literature concerning pulse design in magnetic resonance imaging and spectroscopy via optimal control of the Bloch equation, e.g.,
\cite{conolly_optimal_1986,smith_solvent-suppression_1991,rosenfeld_design_1996,khaneja_optimal_2005,xu_designing_2008,gershenzon_design_2009,grissom_fast_2009,skinner_new_2012}. A mathematical treatment of this problem can be found in, e.g., \cite{bonnard_geometric_2014}. Numerical methods for the computation of optimal pulses are based on conjugate gradient methods (see, e.g., \cite{mao_selective_1986}), Krotov methods \cite{vinding_fast_2012}, quasi-Newton and Newton methods with approximate second derivatives \cite{anand_designing_2012}, and Newton methods using exact second derivatives computed via the adjoint approach \cite{Aigner:2015} (which was also the basis of the winning approaches in the 2015 ISMRM RF Pulse Design Challenge \cite{ISMRMChallenge}). The latter is the basis for the numerical treatment in this work.
To the best of our knowledge, there is so far only a very limited number of works dealing with the design of discrete-valued pulses, which is of interest since the hardware often allows only a finite set $\calM$ of pulses \cite{cruickshank_kilowatt_2009,prigl_theoretical_1996}.
In \cite{Spincontrol:15}, this problem is treated via an extension of the approach from \cite{khaneja_optimal_2005} together with a quantization of a continuous control field obtained via standard optimization methods.

The interest in branched transport as a nonconvex version of optimal transport arose during the past two decades, and
the textbook \cite{BeCaMo09} can serve as a comprehensive starting point into the theory.
The multimaterial version that we consider here was introduced in \cite{MaMaTi17} as a convexification of the original branched transport problem.
So far this approach has numerically only been exploited by computing dual certificates for solutions to particular types of branched transport problems \cite{MaOuVe17}.

\paragraph{Organization}
\Cref{sec:existence} provides the abstract convex analysis framework, including existence of solutions of the optimal control problem, necessary optimality conditions, as well as an appropriate regularization for numerical purposes.
\Cref{sec:stability} then derives stability results based on rather general assumptions on the state operator and the multibang penalty.
\Cref{sec:penalty} gives an explicit characterization of the convex analysis framework for the specific examples of the multibang penalty used in this work, while \cref{sec:stateEq} gives more detail about the model state equations and, in particular, verifies for them the previously exploited assumptions.
\Cref{sec:semiSmoothNewton} discusses the numerical solution using a semismooth Newton method. Finally, \cref{sec:examples} presents and discusses illustrative numerical examples for the three model problems.

\section{Convex analysis framework}\label{sec:existence}

To obtain existence of minimizers and numerically feasible optimality conditions, we follow the general framework of \cite{CK:2015} (stated there for the scalar case), which we briefly summarize in this section and adapt to the vector-valued case. We refer to, e.g., \cite{Bauschke:2011,Schirotzek:2007,Clarke:2013} as well as \cite{CV:2020} for a general introduction to nonsmooth analysis and optimization.
Recall that $U=L^2(\Omega;\R^m)$ for some bounded open domain $\Omega\subset\R^n$ and $m\geq 2$, $Y$ is a Hilbert space, and
\begin{align*}
    \calF&\colon U \to \R \cup \left\{ \infty \right\}, \quad u \mapsto \tfrac12\norm{S(u) - z}_Y^2,\\
    \calG&\colon U \to \R \cup \left\{ \infty \right\}, \quad u \mapsto \textstyle\int_{\Omega} g(u(x))\, \dd x
\end{align*}
for $g:\R^m\to\R\cup\{\infty\}$ proper, convex, lower semicontinuous with $\dom g = \co\calM$ (the convex hull of $\calM$) for some finite set $\calM\subset\R^m$. Relating properties of the integrand $g$ to the corresponding integral functional $\calG$ will be crucial in what follows.
For the operator $S$ we will require the following assumptions:
\begin{enumerate}[label=(\textsc{a}\arabic*),ref={\normalfont(\textsc{a}\arabic*)}]
    \item\label{enm:weakWeakContinuity}
        Weak-to-weak continuity, i.e.,\quad
        $u_i\rightharpoonup u\text{ in }U
        \quad\Rightarrow\quad
        S(u_i)\rightharpoonup S(u)\text{ in }Y$.
    \item\label{enm:FrechetDifferentiability}
        Fr\'echet differentiability.
\end{enumerate}
Throughout, we identify the dual space $U^*$ with $U$ via the Riesz isomorphism.

We now consider the problem
\begin{equation}\label{eq:formal_prob}
    \min_{u\in U} \calE(u)\quad\text{ for }\calE(u):=\calF(u)+\calG(u).
\end{equation}
The following statements are analogous to \cite[Props.~2.1 and 2.2]{CK:2015} for the vector-valued case.

\begin{proposition}[existence of minimizers] \label{thm:existence}
    Let $S$ satisfy \ref{enm:weakWeakContinuity}. Then there exists a solution $\bar u\in U$ to \eqref{eq:formal_prob}.
\end{proposition}
\begin{proof}
    Consider a minimizing sequence $\{u_i\}_{i\in\N}$.
    Since $g$ is infinite outside of $\co\calM$, we know that $\|u_i\|_{L^\infty(\Omega)}$ is uniformly bounded so that we may extract a subsequence, again denoted by $\{u_i\}_{i\in\N}$, weakly converging in $U$ to some $\bar u\in U$.
    Now $\int_\Omega g(u(x))\,\dd x$ is sequentially weakly lower semicontinuous by the convexity of $g$,
    while property \ref{enm:weakWeakContinuity} implies weak convergence $S(u_i)\rightharpoonup S(\bar u)$ so that
    \begin{equation*}
        \frac12\norm{S(\bar u) - z}_Y^2 + \int_{\Omega} g(\bar u(x))\,\dd x
        \leq\liminf_{i\to\infty}\frac12\norm{S(u_i) - z}_Y^2 + \int_{\Omega} g(u_i(x))\,\dd x.
    \end{equation*}
    Hence $\bar u$ must be a minimizer.
\end{proof}

\begin{proposition}[optimality conditions] \label{thm:optCond}
    Let $S$ satisfy \ref{enm:FrechetDifferentiability} and let $\bar u\in U$ be a local minimizer of \eqref{eq:formal_prob}. Then there exists a $\bar p \in U$ satisfying
    \begin{equation}
        \label{eq:optsys}
        \left\{\begin{aligned}
                -\bar p &= \calF'(\bar u)=S'(\bar u)^*(S(\bar u)-z),\\
                \bar u &\in \partial\calG^*(\bar p),
        \end{aligned}\right.
    \end{equation}
    where $S'(u)^*:Y\to U$ denotes the (Hilbert-space) adjoint of the Fr\'echet derivative of $S:U\to Y$,
    \begin{equation*}
        \calG^*: U \to\R\cup\{\infty\},\qquad p \mapsto \sup_{u\in U}\, \langle p,u \rangle - \calG(u),
    \end{equation*}
    denotes the Fenchel conjugate of $\calG$, and $\partial\calG^*$ denotes its convex subdifferential.
\end{proposition}
\begin{proof}
    Abbreviate $u_t=\bar u+t(u-\bar u)$ for arbitrary $t>0$ and $u\in U$.
    Due to the optimality of $\bar u$ we have
    \begin{equation*}
        0\leq[\calF(u_t)+\calG(u_t)]-[\calF(\bar u)+\calG(\bar u)].
    \end{equation*}
    Dividing by $t$ and rearranging, we arrive at
    \begin{equation*}
        0\leq\frac{\calF(u_t)-\calF(\bar u)}t+\frac{\calG(u_t)-\calG(\bar u)}t
        \leq\frac{\calF(u_t)-\calF(\bar u)}t+\frac{(1-t)\calG(\bar u)+t\calG(u)-\calG(\bar u)}t,
    \end{equation*}
    where in the second inequality we used the convexity of $\calG$.
    Taking the limit $t\to0$ and setting $\bar p=-\calF'(\bar u)$, we arrive at
    \begin{equation*}
        0\leq\langle-\bar p,u-\bar u\rangle+\calG(u)-\calG(\bar u).
    \end{equation*}
    As this holds for all $u\in U$, we have by definition of the convex subdifferential that $\bar p\in\partial\calG(\bar u)$.
    By the Fenchel--Young Lemma (e.g., \cite[Lem.~5.8]{CV:2020}), this is equivalent to $\bar u\in\partial\calG^*(\bar p)$.
\end{proof}

By, e.g., \cite[Thms.~4.11 and 5.5]{CV:2020}, we have the pointwise a.e.~expression
\begin{equation}\label{eq:subdiff_pointwise}
    \partial\calG^*(p)=\{u\in U:u(x)\in\partial g^*(p(x))\text{ for a.e.~}x\in\Omega\},
\end{equation}
where $g^*$ is the Fenchel conjugate of $g$.
It is readily seen that for $g$ chosen as in \eqref{eq:reg}, $g^*$ is piecewise affine and thus $\partial g^*$ is single valued in each affine region, the values being precisely the elements of $\calM$ (see \cref{sec:penalty}).
More precisely, for each $u\in\calM$ there is an open convex polyhedron $Q(u)\subset\R^m$ such that $\R^m=\bigcup_{u\in\calM}\overline{Q(u)}$ and $\partial g^*(q)=\{u\}$ for all $q\in Q(u)$.
This property suggests that solutions to \eqref{eq:optsys} generically satisfy $u\in\calM$ almost everywhere, which will be exploited in \cref{sec:stability} to derive corresponding stability properties of optimal controls.

\bigskip

Our goal is now to solve \eqref{eq:optsys} using a semismooth Newton method in function spaces; see \cite{Kunisch:2008a,Ulbrich:2011} as well as \cite[Chap.~14]{CV:2020}.
This requires constructing a so-called Newton derivative that, used in place of the non-existing Fr\'echet derivative of $\partial\calG^*$ in a Newton step for \eqref{eq:optsys}, will lead to local superlinear convergence; see, e.g., \cite[Thm.~14.1]{CV:2020}.
This is challenging in general; however, we know by, e.g., \cite[Thm.~14.10]{CV:2020} that if (and only if) $r>s$, then a superposition operator $H:L^r(\Omega,\R^m)\to L^s(\Omega;\R^m)$ given by $p(x)\mapsto h(p(x))$ for a.e.~$x\in \Omega$ and locally Lipschitz continuous $h:\R^m\to\R^m$ is Newton differentiable with Newton derivative $D_N H(p)$ given pointwise a.e.~by an arbitrary element of the Clarke subdifferential
\begin{equation}\label{eq:newtonderiv}
    \partial_C h(q) = \co \left\{\lim_{n\to\infty} \nabla h(q_n)\right\},
\end{equation}
where $\{q_n\}_{n\in\N}\subset \R^m$ is a sequence of points where $h$ is differentiable with $q_n\to q$; such a sequence always exists in finite dimensions by Rademacher's Theorem; see, e.g., \cite[Thm.~13.26]{CV:2020}.

Since the second relation in \eqref{eq:optsys} involves a set-valued mapping, we first need to apply a regularization.
Here we replace the subdifferential $\partial\calG^*(p)$ by its Yosida approximation
\begin{equation}
    (\partial\calG^*)_\gamma(p) \defgl \frac1\gamma\left(p - \prox_{\gamma\calG^*}(p)\right): U \to U
\end{equation}
for some $\gamma>0$ and the proximal mapping
\begin{equation}
    \prox_{\gamma\calG^*}(p) \defgl \left(\Id + \gamma\partial\calG\right)^{-1}(p) = \argmin_{\tilde p\in U}\frac1{2\gamma}\norm{\tilde p-p}_{U}^2 + \calG^*(\tilde p),
\end{equation}
which is single-valued and Lipschitz continuous since $\calG^*$ is convex and lower semicontinuous; see, e.g., \cite[Thm.~6.11 and Cor.~6.14]{CV:2020}.
We thus consider instead of \eqref{eq:optsys} for $\gamma>0$ the regularized optimality conditions
\begin{equation}\label{eq:regsystem}
    \left\{\begin{aligned}
            -p_\gamma &= \calF'(u_\gamma),\\
            u_\gamma &= (\partial\calG^*)_\gamma(p_\gamma).
    \end{aligned}\right.
\end{equation}
By \eqref{eq:subdiff_pointwise}, we can characterize $\prox_{\gamma\calG^*}$ and therefore $H_\gamma:=(\partial\calG^*)_\gamma$ pointwise a.e.~as well; we will derive explicit pointwise expressions for the Yosida approximation and its Newton derivative for different choices of the finite set $\calM$ in \cref{sec:penalty}. Furthermore, we will argue in \cref{sec:semiSmoothNewton} that $H_\gamma$ and hence \eqref{eq:optsys} is in fact Newton differentiable, thereby guaranteeing local superlinear convergence of the corresponding semismooth Newton method.

To see the relation of \eqref{eq:optsys} to \eqref{eq:formal_prob}, we first note that the Yosida approximation $(\partial\calG^*)_\gamma$ is linked to the Moreau envelope
\begin{equation}
    (\calG^*)_\gamma(p)=\min_{\tilde p\in U}\frac1{2\gamma}\norm{\tilde p-p}_{U}^2+\calG^*(\tilde p)
\end{equation}
via $(\partial \calG^*)_\gamma = \partial(\calG^*)_\gamma$;
see, e.g.,
\cite[Thm.~7.9]{CV:2020}, which justifies the term \emph{Moreau--Yosida regularization} (of $\calG^*$). Furthermore, from
\cite[Thm.~7.11]{CV:2020}, we have that
\begin{equation}
    ((\calG^*)_\gamma)^*(u) = \calG(u) + \frac\gamma2\norm{u}_U^2
\end{equation}
and, hence, \eqref{eq:regsystem} coincides with the necessary and sufficient optimality conditions for the strictly convex minimization problem
\begin{equation}
    \label{eq:problem_reg}
    \min_{u\in U}\calE_\gamma(u)\quad\text{ for }\calE_\gamma(u)= \calF(u) + \calG(u) + \frac\gamma2 \norm{u}_{U}^2.
\end{equation}
By the same arguments as in the proof of \cref{thm:existence}, we obtain the existence of a minimizer $u_\gamma\in U$ and thus of a corresponding $p_\gamma = -\calF(u_\gamma)\in U$.

\begin{remark}
    An alternative regularization leading to Newton differentiability is to instead apply the Yosida approximation to the equivalent subdifferential inclusion $\bar p\in\partial\calG(\bar u)$ in \eqref{eq:optsys}. This would correspond to replacing $\calG$ in \eqref{eq:formal_prob} with its (Fr\'echet-differentiable) Moreau envelope $\calG_\gamma:u\mapsto \min_{\tilde u\in U} \frac1{2\gamma}\norm{\tilde u-u}_U^2+\calG(\tilde u)$, thus smoothing out the nondifferentiability that is responsible for the structural properties of the penalty.
    In contrast, our regularization does not remove the nondifferentiability but merely makes the functional (more) strongly convex so that the structural features of the multibang regularization are preserved.
\end{remark}

We now address convergence of solutions to \eqref{eq:problem_reg} as $\gamma\to 0$. The following statement is a slight generalization of \cite[Prop.~4.1]{CK:2015}.
We here prove it by $\Gamma$-convergence (see \cite{Braides:2002} for a gentle introduction),
which is a classical technique to check whether the solution of a perturbed optimization problem converges, as the perturbation tends to zero, to the solution of the unperturbed problem.
The term ``perturbation'' may here be interpreted in a broad sense;
in the subsequent statement it is used in the sense of a so-called singular perturbation
(where the optimization problem depends on a small model or regularization parameter that approaches zero),
but in the next section it will represent perturbations of the target state or measurement $z$
(and it could for instance just as well represent discretizations of a continuous optimization problem).
If a sequence $\calE_n$ of energies $\Gamma$-converges to some energy $\calE$
(which means that for any sequence $u_n\to u$ we have $\calE(u)\leq\liminf_{n\to\infty}\calE_n(u_n)$
and that for every $u$ there exists a so-called recovery sequence $u_n\to u$ with $\limsup_{n\to\infty}\calE_n(u_n)\leq\calE(u)$)
and if the $\calE_n$ are uniformly coercive (or just boundedness of $\calE_n(u_n)$ implies precompactness of the sequence $u_n$),
then minimizers of $\calE_n$ are known to converge (up to subsequences) to minimizers of $\calE$.

\begin{proposition}[limit for vanishing regularization]
    Let $S$ satisfy \ref{enm:weakWeakContinuity}. Then it holds that $\Gamma\text{-}\lim_{\gamma\to0}\calE_\gamma=\calE$ with respect to weak convergence in $U$.
    As a consequence, any sequence $u_{\gamma_n}$ of global minimizers to \eqref{eq:problem_reg} for $\gamma_n\to0$ contains a subsequence converging weakly in $U$ to a global minimizer of \eqref{eq:formal_prob}.
    Moreover, this convergence is strong.
\end{proposition}
\begin{proof}
    For the $\Gamma$-limit, we first have to show that for any sequence $\gamma_n\to0$ and any weakly converging sequence $u_n\rightharpoonup u$ we have $\liminf_{\gamma_n\to0}\calE_{\gamma_n}(u_{\gamma_n})\geq\calE(u)$,
    which is an immediate consequence of the sequential weak lower semicontinuity of $\calE$ (shown in the proof of \cref{thm:existence}) and of $\|\cdot\|_{U}$.
    Second, the required recovery sequence is just the constant sequence $u_n=u$.
    Furthermore, minimizers of $\calE_\gamma$ are uniformly bounded in $U$, since $g$ is infinite outside the convex hull $\co\calM$,
    which together with the $\Gamma$-convergence is well-known to imply the weak convergence in $U$ of minimizers of $\calE_\gamma$ to minimizers of $\calE$.
    Finally, for such a weakly converging sequence $u_{\gamma_n}\rightharpoonup u$ of minimizers of $\calE_{\gamma_n}$ we have
    \begin{equation}
        \calE(u_{\gamma_n})+\frac{\gamma_n}2\norm{u_{\gamma_n}}_{U}^2\leq\calE_{\gamma_n}(u)\leq\calE(u_{\gamma_n})+\frac{\gamma_n}2\norm{u}_{U}^2,
    \end{equation}
    which implies $\norm{u}_{U}\geq\norm{u_{\gamma_n}}_{U}$ so that the convergence $u_{\gamma_n}\to u$ is actually strong.
\end{proof}
For error estimates of the Moreau--Yosida approximation in terms of $\gamma$ (as well as of a finite element discretization in the scalar case) under a regularity assumption, we refer to \cite{CDP:2018}.

\section{Stability properties of multibang controls}\label{sec:stability}

We now discuss stability properties of the controls by exploiting the special structure of the optimality conditions for the multibang control problem.
In particular we consider in what sense the controls converge as the target state converges; what can be said about controls with values in $\calM$; and when exact controls (which achieve the target state) can be retrieved by the optimization.
To keep the notation concise, we set
\begin{equation}
    \calE^z(u):=\frac12\norm{S(u) - z}_Y^2 + \int_{\Omega} g(u(x))\,\dd x,
\end{equation}
where $g:\R\to\extR$ is again proper, convex, and weakly lower semicontinuous with $\dom g = \co \calM$.

\subsection{Stability with respect to target perturbations}

First, we examine how perturbations of the target $z$ influence the minimizer of \eqref{eq:formal_prob}.
This is for instance of interest if our control problem actually represents an inverse problem,
in which the measurement $z$ is typically slightly perturbed by noise.
We will see that as $z_n$ converges strongly to $z$ in $Y$, the corresponding minimizers converge in $U$ in the weak sense.
Strong convergence cannot be expected in general due to worst-case scenarios in which the limit minimizer $\bar u$ has a nonempty ``singular arc''
\begin{equation}
    \calS_{\bar u}=\set{x\in\Omega}{\bar u(x)\notin\calM},
\end{equation}
i.e., the region in which $\bar u$ does not attain any of the distinguished values $\calM$.
However, away from that singular arc one obtains strong convergence and, as a consequence, controls in $\calM$ even for perturbed targets.
In this section we use the following additional assumptions on $S$
(which will be shown to hold for our model forward operators in \cref{sec:stateEq}):
\begin{enumerate}[label=(\textsc{a}\arabic*),ref={\normalfont(\textsc{a}\arabic*)}]\setcounter{enumi}{2}
    \item\label{enm:compactness}
        $S:U\to Y$ is compact.
    \item\label{enm:adjointConvergence}
        For some Banach space $V\hookleftarrow U$ with $V^*\hookrightarrow L^\infty(\Omega;\R^m)$, we have
        \begin{equation*}
            \lim_{\tilde u\rightharpoonup u\text{ in }U}\norm{[S'(\tilde u)-S'(u)]^*y}_{V^*}=0\quad\text{ for all }y\in Y.
        \end{equation*}
\end{enumerate}

\begin{proposition}[$\Gamma$-convergence of objective functional]\label{thm:GammaConvergence}
    Let $z_n\to z$ in $Y$ and $S$ satisfy \ref{enm:weakWeakContinuity}. Then with respect to weak convergence in $U$, we have
    \begin{equation*}
        \Gamma\text{-}\lim_{n\to\infty}\calE^{z_n}=\calE^z.
    \end{equation*}
\end{proposition}
\begin{proof}
    For the $\liminf$ inequality, let $u_n\rightharpoonup u$ weakly in $U$, then by property \ref{enm:weakWeakContinuity} and the weak lower semicontinuity of $\norm{\cdot}_Y$ and the convexity of $g$, we have
    \begin{equation*}
        \begin{aligned}
            \liminf_{n\to\infty}\calE^{z_n}(u_n)
            &=\liminf_{n\to\infty}\frac12\norm{S(u_n) - z_n}_Y^2 + \int_{\Omega} g(u_n(x))\,\dd x\\
            &\geq\frac12\norm{S(u) - z}_Y^2 + \int_{\Omega} g(u(x))\,\dd x
            =\calE^z(u).
        \end{aligned}
    \end{equation*}
    For the $\limsup$ inequality, choose $u_n=u\in U$ to obtain
    \begin{displaymath}
        \limsup_{n\to\infty}\calE^{z_n}(u_n)
        =\limsup_{n\to\infty}\frac12\norm{S(u) - z_n}_Y^2 + \int_{\Omega} g(u(x))\,\dd x
        =\calE^z(u).
    \end{displaymath}
\end{proof}

This proposition now implies a weak stability of the control.

\begin{corollary}[stability of control and state]
    Under the conditions of \cref{thm:GammaConvergence} and \ref{enm:compactness}, any sequence $\{u_n\}_{n\in\N}$ of minimizers of $\calE^{z_n}$ contains a subsequence converging weakly in $U$ to a minimizer $\bar u$ of $\calE^z$.
    The corresponding states $y_n=S(u_n)$ converge strongly in $Y$ to $\bar y=S(\bar u)$.
\end{corollary}
\begin{proof}
    Since $g$ is infinite outside $\co\calM$ we know that $\norm{u}_{L^\infty(\Omega;\R^n)}$ is uniformly bounded among all $u\in U$ with finite energy $\calE^{z_n}(u)$,
    where the bound is independent of $n$.
    Thus, the $\calE^{z_n}$ are equimildly coercive so that the convergence of minimizers $u_n$ follows from the $\Gamma$-convergence of the functionals.
    The convergence of states $y_n=S(u_n)\to \bar y=S(\bar u)$ along the subsequence follows from $u_n\rightharpoonup \bar u$ together with properties \ref{enm:weakWeakContinuity} and
    \ref{enm:compactness} (weak-to-weak continuity and compactness of $S$, respectively).
\end{proof}

Under additional assumptions, we also obtain convergence of the dual variable.

\begin{corollary}[stability of dual]\label{thm:stability}
    Under the conditions of \cref{thm:GammaConvergence} and \ref{enm:weakWeakContinuity}--\ref{enm:adjointConvergence},
    consider the sequence of minimization problems $\min_{u\in U}\calE^{z_n}(u)$.
    The corresponding optimal controls $u_n$, states $y_n$, and dual variables $p_n$ satisfy up to a subsequence
    \begin{equation*}
        u_n\rightharpoonup \bar u\;\text{in }U,
        \quad
        y_n\to \bar y\;\text{in }Y,
        \quad\text{and}\quad
        p_n\to \bar p\;\text{in }V^*,
    \end{equation*}
    where $\bar u$ is a minimizer of $\calE^z$, $\bar y = S(\bar u)$, and $\bar p$ satisfies \eqref{eq:optsys}.
\end{corollary}
\begin{proof}
    We already know $u_n\rightharpoonup \bar u$ and $y_n\to \bar y$.
    By the Banach--Steinhaus theorem and \ref{enm:adjointConvergence}, $[S'(u_n)-S'(\bar u)]^*$ is uniformly bounded in $L(Y;V^*)$ and thus also $S'(u_n)^*$. Now
    \begin{equation}
        \begin{aligned}[b]
            \norm{p_n-\bar p}_{V^*}
            &=\norm{S'(u_n)^*(z_n-y_n)-S'(\bar u)^*(z-\bar y)}_{V^*}\\
            &\leq\norm{S'(u_n)^*(z_n-y_n)-S'(u_n)^*(z-\bar y)}_{V^*} +
            \norm{S'(u_n)^*(z-\bar y)-S'(\bar u)^*(z-\bar y)}_{V^*}\\
            &\leq\norm{S'(u_n)^*}_{L(Y;V^*)}\norm{z_n-y_n-(z-\bar y)}_Y+\norm{[S'(u_n)^*-S'(\bar u)^*](z-\bar y)}_{V^*}
            \to0.
        \end{aligned}
        \qedhere
    \end{equation}
\end{proof}

The final result shows strong convergence of controls outside the singular arc, which will be seen to correspond to the case where $\partial g^*(\bar p(x))$ is set valued (cf.~\eqref{eq:conj_subdiff_radial} and \eqref{eq:conj_subdiff_concentric}).

\begin{proposition}[locally strong convergence of control]\label{thm:localConvergence}
    Let the conditions of \cref{thm:GammaConvergence} and {\ref{enm:weakWeakContinuity}--\ref{enm:adjointConvergence}} hold.
    Furthermore, let $Q$ be the set on which $\partial g^*$ is single valued, and abbreviate $\Omega_P=\{x\in\Omega:p(x)\in P\}$ for given $P\subset\R^m$. Then we have
    \begin{enumerate}[{label={(\roman*)}, ref={(\roman*)}}]
        \item for any $P\subset\subset Q$ compact and $n$ large enough, $u_n(x) = \bar u(x) \in \calM$ for a.e.~$x\in \Omega_P$;
        \item $u_n|_{\Omega_Q}\to \bar u|_{\Omega_Q}$ strongly in $L^2(\Omega_Q;\R^m)$ and $\bar u(x)\in \calM$ for a.e.~$x\in \Omega_Q$.
    \end{enumerate}
\end{proposition}
\begin{proof}
    By \cref{thm:stability}, we have $p_n\to \bar p$ in $L^\infty(\Omega;\R^m)$.
    In particular, for $n$ large enough, for all $x\in\Omega_P$ the value $p_n(x)$ lies in the same connected component of $Q$ as $\bar p(x)$.
    Hence, $u_n(x)=\bar u(x)$ due to $u_n(x)\in\partial g^*(p_n(x))=\partial g^*(\bar p(x))$ and $\bar u(x)\in\partial g^*(\bar p(x))$.
    Since this holds for any compact subset $P$ of $Q$, we actually have pointwise convergence $u_n(x)\to \bar u(x)$ for almost all $x\in\Omega_Q$.
    The uniform boundedness of $u_n$ (since otherwise $g(u_n(x))=\infty$) then implies strong convergence by the dominated convergence theorem.
\end{proof}

\subsection{Controls in \boldmath$\calM$}

Here, we examine more closely controls taking values only in $\calM$. In the following, we refer to minimizers $\bar u\in U$ of $\calE^z$ with $\bar u(x) \in \calM$ for almost everywhere $x\in \Omega$ as \emph{multibang controls}.
First, we note that such controls allow us to achieve an energy arbitrarily close to the optimum.

\begin{remark}[near-optimality]
    Under assumptions \ref{enm:weakWeakContinuity} and \ref{enm:compactness}, we have
    \begin{equation*}
        \min_{u\in U} \calE^z(u)= \inf_{\substack{u\in U\\ u(x)\in\calM \text{ a.e.}}} \calE^z(u).
    \end{equation*}
    Indeed, let $\bar u\in U$ minimize $\calE^z$.
    By the definition of $g$, there exists a sequence $\{u_n\}_{n\in\N}\subset U$ with $u_n(x)\in \calM$ a.e., $u_n\rightharpoonup \bar u$ in $U$, and $\int_\Omega g(u_n(x))\,\dd x\to\int_\Omega g(\bar u(x))\,\dd x$.
    Furthermore, $S(u_n)\to S(\bar u)$ in $Y$ so that $\calE^z(u_n)\to\calE^z(\bar u)$.
\end{remark}

In the remainder of this subsection, we shall restrict ourselves to the case that
\begin{enumerate}[label=(a\arabic*),ref={\normalfont(a\arabic*)}]\setcounter{enumi}{4}
    \item\label{enm:linearity}
        $S:U\to Y$ is linear,
\end{enumerate}
which will only apply to the elasticity example, but not to the Bloch setting.
The intuition is that the case with multibang controls is generic (or even that targets with nonmultibang controls, i.e., $u(x)\notin\calM$ on a nonnegligible set, are nowhere dense in $Y$).
This is consistent with \cref{thm:localConvergence}, since targets with a singular arc of zero measure (or rather with $\Omega_Q=\Omega$) can be perturbed without producing a singular arc.
Below we will at least see that targets leading to multibang controls are dense in $Y$;
and that the mapping $z\mapsto\argmin_{u\in U}\calE^z(u)$ is not continuous in any target $z$ for which the singular arc has positive measure.

\begin{proposition}[approximation via multibang control]\label{thm:discrControlAppr}
    Let $S$ satisfy {\ref{enm:weakWeakContinuity}--\ref{enm:linearity}.}
    Then for any $z\in Y$ and corresponding minimizer $\bar u\in U$ of $\calE^z$, there exists a sequence $\{z_n\}_{n\in\N} \subset Y$ with $z_n\to z$ such that the corresponding minimizers $u_n\in U$ of $\calE^{z_n}$ satisfy $u_n(x)\in\calM$ almost everywhere, $u_n\rightharpoonup \bar u$, and $\calE^{z_n}(u_n)=\calE^z(\bar u)$.
\end{proposition}
\begin{proof}[Sketch of proof]
    By \eqref{eq:optsys}, we have $\bar p=S^*(z-S\bar u)$ and $\bar u(x)\in\partial g^*(\bar p(x))$ for almost all $x\in\Omega$.
    The piecewise affine structure of $g^*:\R^m\to\R$ implies that $\bar u(x)$ is a convex combination of (at most) $m+1$ values $\hat u_j\in\calM\cap\partial g^*(\bar p(x))$.
    Thus one can find $u_n\rightharpoonup \bar u$ with $u_n(x)\in\calM\cap\partial g^*(\bar p(x))$ for almost all $x\in\Omega$.
    Choosing $z_n=Su_n+(z-S\bar u)$, we have $z_n\to z$ as well as $\bar p=S^*(z_n-Su_n)$ and $u_n(x)\in\partial g^*(\bar p(x))$ for almost all $x\in \Omega$. Hence by the convexity of the energy $\calE^{z_n}$, $u_n$ is a minimizer of $\calE^{z_n}$.
    Furthermore, one can even choose $u_n$ such that $\int_\Omega g(u_n(x))\,\dd x=\int_\Omega g(\bar u(x))\,\dd x$ so that $\calE^{z_n}(u_n) = \calE^z(\bar u)$ as claimed.
\end{proof}

\begin{corollary}[strong convergence of control]
    Let the conditions of \cref{thm:discrControlAppr} hold. Then
    \begin{enumerate}[{label={(\roman*)}, ref={(\roman*)}}]
        \item the targets $z$ admitting a multibang control $\bar u$ minimizing $\calE^z$ are dense in $Y$;
        \item if $S$ is injective and the minimizer $\bar u$ to $\calE^z$ has a singular arc of positive measure, then one cannot have strong convergence of minimizers $u_n$ of $\calE^{z_n}$ for all $z_n\to z$.
    \end{enumerate}
\end{corollary}
\begin{proof}
    The first statement is a direct consequence of \cref{thm:discrControlAppr}. The second statement follows from the strict convexity of $\calE^z$ and thus the uniqueness of its minimizers, together with the fact that strong convergence in $U$ implies pointwise convergence:
    Indeed, let $\bar u$ have a singular arc $\calS_{\bar u}$ of positive measure and choose $z_n\to z$ such that the unique minimizers $u_n$ of $\calE^{z_n}$ are multibang controls (which is possible by the first statement).
    If we had strong convergence $u_n\to\bar u$ in $U$, then (up to a subsequence) also $u_n\to\bar u$ pointwise almost everywhere, in particular, on $\calS_{\bar u}$.
    This contradicts $u_n(x)\in\calM$ almost everywhere.
\end{proof}

\subsection{Retrieval of exact controls}

We now consider more specifically the consequence of the convex relaxation \eqref{eq:reg} for some nonnegative and strictly convex $c:\R^m\to\R$.
A peculiar feature of the multibang control in this case is that for attainable targets, i.e., if there exists a $\hat u\in U$ such that $z=S(\hat u)$, the generating control $\hat u$ can only be recovered as a minimizer $\bar u$ of the optimization problem~\eqref{eq:formal_prob} if
$c(\hat u(x))= \min_{v\in\calM} c(v)$ almost everywhere. This demonstrates the desirability of allowing multiple admissible control values of equal magnitude.

\begin{proposition}[achievement of target]
    If $S$ satisfies \ref{enm:FrechetDifferentiability}, then, for any minimizer $\bar u\in U$ of $\calE^z$ that satisfies $S(\bar u) = z$, it holds that $g(\bar u(x)) = \min_{v\in\calM} g(v)$ almost everywhere.
    In particular, if in addition $\bar u(x)\in\calM$ almost everywhere, then $c(\bar u(x))=\min_{v\in\calM} c(v)$.
\end{proposition}
\begin{proof}
    If $S(\bar u)=z$, the first relation in the optimality condition \eqref{eq:optsys} together with linearity of $S'(\bar u)$ implies $\bar p=0$. Hence, the second relation yields $\bar u\in\partial\calG^*(0)$ and therefore $0\in\partial\calG(\bar u)$. By \eqref{eq:subdiff_pointwise}, this implies $0\in\partial g(\bar u(x))$ for almost all $x\in\Omega$ and therefore
    \begin{equation}
        g(\bar u(x)) = \min_{v\in \R^m} g(v) = \inf_{v\in\R^m} g_\infty(v) = \inf_{v\in \calM} \alpha c(v) = \min_{v\in \calM} g(v)
    \end{equation}
    since $\min f^{**} = \inf f$ by the properties of the convex hull; see, e.g., \cite[Prop.~12.9\,(iii)]{Bauschke:2011}.
\end{proof}

If, however, $c(\hat u(x))=\min_{v\in\calM} c(v)$ is not satisfied almost everywhere, then the generating control $\hat u$ can only be recovered in the limit $\alpha\to 0$.
In fact, in this limit the best approximation is achieved, i.e., an optimal control which yields the minimum possible tracking term $\calF$. In the following, we denote by $u_\alpha$ the minimizer of $\calE^z$ (which depends on $\alpha$ via the definition \eqref{eq:reg} of $g$) for given $\alpha>0$.

\begin{proposition}[$\Gamma$-convergence for vanishing regularization]
    For given $z\in Y$, let $M:=\inf_{u\in U}\norm{S(u)-z}_Y$ and $\calO:=\{u\in U:\norm{S( u)-z}_Y=M\}$.
    If $S$ satisfies \ref{enm:weakWeakContinuity}, then with respect to weak convergence in $U$ we have
    \begin{equation}
        \Gamma\text{-}\lim_{\alpha\to0}\frac1\alpha\left(\calE^z-\frac{M^2}2\right)
        =\delta_{\calO}+\calG_1,
    \end{equation}
    where
    \begin{equation}
        \calG_1(u)=\int_{\Omega} g_1^{**}(u(x))\,\dd x
        \quad\text{ for }\quad
        g_1(u)=c(u) + \delta_{\calM}(u).
    \end{equation}
\end{proposition}
\begin{proof}
    The $\limsup$ inequality is trivial using the constant sequence; for the $\liminf$ inequality we only have to consider a sequence $u_\alpha\rightharpoonup u\notin \calO$.
    In that case,
    \begin{equation}
        \liminf_{\alpha\to0}\norm{S(u_\alpha)-z}_Y\geq\norm{S(u)-z}_Y>M
    \end{equation}
    so that
    \begin{displaymath}
        \frac1\alpha\left(\min_{u \in U} \frac12\norm{S(u) - z}_Y^2 + \int_{\Omega} g(u(x))\,\dd x-\frac{M^2}2\right)\to\infty.
    \end{displaymath}
\end{proof}

\begin{corollary}[approximation of target]
    Under the conditions of the previous proposition, if $\calO\neq\emptyset$, then any family $\{u_\alpha\}_{\alpha>0}$ of minimizers of $\calE^z$ contains a subsequence converging weakly to a minimizer $\bar u\in \calO$ of $\calG_1$.
\end{corollary}
\begin{proof}
    This follows from the equimild coerciveness of the energies and the $\Gamma$-convergence; see \cite[Def.~1.19 and Thm.~1.21]{Braides:2002}.
\end{proof}

\section{Vector-valued multibang penalty}\label{sec:penalty}

To implement the general framework of \cref{sec:existence}, we need explicit characterizations of the Fenchel conjugate and its subdifferential as well as its Moreau--Yosida regularization
for the multibang penalty \eqref{eq:reg}. Recall that $\calG$ is defined as an integral functional for the proper, convex, and lower semicontinuous integrand
\begin{equation}\label{eq:multibangPenalty}
    g = \left(\alpha c(\cdot) +\delta_\calM\right)^{**}=g_\infty^{**}.
\end{equation}
We can thus proceed by pointwise computation.

We first summarize the general procedure. Since $g^* = (g_\infty^{**})^* = (g_\infty^*)^{**} = g_\infty^*$, the Fenchel conjugate of $g$ is given by
\begin{equation}\label{eq:conj}
    g^*(q) = g_\infty^*(q) \defgl \sup_{v\in\R^m} \scalprod{v,q} - g_\infty(v) = \max_{v\in \calM}\, \scalprod{v,q} - \alpha c(v).
\end{equation}
Hence, $g^*$ is the maximum of a finite number of convex and continuous functions of $q$, and we can thus compute its subdifferential (without computing $g^*$ first!) using the maximum rule; see, e.g., \cite[Prop.~4.5.2, Rem.~4.5.3]{Schirotzek:2007}. Setting
\begin{equation*}
    g_v^*(q) \defgl \scalprod{v,q} - \alpha c(v),
\end{equation*}
we have
\begin{equation}\label{eq:conj_subdiff}
    \partial g^*(q) = \co \bigcup_{\substack{v\in\calM:\\ g^*(q) = g_v^*(q)}}\partial g_v^*(q) = \co \set{v\in\calM}{g^*(q) = g_v^*(q)}
\end{equation}
with $\co$ denoting the convex hull.
Obviously, the subdifferential $\partial g^*$ is piecewise constant.
If we denote the elements of $\calM$ by $\bar u_i$ for $i$ from some index set $I$, then $\partial g^*$ takes the value $\co\,\{\bar u_{i_1},\ldots,\bar u_{i_k}\}$ with $i_1,\ldots,i_k\in I$ on
\begin{equation*}
    Q_{i_1\ldots i_k}=\set{q\in\R^m}{g^*(q)=g_{\bar u_i}^*(q)\text{ if and only if }i\in\{i_1,\ldots,i_k\}}.
\end{equation*}
For the proximal mapping
\begin{equation*}
    \prox_{\gamma g^*}(q) := \argmin_{w\in\R^m} \frac1{2\gamma} |w-q|_2^2 + g^*(w) = (\Id + \gamma \partial g^*)^{-1}(q)
\end{equation*}
we then make use of the equivalence
\begin{equation}\label{eq:proxMapRelation}
    w = (\Id + \gamma \partial g^*)^{-1}(q)\qquad \Leftrightarrow\qquad q \in (\Id + \gamma \partial g^*)(w) = \{w\} + \gamma \partial g^*(w)
\end{equation}
and follow the case distinction in the maximum rule \eqref{eq:conj_subdiff}.
In detail, we first define $Q_{i_1\ldots i_k}^\gamma$ to be the image of $Q_{i_1\ldots i_k}$ under $(\Id+\gamma\partial g^*)$,
\begin{equation*}
    Q_{i_1\ldots i_k}^\gamma
    =(\Id+\gamma\partial g^*)(Q_{i_1\ldots i_k})
    =Q_{i_1\ldots i_k}+\gamma\co\,\{\bar u_{i_1},\ldots,\bar u_{i_k}\}.
\end{equation*}
The preimage $w\in Q_{i_1\ldots i_k}$ of $q\in Q_{i_1\ldots i_k}^\gamma$ under $(\Id+\gamma\partial g^*)$ is thus obtained by solving the linear system of equations
\begin{equation*}
    \begin{aligned}
        0&=g_{\bar u_{i_l}}^*(w)-g_{\bar u_{i_1}}^*(w),\quad l=2,\ldots,k,\\
        q&=w+\gamma(\lambda_1\bar u_{i_1}+\ldots+\lambda_k\bar u_{i_k}),\\
        1&=\lambda_1+\ldots+\lambda_k,
    \end{aligned}
\end{equation*}
for $w\in\R^m$ and the convex combination coefficients $\lambda_1,\ldots,\lambda_k\in\R$.
Let us express the solution of this system (obtained by inverting the system matrix; if not invertible use the minimum norm solution) as
\begin{equation}\label{eqn:multibangGeneralSystem}
    w=A_{i_1\ldots i_k}q+b_{i_1\ldots i_k}\,,\quad
    \lambda_l=A_{i_1\ldots i_k;l}q+b_{i_1\ldots i_k;l},
\end{equation}
for some $A_{i_1\ldots i_k}\in\R^{m\times m}$, $b_{i_1\ldots i_k}\in\R^m$ and $A_{i_1\ldots i_k;l}\in\R^{1\times m}$, $b_{i_1\ldots i_k;l}\in\R$, $l=1,\ldots,k$.
The Moreau--Yosida regularization $h_\gamma=(\partial g^*)_\gamma$ of $\partial g^*$ is then given by
\begin{equation}\label{eq:my_reg}
    h_\gamma(q)=(\partial g^*)_{\gamma}(q) = \tfrac{1}{\gamma}\left(q - \prox_{\gamma g^*}(q)\right)=\tfrac1\gamma(q-A_{i_1\ldots i_k}q-b_{i_1\ldots i_k}).
\end{equation}
Since $h_\gamma = (\partial g^*)_\gamma$ is continuous and piecewise continuously differentiable ($PC^1$), its Clarke subdifferential \eqref{eq:newtonderiv} at $q\in \R^m$ is given by the convex hull of the derivatives of the branches active at $q$; see, e.g., \cite[Thm.~14.7]{CV:2020}. We can thus take as a Newton derivative
\begin{equation*}
    D_Nh_\gamma(q)=\tfrac1\gamma(\mathrm{Id}-A_{i_1\ldots i_k}).
\end{equation*}
To compute $h_\gamma(q)$ and $D_Nh_\gamma(q)$ for an arbitrary $q\in\R^m$ it remains to identify the set $Q_{i_1\ldots i_k}^\gamma$ in which $q$ lies.
To this end, note that $q\in Q_{i_1\ldots i_k}^\gamma$ if and only if
\begin{align*}
    \lambda_l&=A_{i_1\ldots i_k;l}q+b_{i_1\ldots i_k;l}\in[0,1]\quad\text{for }l=1,\ldots,k
    \quad\text{and}\\
    w&=A_{i_1\ldots i_k}q+b_{i_1\ldots i_k}\in Q_{i_1\ldots i_k},
\end{align*}
since $w$ represents the preimage of $q$ under $(\mathrm{Id}+\gamma\partial g^*)$
and $\lambda_1,\ldots,\lambda_k$ represent the convex combination coefficients such that $w+\lambda_1\bar u_{i_1}+\ldots+\lambda_k\bar u_{i_k}=q$.
To check the condition $w\in Q_{i_1\ldots i_k}$ it suffices to check that there is no $i\notin\{i_1,\ldots,i_k\}$ with $g_{\bar u_i}^*(w)\geq g_{\bar u_{i_1}}^*(w)$.

\Cref{sec:multibang:bloch,sec:multibang:elast} make this calculation explicit for the first two choices of $\calM$ introduced in the introduction
and the quadratic cost $c(v) = \frac12|v|_2^2$,
which is particularly useful for generating more efficient code. We finally address in \cref{sec:multibang:general} the algorithmic evaluation in the general case as relevant for the multimaterial branched transport example.

\subsection{Radially distributed control values}\label{sec:multibang:bloch}

Here, we take as set $\calM\subset\R^2$ of admissible control values the vector $0$ together with vectors of fixed amplitude $\omega_0>0$ and $M>2$ equidistributed phases
\begin{equation*}
    0\leq\theta_1<\cdots<\theta_M<2\pi
\end{equation*}
(where we shall assume $\theta_{i+1}-\theta_i<\pi$ for $i=1,\ldots,M-1$ and $\theta_1-(\theta_M-2\pi)<\pi$),
that is,
\begin{equation*}\label{eq:setBloch}
    \calM = \left\{ \begin{pmatrix} 0 \\ 0 \end{pmatrix}, \begin{pmatrix} \omega_0 \cos\theta_1 \\ \omega_0 \sin \theta_1 \end{pmatrix}, \dots, \begin{pmatrix} \omega_0 \cos\theta_M \\ \omega_0 \sin \theta_M \end{pmatrix}\right\} \eqqcolon \left\{ \u_0, \u_1, \dots \u_M\right\}.
\end{equation*}

In the following it will be helpful to identify an angle $\theta\in[0,2\pi)$ with the corresponding point $\vec\theta=(\cos\theta,\sin\theta)$ on the unit circle $S^1$.
Let $\phi_i$ denote the midpoint between $\theta_i$ and $\theta_{i+1}$ (identifying $\theta_{M+1}=\theta_1$ for simplicity), that is, $\vec\phi_i=(\vec\theta_i+\vec\theta_{i+1})/|\vec\theta_i+\vec\theta_{i+1}|_2$, and introduce the circular sectors
\begin{equation*}
    C_i=\set{\omega\vec\theta\in\R^2}{\theta\in(\phi_i,\phi_{i+1}),\,\omega\geq0}.
\end{equation*}
Here, $\theta\in(\phi_i,\phi_{i+1})$ is to be understood $2\pi$-periodically, that is, $\phi_{M+1}$ is identified with $\phi_1$, and $(\phi_i,\phi_{i+1})$ with $\phi_{i+1}<\phi_i$ is interpreted as $(\phi_i,\phi_{i+1}+2\pi)$.

\paragraph{Fenchel conjugate}

Using the equivalence of angles and sectors introduced above, it is straightforward to see
\begin{equation*}
    \scalprod{q, \u_i} \geq \scalprod{q, \u_j}\;\text{ for all } q\in\overline{C_i},\,j\neq0.
\end{equation*}
Thus, inserting the concrete choice of $\calM$ into \eqref{eq:conj}, we obtain
\begin{equation*}
    g^*(q) =
    \begin{cases}
        0 & \text{if } \scalprod{q, \u_i} \leq \frac\alpha2\omega_0^2\text{ for all } 1\leq i\leq M,\\
        \scalprod{q,\u_i} - \frac\alpha2\omega_0^2 & \text{if }q\in\overline{C_i}\text{ and }\scalprod{q, \u_i} \geq \frac\alpha2\omega_0^2.
    \end{cases}
\end{equation*}
Let us therefore introduce the sets (cf.~\cref{fig:subdomainsbloch:subdiff})
\begin{align*}
    Q_0 &\defgl \set{ q \in \R^2}{\scalprod{q, \u_i} < \tfrac\alpha2\omega_0^2\text{ for all } 1\leq i\leq M} ,\\
    Q_i &\defgl \set{ q \in C_i}{\scalprod{q, \u_i} > \tfrac\alpha2\omega_0^2},&1\leq i\leq M,\\
    Q_{i_1\ldots i_k} &\defgl \bigcap_{i\in\{i_1,\ldots,i_k\}}\overline{Q_i}\setminus\bigcup_{i\notin\{i_1,\ldots,i_k\}}\overline{Q_i},&0\leq i_1,\ldots,i_k\leq M.
\end{align*}
With this notation we obtain
\begin{equation*}
    g^*(q)=
    \begin{cases}
        0 & \text{if } q\in\overline{Q_0},\\
        \scalprod{q,\u_i} - \frac\alpha2\omega_0^2 & \text{if }q\in\overline{Q_i},\quad 1\leq i\leq M.
    \end{cases}
\end{equation*}

\paragraph{Subdifferential}

From the maximum rule \eqref{eq:conj_subdiff}, we directly obtain
\begin{equation}\label{eq:conj_subdiff_radial}
    \partial g^*(q)=
    \begin{cases}
        \{\u_i\}&\text{if }q\in Q_i,\quad\qquad 0\leq i\leq M,\\
        \co\{\u_{i_1},\ldots,\u_{i_k}\}&\text{if }q\in Q_{i_1\ldots i_k},\quad\,0\leq i_1,\ldots,i_k\leq M.
    \end{cases}
\end{equation}

\begin{figure}[t]
    \begin{subfigure}[t]{0.45\linewidth}
        \centering
        \begin{tikzpicture}[scale=0.8]
    \draw[gray,->](0,-3.5) -- (0,3.5) node[left] {\small $q_2$};
    \draw[gray,->](-3.5,0) -- (3.5,0) node[below] {\small $q_1$};
    \draw (0,0) node[circle,fill=white,inner sep=0pt] {$Q_0$};
    \foreach \x  in {1,...,6}{
        \coordinate (q\x) at (20+\x*60:1.5cm);
        \coordinate (Q\x) at (20+\x*60:3.5cm);
        \node at (\x*60-10:2.5cm) {$Q_\x$};
        \draw (q\x) -- (Q\x);
    }
    \draw (q1) -- (q2) -- (q3) -- (q4) -- (q5) -- (q6) -- cycle;
    \foreach \x [evaluate=\x as \xnext using {int(mod(\x,6)+1)}] in {1,...,6}{
        \draw (\x*60-10:1.3cm) node[circle,fill=white,inner sep=0pt] {\tiny $Q_{0\x}$};
        \draw (20+\x*60:2.5cm) node[circle,fill=white,inner sep=0pt] {\tiny $Q_{\x\xnext}$};
        \draw (20+\x*60:1.5cm) node[circle,fill=white,inner sep=0pt] {\tiny $Q_{0\x\xnext}$};
    }
\end{tikzpicture}
        \caption{subgradient $\partial g^*$}\label{fig:subdomainsbloch:subdiff}
    \end{subfigure}
    \hfill
    \begin{subfigure}[t]{0.45\linewidth}
        \centering
        \begin{tikzpicture}[scale=0.8]
    \draw[gray,->](0,-3.5) -- (0,3.5) node[left] {\small $q_2$};
    \draw[gray,->](-3.5,0) -- (3.5,0) node[below] {\small $q_1$};
    \draw (0,0) node[circle,fill=white] {$Q^\gamma_0$};
    \foreach \x [evaluate=\x as \xnext using {int(mod(\x,6)+1)}] in {1,...,6}{
        \coordinate (q\x) at (20+\x*60:1.5cm);
        \coordinate (q\x1) at (20+\x*60-9:2.2cm);
        \coordinate (q\x2) at (20+\x*60+9:2.2cm);
        \coordinate (Q\x1) at (20+\x*60-6:3.5cm);
        \coordinate (Q\x2) at (20+\x*60+6:3.5cm);
        \node at (\x*60-10:3cm) {$Q^\gamma_\x$};
        \node at (\x*60-10:1.7cm) {\tiny $Q^\gamma_{0\x}$};
        \node at (20+\x*60:3cm) {\tiny $Q^\gamma_{\x\xnext}$};
        \node at (20+\x*60:1.95cm) {\tiny $Q^\gamma_{0\x\xnext}$};
    }
    \foreach \x [evaluate=\x as \xnext using {int(mod(\x,6)+1)}] in {1,...,6}{
        \draw (Q\x1) -- (q\x1) -- (q\x) -- (q\x2) -- (Q\x2);
        \draw (q\x) -- (q\xnext);
        \draw (q\x1) -- (q\x2) -- (q\xnext1);
    }
\end{tikzpicture}
        \caption{Moreau--Yosida regularization $(\partial g^*)_\gamma$}\label{fig:subdomainsbloch:my}
    \end{subfigure}
    \caption{Subdomains for radially distributed $\calM$}
    \label{fig:subdomainsbloch}
\end{figure}
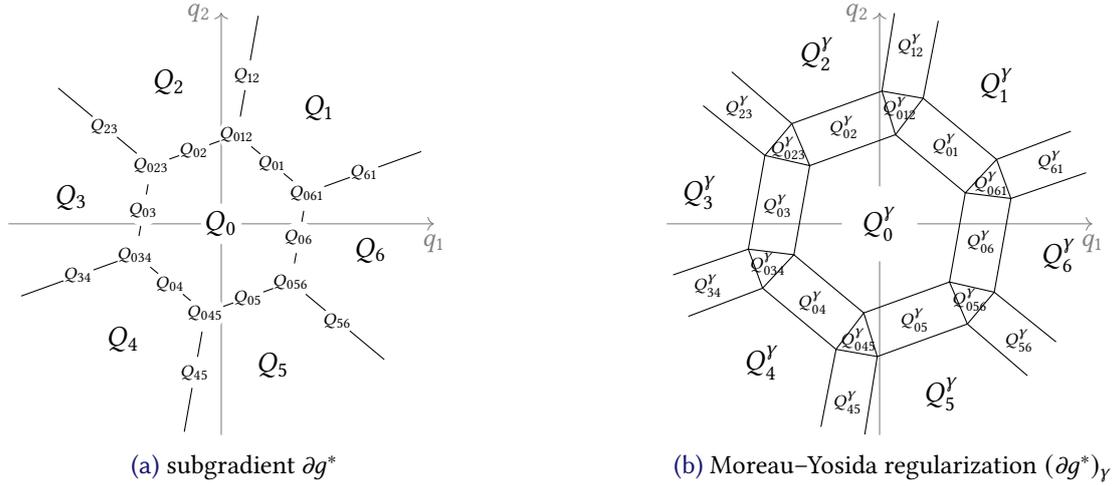

\paragraph{Proximal mapping}

Here, we proceed as follows: For each $Q_{i_1\ldots i_k}$, we
\begin{enumerate}
    \item compute the set $Q_{i_1\ldots i_k}^\gamma\defgl(\Id+\gamma\partial g^*)(Q_{i_1\ldots i_k})$;
    \item solve for $w\in Q_{i_1\ldots i_k}$ the relation $q\in\{w\}+\gamma\partial g^*(w)$ for arbitrary $q\in Q_{i_1\ldots i_k}^\gamma$.
\end{enumerate}
By \eqref{eq:proxMapRelation}, we then have $w=\prox_{\gamma g^*}(q)$.
The details are provided in \cref{tab:prox_radial}, while the sets  $Q_{i_1\ldots i_k}^\gamma$ are visualized in \cref{fig:subdomainsbloch:my}.
\begin{table}
    \caption{Computation of proximal map for radially distributed control values ($i+1$ is to be understood modulo $M$)}
    \label{tab:prox_radial}
    \begin{tabular}{llll}
        \toprule
        $Q_{i_1\ldots i_k}$ & $(\Id+\gamma\partial g^*)(w)$  & $Q_{i_1\ldots i_k}^\gamma$               & $(\Id+\gamma\partial g^*)^{-1}(q)$\\
        \midrule
        $Q_0$               & $w$                            & $Q_0$                                    & $q$\\
        $Q_i$               & $w+\gamma \u_i$                 & $Q_i+\gamma \u_i$                         & $q-\gamma \u_i$\\
        $Q_{0,i}$            & $w+\gamma\co\{0,\u_i\}$         & $Q_{0,i}+[0,\gamma]\u_i$                   & $q - \left(\frac{\scalprod{q,\u_i}}{\omega_0^2} - \frac\alpha2\right)\u_i$\\
        $Q_{i,i+1}$         & $w+\gamma\co\{\u_i,\u_{i+1}\}$   & $Q_{i,i+1}+\gamma\co\{\u_i,\u_{i+1}\}$     & $q-\frac{\gamma(\u_i + \u_{i+1})}2-\frac{\scalprod{q,\u_i-\u_{i+1}}(\u_i-\u_{i+1})}{|\u_i-\u_{i+1}|_2^2}$\\
        $Q_{0,i,i+1}$       & $w+\gamma\co\{0,\u_i,\u_{i+1}\}$ & $Q_{0,i,i+1}+\gamma\co\{0,\u_i,\u_{i+1}\}$ & $\alpha\left(\frac{\omega_0}{|\u_i + \u_{i+1}|_2}\right)^2 (\u_i + \u_{i+1})$\\
        \bottomrule
    \end{tabular}
\end{table}

To explain the case $Q_{0,i}$, note that for $q\in Q_{0,i}^\gamma$ we must have by definition of the set $Q_{0,i}^\gamma$ that
\begin{equation*}
    (\Id+\gamma\partial g^*)^{-1}(q)=q-\lambda \u_i \in Q_{0,i}\subset\set{v\in\R^2}{\scalprod{v,\u_i}=\frac\alpha2\omega_0^2}
\end{equation*}
for an appropriate choice of $\lambda\in[0,\gamma]$. Thus,
\begin{equation*}
    \scalprod{q-\lambda \u_i,\u_i}=\tfrac\alpha2\omega_0^2
    \quad\text{ and so }\quad
    \lambda=\frac{\scalprod{q,\u_i}}{\omega_0^2} - \frac\alpha2.
\end{equation*}

Likewise, for $q\in Q_{i,i+1}^\gamma$ we must have
\begin{equation*}
    (\Id+\gamma\partial g^*)^{-1}(q)=q-\lambda \u_i-(\gamma-\lambda) \u_{i+1}\in Q_{i,i+1}\subset(\u_i-\u_{i+1})^\perp
\end{equation*}
for some $\lambda\in[0,\gamma]$. Thus,
\begin{equation*}
    0=\scalprod{q-\lambda \u_i-(\gamma-\lambda) \u_{i+1},\u_i-\u_{i+1}}=\scalprod{q,\u_i-\u_{i+1}}+(\tfrac\gamma2-\lambda)|\u_i-\u_{i+1}|_2^2
\end{equation*}
and so
\begin{equation*}
    \lambda=\frac\gamma2+\frac{\scalprod{q,\u_i-\u_{i+1}}}{|\u_i-\u_{i+1}|_2^2}.
\end{equation*}

Finally, note that $Q_{0,i,i+1}=\{\alpha(\frac{\omega_0}{|\u_i + \u_{i+1}|_2})^2 (\u_i + \u_{i+1})\}$ only contains a single element,
which must therefore be equal to $(\Id+\gamma\partial g^*)^{-1}(q)$ for all $q\in Q_{0,i,i+1}^\gamma$.

\paragraph{Moreau--Yosida regularization}

Inserting the above into definition \eqref{eq:my_reg} of the Moreau--Yosida regularization yields
\begin{equation}\label{eq:my_bloch}
    (\partial g^*)_{\gamma}(q) =
    \begin{cases}
        0 & \text{ if } q \in Q_0^{\gamma}, \\
        \u_i & \text{ if } q \in Q_i^{\gamma}, \\
        \left(\tfrac{\scalprod{q, \u_i}}{\gamma \omega_0^2} - \tfrac{\alpha}{2\gamma}\right) \u_i & \text{ if } q \in Q_{0,i}^{\gamma}, \\
        \frac{\u_i + \u_{i+1}}2+\frac{\scalprod{q,\u_i-\u_{i+1}}(\u_i-\u_{i+1})}{\gamma|\u_i-\u_{i+1}|_2^2} & \text{ if } q \in Q_{i,i+1}^\gamma,\\
        \frac{q}{\gamma} - \frac\alpha{\gamma}\left(\frac{\omega_0}{|\u_i + \u_{i+1}|_2}\right)^2 \left(\u_i + \u_{i+1}\right) &\text{ if } q \in Q_{0,i,i+1}^\gamma.
    \end{cases}
\end{equation}
Finally, in a numerical implementation it will be necessary to efficiently identify for a given $q\in\R^2$ the set $Q_{i_1\ldots i_k}^\gamma$ in which it is contained.
To this end, determine $i_q,j_q,k_q\in \{1,\ldots,M\}$ via
\begin{equation*}
    q\in\overline{C_{i_q}}\,,\qquad
    q-\gamma \u_{i_q}\in\overline{C_{j_q}}\,,\qquad
    q - \left(\frac{\scalprod{q,\u_{i_q}}}{\omega_0^2} - \frac\alpha2\right)\u_{i_q}\in\overline{C_{k_q}},
\end{equation*}
and set
\begin{equation*}
    \rho_q:=\scalprod{q,\u_{i_q}}\,,\qquad
    \sigma_q:=\scalprod{q-\tfrac\gamma2(\u_{i_q}+\u_{j_q}),\u_{i_q}+\u_{j_q}}.
\end{equation*}
Now it is straightforward to identify the correct subdomain via
\begin{align*}
    Q_0^\gamma&=\set{q\in\R^2}{\rho_q<\tfrac\alpha2\omega_0^2},\\
    Q_i^\gamma&=\set{q\in\R^2}{\rho_q>(\tfrac\alpha2+\gamma)\omega_0^2,\,i_q=i,\,j_q=i},\\
    Q_{0,i}^\gamma&=\set{q\in\R^2}{\tfrac\alpha2\omega_0^2\leq\rho_q\leq (\tfrac\alpha2+\gamma)\omega_0^2,\,i_q=i,\,k_q=i},\\
    Q_{i,i+1}^\gamma&=\set{q\in\R^2}{\{i,i+1\}=\{i_q,j_q\},\,\sigma_q>\alpha\omega_0^2},\\
    Q_{0,i,i+1}^\gamma&=\set{q\in\R^2}{\{i,i+1\}=\{i_q,i_q+\sign(\u_{i_q}\times q)\},\,k_q\neq i_q,\,\sigma_q\leq\alpha\omega_0^2}.
\end{align*}

\paragraph{Newton derivative}

We can take as a Newton derivative of \eqref{eq:my_bloch} at $q$ any element of the convex hull of the derivatives of the branches active at $q$; we choose here
\begin{equation}\label{eq:newton_bloch}
    D_N h_{\gamma}(q) = \begin{cases}
        0 & \text{ if } q \in Q_i^{\gamma}, \\
        \frac{1}{\gamma\omega_0^2} \u_i\u_i^T & \text{ if } q \in Q_{0,i}^{\gamma},\\
        \frac{1}{\gamma\abs{\u_i-\u_{i+1}}_2^2}(\u_i - \u_{i+1})(\u_i - \u_{i+1})^T & \text{ if } q \in Q_{i,i+1}^{\gamma},\\
        \frac{1}{\gamma}\Id & \text{ if } q \in Q_{0,i,i+1}^{\gamma}.
    \end{cases}
\end{equation}

\subsection{Concentric corners}\label{sec:multibang:elast}

We now address the case of admissible control values of different magnitudes, where we consider for the sake of an example the concrete set
\begin{equation}\label{eq:M_elast}
    \begin{aligned}
        \calM &= \left\{ \begin{pmatrix} 1 \\ 1 \end{pmatrix}, \begin{pmatrix} 1 \\ -1 \end{pmatrix}, \begin{pmatrix} -1 \\ 1 \end{pmatrix}, \begin{pmatrix} -1 \\ -1 \end{pmatrix}, \begin{pmatrix} 2 \\ 2 \end{pmatrix}, \begin{pmatrix} 2 \\ -2 \end{pmatrix}, \begin{pmatrix} -2 \\ 2 \end{pmatrix}, \begin{pmatrix} -2 \\ -2 \end{pmatrix}\right\} \\[0.5em]
        &= \left\{\u_{1,1}^1, \u_{1,-1}^1, \u_{-1,1}^1, \u_{-1,-1}^1, \u_{1,1}^2, \u_{1,-1}^2, \u_{-1,1}^2, \u_{-1,-1}^2\right\}.
    \end{aligned}
\end{equation}

\paragraph{Fenchel conjugate}
Again inserting $\calM$ into \eqref{eq:conj}, we see that the maximum is attained either by $v=(q_1/|q_1|,q_2/|q_2|)$ or by $v=2(q_1/|q_1|,q_2/|q_2|)$,
where in the case $q_i=0$ we may define $q_i/|q_i|\in\{-1,1\}$ arbitrarily. Hence we obtain after some algebraic manipulations
\begin{equation*}
    g^*(q)
    =\max\left\{\abs{q}_1-\alpha,2\abs{q}_1-4\alpha\right\}
    =
    \begin{cases}
        \abs{q}_1 - \alpha & \text{if }\abs{q}_1 \leq 3\alpha,\\
        2\abs{q}_1 - 4\alpha & \text{if }\abs{q}_1 \geq 3\alpha.
    \end{cases}
\end{equation*}

\paragraph{Subdifferential}

From \eqref{eq:conj_subdiff}, we directly obtain
\begin{equation*}
    \partial g^*(q) = \co\bigcup_{\substack{i\in\{1,2\}: \\ g^*(q) = g_i^*(q)}} \partial g_i^*(q)
    \quad\text{ for }\quad
    \begin{cases}
        g_1^*(q) = \abs{q}_1 - \alpha,\\
        g_2^*(q) = 2\abs{q}_1 - 4\alpha.
    \end{cases}
\end{equation*}
In the above we have
\begin{equation*}
    \partial g_1^*(q) = \begin{pmatrix} \sign(q_1) \\ \sign(q_2)\end{pmatrix},
    \qquad
    \partial g_2^*(q) = 2\begin{pmatrix} \sign(q_1) \\ \sign(q_2)\end{pmatrix},
\end{equation*}
where $\sign$ denotes the set-valued sign of convex analysis, i.e., $\sign(0)=[-1,1]$.
Therefore we obtain
\begin{equation*}
    \partial g^*(q) = \left.\begin{cases}
            \partial g_1^*(q) & \text{if }\abs{q}_{1} < 3\alpha\\
            \partial g_2^*(q) & \text{if }\abs{q}_{1} > 3\alpha\\
            \co\{\partial g_1^*(q),\partial g_2^*(q) \} & \text{if }\abs{q}_{1} = 3\alpha
    \end{cases}\right\}
    = \begin{pmatrix} \sign(q_1) \\ \sign(q_2)\end{pmatrix}\cdot
    \begin{cases}
        1 & \text{if }\abs{q}_{1} < 3\alpha,\\
        2 & \text{if }\abs{q}_{1} > 3\alpha,\\
        [1,2] & \text{if }\abs{q}_{1} = 3\alpha.
    \end{cases}
\end{equation*}
For economy, let us introduce for $i,j,k\in\{-1,0,1\}$ the sets
\begin{equation*}
    I_k
    =\begin{cases}(-\infty,0)&\text{if }k=-1,\\
    \{0\}&\text{if }k=0\\(0,\infty),&\text{if }k=1,\\\end{cases}
    \qquad\text{and}\qquad
    Q_{ijk}=\set{q\in\R^2}{q_1\in I_i,\,q_2\in I_j,\,\abs{q}_1-3\alpha\in I_k}.
\end{equation*}
A visualization is given in \cref{fig:subdomainselast:subdiff}.
Note that the index $0$ always indicates a lower-dimensional structure; in particular, we have
\begin{equation*}
    Q_{0jk}\subset\overline{Q_{-1,j,k}}\cap\overline{Q_{1,j,k}},\quad
    Q_{i0k}\subset\overline{Q_{i,-1,k}}\cap\overline{Q_{i,1,k}},\quad
    Q_{ij0}\subset\overline{Q_{i,j,-1}}\cap\overline{Q_{i,j,1}}.
\end{equation*}
Using this notation, we can write the subdifferential as
\begin{equation}\label{eq:conj_subdiff_concentric}
    \partial g^*(q)=\begin{cases}
        \u_{ij}^{(k+3)/2} & \text{if }q\in Q_{ijk},\ i,j,k\in\{-1,1\},\\
        \co\big\{\u_{rs}^{(t+3)/2}: r,s,t\in\{-1,1\},\,|r-i|,|s-j|,|t-k|\leq1\big\} & \text{if }q\in Q_{ijk},\ 0\in\{i,j,k\},
    \end{cases}
\end{equation}
which provides more insight into its structure.
In particular, on each lower-dimensional $Q_{ijk}$ the subdifferential is the convex hull of the subdifferentials on the adjacent two-dimensional sets.

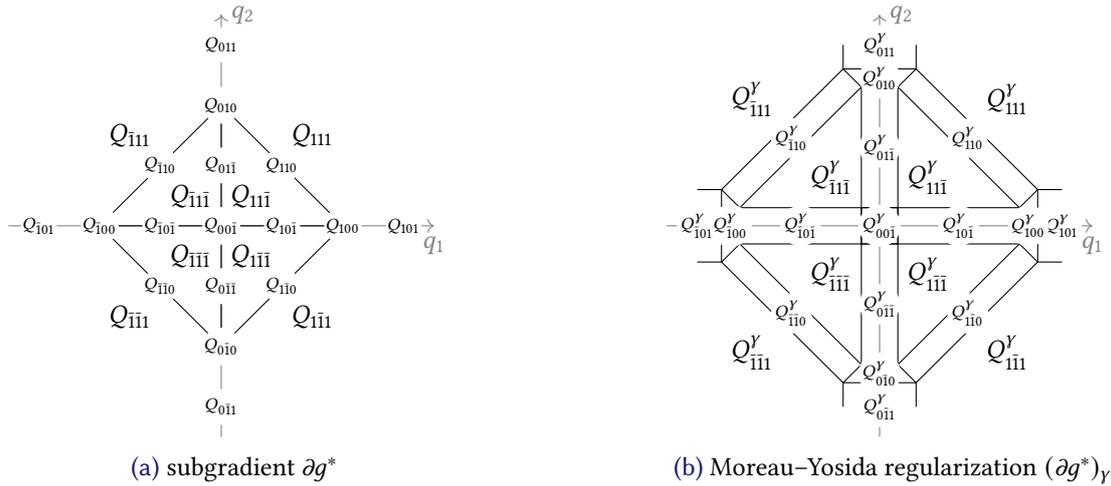
\begin{figure}[t]
    \begin{subfigure}[t]{0.45\linewidth}
        \centering
        \begin{tikzpicture}[scale=0.8]
    \draw[gray,->](0,-3.5) -- (0,3.5) node[right] {\small $q_2$};
    \draw[gray,->](-3.5,0) -- (3.5,0) node[below] {\small $q_1$};

    \draw (2,0) -- (0,2) -- (-2,0) -- (0,-2) -- cycle;
    \draw(0,-2) -- (0,2);
    \draw(-2,0) -- (2,0);

    \draw (0,0) node[circle,fill=white,inner sep=0pt] {\tiny $Q_{00\bar1}$};
    \draw (0.5,0.5) node[circle,fill=white,inner sep=0pt] {\small$Q_{11\bar1}$};
    \draw (0.5,-0.5) node[circle,fill=white,inner sep=0pt] {\small$Q_{1\bar1\bar1}$};
    \draw (-0.5,0.5) node[circle,fill=white,inner sep=0pt] {\small$Q_{\bar11\bar1}$};
    \draw (-0.5,-0.5) node[circle,fill=white,inner sep=0pt] {\small$Q_{\bar1\bar1\bar1}$};

    \draw (1.5,1.5) node[circle,fill=white,inner sep=0pt] {\small$Q_{111}$};
    \draw (1.5,-1.5) node[circle,fill=white,inner sep=0pt] {\small$Q_{1\bar11}$};
    \draw (-1.5,1.5) node[circle,fill=white,inner sep=0pt] {\small$Q_{\bar111}$};
    \draw (-1.5,-1.5) node[circle,fill=white,inner sep=0pt] {\small$Q_{\bar1\bar11}$};

    \draw (1,0) node[circle,fill=white,inner sep=0pt] {\tiny $Q_{10\bar1}$};
    \draw (0,1) node[circle,fill=white,inner sep=0pt] {\tiny $Q_{01\bar1}$};
    \draw (-1,0) node[circle,fill=white,inner sep=0pt] {\tiny $Q_{\bar10\bar1}$};
    \draw (0,-1) node[circle,fill=white,inner sep=0pt] {\tiny $Q_{0\bar1\bar1}$};

    \draw (3,0) node[circle,fill=white,inner sep=0pt] {\tiny $Q_{101}$};
    \draw (0,3) node[circle,fill=white,inner sep=0pt] {\tiny $Q_{011}$};
    \draw (-3,0) node[circle,fill=white,inner sep=0pt] {\tiny $Q_{\bar101}$};
    \draw (0,-3) node[circle,fill=white,inner sep=0pt] {\tiny $Q_{0\bar11}$};

    \draw (1,1) node[circle,fill=white,inner sep=0pt] {\tiny $Q_{110}$};
    \draw (1,-1) node[circle,fill=white,inner sep=0pt] {\tiny $Q_{1\bar10}$};
    \draw (-1,1) node[circle,fill=white,inner sep=0pt] {\tiny $Q_{\bar110}$};
    \draw (-1,-1) node[circle,fill=white,inner sep=0pt] {\tiny $Q_{\bar1\bar10}$};

    \draw (2,0) node[circle,fill=white,inner sep=0pt] {\tiny $Q_{100}$};
    \draw (0,2) node[circle,fill=white,inner sep=0pt] {\tiny $Q_{010}$};
    \draw (-2,0) node[circle,fill=white,inner sep=0pt] {\tiny $Q_{\bar100}$};
    \draw (0,-2) node[circle,fill=white,inner sep=0pt] {\tiny $Q_{0\bar10}$};
\end{tikzpicture}
        \caption{subgradient $\partial g^*$}\label{fig:subdomainselast:subdiff}
    \end{subfigure}
    \hfill
    \begin{subfigure}[t]{0.45\linewidth}
        \centering
        \begin{tikzpicture}[scale=0.8]
    \draw[gray,->](0,-3.5) -- (0,3.5) node[right] {\small $q_2$};
    \draw[gray,->](-3.5,0) -- (3.5,0) node[below] {\small $q_1$};

    \draw (0.3,0.3) -- (0.3,-0.3) -- (-0.3,-0.3) -- (-0.3,0.3) -- cycle;
    \draw (0.3,0.3) -- (0.3,2.3) -- (-0.3,2.3) -- (-0.3,0.3) -- cycle;
    \draw (0.3,-0.3) -- (0.3,-2.3) -- (-0.3,-2.3) -- (-0.3,-0.3) -- cycle;
    \draw (0.3,0.3) -- (2.3,0.3) -- (2.3,-0.3) -- (0.3,-0.3) -- cycle;
    \draw (-0.3,-0.3) -- (-2.3,-0.3) -- (-2.3,0.3) -- (-0.3,0.3) -- cycle;

    \draw (0.6,3) -- (0.6,2.6) -- (-0.6,2.6) -- (-0.6,3);
    \draw (0.6,-3) -- (0.6,-2.6) -- (-0.6,-2.6) -- (-0.6,-3);
    \draw (3,0.6) -- (2.6,0.6) -- (2.6,-0.6) -- (3,-0.6);
    \draw (-3,0.6) -- (-2.6,0.6) -- (-2.6,-0.6) -- (-3,-0.6);

    \draw (2.3,0.3) -- (0.3,2.3) -- (0.6,2.6) -- (2.6,0.6) -- cycle;
    \draw (2.3,-0.3) -- (0.3,-2.3) -- (0.6,-2.6) -- (2.6,-0.6) -- cycle;
    \draw (-2.3,0.3) -- (-0.3,2.3) -- (-0.6,2.6) -- (-2.6,0.6) -- cycle;
    \draw (-2.3,-0.3) -- (-0.3,-2.3) -- (-0.6,-2.6) -- (-2.6,-0.6) -- cycle;
    
    \draw (0,0) node[circle,fill=white,inner sep=0pt] {\tiny $Q^\gamma_{00\bar1}$};
    \draw (0.8,0.8) node[circle,fill=white,inner sep=0pt] {\small$Q^\gamma_{11\bar1}$};
    \draw (0.8,-0.8) node[circle,fill=white,inner sep=0pt] {\small$Q^\gamma_{1\bar1\bar1}$};
    \draw (-0.8,0.8) node[circle,fill=white,inner sep=0pt] {\small$Q^\gamma_{\bar11\bar1}$};
    \draw (-0.8,-0.8) node[circle,fill=white,inner sep=0pt] {\small$Q^\gamma_{\bar1\bar1\bar1}$};

    \draw (2.1,2.1) node[circle,fill=white,inner sep=0pt] {\small$Q^\gamma_{111}$};
    \draw (2.1,-2.1) node[circle,fill=white,inner sep=0pt] {\small$Q^\gamma_{1\bar11}$};
    \draw (-2.1,2.1) node[circle,fill=white,inner sep=0pt] {\small$Q^\gamma_{\bar111}$};
    \draw (-2.1,-2.1) node[circle,fill=white,inner sep=0pt] {\small$Q^\gamma_{\bar1\bar11}$};

    \draw (1.3,0) node[circle,fill=white,inner sep=0pt] {\tiny $Q^\gamma_{10\bar1}$};
    \draw (0,1.3) node[circle,fill=white,inner sep=0pt] {\tiny $Q^\gamma_{01\bar1}$};
    \draw (-1.3,0) node[circle,fill=white,inner sep=0pt] {\tiny $Q^\gamma_{\bar10\bar1}$};
    \draw (0,-1.3) node[circle,fill=white,inner sep=0pt] {\tiny $Q^\gamma_{0\bar1\bar1}$};

    \draw (3,0) node[circle,fill=white,inner sep=0pt] {\tiny $Q^\gamma_{101}$};
    \draw (0,3) node[circle,fill=white,inner sep=0pt] {\tiny $Q^\gamma_{011}$};
    \draw (-3,0) node[circle,fill=white,inner sep=0pt] {\tiny $Q^\gamma_{\bar101}$};
    \draw (0,-3) node[circle,fill=white,inner sep=0pt] {\tiny $Q^\gamma_{0\bar11}$};

    \draw (1.45,1.45) node[circle,fill=white,inner sep=0pt] {\tiny $Q^\gamma_{110}$};
    \draw (1.45,-1.45) node[circle,fill=white,inner sep=0pt] {\tiny $Q^\gamma_{1\bar10}$};
    \draw (-1.45,1.45) node[circle,fill=white,inner sep=0pt] {\tiny $Q^\gamma_{\bar110}$};
    \draw (-1.45,-1.45) node[circle,fill=white,inner sep=0pt] {\tiny $Q^\gamma_{\bar1\bar10}$};

    \draw (2.45,0) node[circle,fill=white,inner sep=0pt] {\tiny $Q^\gamma_{100}$};
    \draw (0,2.45) node[circle,fill=white,inner sep=0pt] {\tiny $Q^\gamma_{010}$};
    \draw (-2.45,0) node[circle,fill=white,inner sep=0pt] {\tiny $Q^\gamma_{\bar100}$};
    \draw (0,-2.45) node[circle,fill=white,inner sep=0pt] {\tiny $Q^\gamma_{0\bar10}$};
\end{tikzpicture}
        \caption{Moreau--Yosida regularization $(\partial g^*)_\gamma$}\label{fig:subdomainselast:my}
    \end{subfigure}
    \caption{Subdomains for concentric corners, where $\bar1$ is written for $-1$ to simplify notation (the line dimensions are provided in \cref{fig:subdiff:elast:dimensions})}
    \label{fig:subdiff:elast}
\end{figure}
\begin{figure}[t]
    \begin{subfigure}[t]{0.45\linewidth}
        \centering
        \begin{tikzpicture}[scale=0.8]
    \draw[gray,->](0,-3.5) -- (0,3.5) node[right] {\small $q_2$};
    \draw[gray,->](-3.5,0) -- (3.5,0) node[below] {\small $q_1$};

    \draw (2,0) -- (0,2) -- (-2,0) -- (0,-2) -- cycle;
    \draw(0,-2) -- (0,2);
    \draw(-2,0) -- (2,0);

    \draw (2,0) node[below] {\small\hspace*{1ex}$3\alpha$};
    \draw (0,2) node[left] {\small$3\alpha$};
    \draw (-2,0) node[below] {\small $-3\alpha$\hspace*{1ex}};
    \draw (0,-2) node[left] {\small $-3\alpha$};
\end{tikzpicture}
        \caption{subgradient $\partial g^*$}\label{fig:dimensionselast:subdiff}
    \end{subfigure}
    \hfill
    \begin{subfigure}[t]{0.45\linewidth}
        \centering
        \begin{tikzpicture}[scale=0.8]
    \draw (0.3,0.3) -- (0.3,-0.3) -- (-0.3,-0.3) -- (-0.3,0.3) -- cycle;
    \draw (0.3,0.3) -- (0.3,2.3) -- (-0.3,2.3) -- (-0.3,0.3) -- cycle;
    \draw (0.3,-0.3) -- (0.3,-2.3) -- (-0.3,-2.3) -- (-0.3,-0.3) -- cycle;
    \draw (0.3,0.3) -- (2.3,0.3) -- (2.3,-0.3) -- (0.3,-0.3) -- cycle;
    \draw (-0.3,-0.3) -- (-2.3,-0.3) -- (-2.3,0.3) -- (-0.3,0.3) -- cycle;

    \draw (0.6,3) -- (0.6,2.6) -- (-0.6,2.6) -- (-0.6,3);
    \draw (0.6,-3) -- (0.6,-2.6) -- (-0.6,-2.6) -- (-0.6,-3);
    \draw (3,0.6) -- (2.6,0.6) -- (2.6,-0.6) -- (3,-0.6);
    \draw (-3,0.6) -- (-2.6,0.6) -- (-2.6,-0.6) -- (-3,-0.6);

    \draw (2.3,0.3) -- (0.3,2.3) -- (0.6,2.6) -- (2.6,0.6) -- cycle;
    \draw (2.3,-0.3) -- (0.3,-2.3) -- (0.6,-2.6) -- (2.6,-0.6) -- cycle;
    \draw (-2.3,0.3) -- (-0.3,2.3) -- (-0.6,2.6) -- (-2.6,0.6) -- cycle;
    \draw (-2.3,-0.3) -- (-0.3,-2.3) -- (-0.6,-2.6) -- (-2.6,-0.6) -- cycle;
    
    \draw [gray,line width = 3pt] (0,0.3) -- (2.3,0.3) -- (2.6,0.6) -- (3,0.6) node [above] {\small $\eta(q_1)$};
    
    \draw [white,line width = 15pt] (0,-0.6) -- (0.3,-0.6);
    \draw [white,line width = 20pt] (0,-1.2) -- (0.6,-1.2);
    \draw [white,line width = 20pt] (0,-1.8) -- (2.3,-1.8);
    \draw [white,line width = 35pt] (0,-2.4) -- (2.6,-2.4);
    
    \draw [<->] (0,-0.6) -- (0.3,-0.6);
    \draw [<->] (0,-1.2) -- (0.6,-1.2);
    \draw [<->] (0,-1.8) -- (2.3,-1.8);
    \draw [<->] (0,-2.4) -- (2.6,-2.4);

    \draw (0.15,-0.6) node[below]{\small $\gamma$};
    \draw (0.3,-1.2) node[below]{\small $2\gamma$};
    \draw (1.15,-1.8) node[below]{\small $3\alpha+\gamma$};
    \draw (1.3,-2.4) node[below]{\small $3\alpha+2\gamma$};

    \draw [gray,->] (0,-3.5) -- (0,3.5) node[right] {\small $q_2$};
    \draw [gray,->] (-3.5,0) -- (3.5,0) node[below] {\small $q_1$};
\end{tikzpicture}
        \caption{Moreau--Yosida regularization $(\partial g^*)_\gamma$}\label{fig:dimensionselast:my}
    \end{subfigure}
    \centering
    \caption{Dimensions for \cref{fig:subdiff:elast}}
    \label{fig:subdiff:elast:dimensions}
\end{figure}

\paragraph{Proximal mapping}

To obtain the Moreau--Yosida regularization of $\partial g^*$ for $\gamma>0$, we proceed as above by first noting that $w=(\Id + \gamma\partial g^*)^{-1}(q) \in Q_{ijk}$ holds if and only if
\begin{equation*}
    q \in (\Id+\gamma\partial g^*)(Q_{ijk})=: Q_{ijk}^{\gamma}.
\end{equation*}
A visualization of these sets is provided in \cref{fig:subdomainselast:my}; we postpone their discussion to the end of the section and first calculate the specific value of the proximal mapping based on \eqref{eq:proxMapRelation} together with the case distinction in the subdifferential.

Let $w\in Q_{ijk}$ and correspondingly $q\in Q_{ijk}^\gamma$ for some $i,j,k\in\{-1,0,1\}$.
\begin{enumerate}[{label={(\roman*)}, ref={(\roman*)}}]
    \item
        If $i,j,k\in\{-1,1\}$ we have $(\Id+\gamma\partial g^*)(w)=w+\gamma \u_{ij}^{(k+3)/2}$ so that
        \begin{equation*}
            (\Id+\gamma\partial g^*)^{-1}(q)=q-\gamma \u_{ij}^{(k+3)/2}\qquad\text{for }q\in Q_{ijk}^\gamma\text{ with }i,j,k\in\{-1,1\}.
        \end{equation*}
    \item
        If two of $i,j,k$ are zero, $(\Id+\gamma\partial g^*)^{-1}(q)$ must be the single unique element of $Q_{ijk}$ and thus
        \begin{equation*}
            (\Id+\gamma\partial g^*)^{-1}(q)=\begin{cases}
                0&\text{if }q\in Q_{0,0,-1}^\gamma,\\
                3\alpha(i,j)&\text{if }q\in Q_{i,j,0}^\gamma\text{ with }i=0\text{ or }j=0.
            \end{cases}
        \end{equation*}
    \item
        If $i=0$ and $j,k\neq0$, then for $w\in Q_{0jk}$ we have
        \begin{equation*}
            (\Id+\gamma\partial g^*)(w)=w+\gamma\co\left\{\u_{-1,j}^{(k+3)/2},\u_{1,j}^{(k+3)/2}\right\}=w+\gamma\frac{k+3}2([-1,1],j).
        \end{equation*}
        Thus for $q\in Q_{0jk}^\gamma$ we have $(\Id+\gamma\partial g^*)^{-1}(q)=q-\gamma\frac{k+3}2(\lambda,j)$, where $\lambda\in[-1,1]$ is such that $q-\gamma\frac{k+3}2(\lambda,j)\in Q_{0jk}\subset\{0\}\times\R$.
        Therefore $\lambda=\frac2{\gamma(k+3)}q_1$ and
        \begin{equation*}
            (\Id+\gamma\partial g^*)^{-1}(q)=\left(0,q_2-\gamma\tfrac{k+3}2j\right)\qquad\text{for }q\in Q_{0jk}^\gamma\text{ with }j,k\in\{-1,1\}.
        \end{equation*}
        Analogously,
        \begin{equation*}
            (\Id+\gamma\partial g^*)^{-1}(q)=\left(q_1-\gamma\tfrac{k+3}2i,0\right)\qquad\text{for }q\in Q_{i0k}^\gamma\text{ with }i,k\in\{-1,1\}.
        \end{equation*}
    \item
        If $k=0$ and $i,j\neq0$, then for $w\in Q_{ij0}$ we have
        \begin{equation*}
            (\Id+\gamma\partial g^*)(w)=w+\gamma\co\left\{\u_{ij}^1,\u_{ij}^2\right\}=w+\gamma[1,2](i,j).
        \end{equation*}
        Thus for $q\in Q_{ij0}^\gamma$ we have $(\Id+\gamma\partial g^*)^{-1}(q)=q-\gamma\lambda(i,j)$, where $\lambda\in[1,2]$ is such that $q-\gamma\lambda(i,j)\in Q_{ij0}\subset\{w\in\R^2:\abs{w}_1=3\alpha\}$.
        Therefore $\lambda=\frac{\abs{q}_1-3\alpha}{2\gamma}$ and
        \begin{equation*}
            (\Id+\gamma\partial g^*)^{-1}(q)=q-\tfrac{\abs{q}_1-3\alpha}{2}(i,j)\qquad\text{for }q\in Q_{ij0}^\gamma\text{ with }i,j\in\{-1,1\}.
        \end{equation*}
\end{enumerate}

It remains to discuss the sets $Q_{ijk}^\gamma$.
Rather than list all sets explicitly, we instead provide a procedure for determining for a given $q\in\R^2$ the corresponding subdomain, which is what is actually required for the numerical implementation.
For that purpose, let us introduce the function (compare the illustration in \cref{fig:dimensionselast:my})
\begin{equation*}
    \eta(x)=\begin{cases}\gamma&\text{if }x<3\alpha+\gamma,\\x-3\alpha&\text{if }3\alpha+\gamma\leq x\leq3\alpha+2\gamma,\\2\gamma&\text{if }x>3\alpha+2\gamma.\end{cases}
\end{equation*}
With this function we have $q\in Q_{ijk}^\gamma$ for $i,j,k$ given by
\begin{align*}
    i&=\begin{cases}0&\text{if }|q_1|\leq\eta(|q_2|),\\\sign(q_1)&\text{otherwise,}\end{cases}\\[1em]
    j&=\begin{cases}0&\text{if }|q_2|\leq\eta(|q_1|),\\\sign(q_2)&\text{otherwise,}\end{cases}\\[1em]
    k&=\begin{cases}
        -1&\text{if }\abs{q}_\infty<3\alpha+\gamma\text{ and }\abs{q}_1<3\alpha+2\gamma,\\
        1&\text{if }\abs{q}_\infty>3\alpha+2\gamma\text{ or }\abs{q}_1>3\alpha+4\gamma,\\
    0&\text{otherwise.}\end{cases}
\end{align*}

\paragraph{Moreau--Yosida regularization}

Inserting this into the definition \eqref{eq:my_reg} of the Moreau--Yosida regularization yields
\begin{equation}\label{eq:myreg_elast}
    (\partial g^*)_\gamma(q)=\begin{cases}
        \u_{ij}^{(k+3)/2} & \text{if }q\in Q_{ijk}^\gamma\text{ with }i,j,k\neq 0,\\
        \frac1\gamma(q-3\alpha(i,j)) & \text{if }q\in Q_{ijk}^\gamma\text{ with } |i|+|j|+|k| = 1,\\
        (\frac1\gamma q_1,\frac{k+3}2j) & \text{if }q\in Q_{0jk}^\gamma\text{ with }j,k\neq 0,\\
        (\frac{k+3}2i,\frac1\gamma q_2) & \text{if }q\in Q_{i0k}^\gamma\text{ with }i,k\neq 0,\\
        \frac{\abs{q}_1-3\alpha}{2\gamma}(i,j) & \text{if }q\in Q_{ij0}^\gamma\text{ with }i,j\neq 0.
    \end{cases}
\end{equation}

\paragraph{Newton derivative}

Finally, we can again take as a Newton derivative any element of the Clarke gradient; here, we choose
\begin{equation}\label{eq:newton_elast}
    D_Nh_{\gamma}(q) = \begin{cases}
        0 & \text{ if } q\in Q_{ijk}^\gamma\text{ with }i,j,k\neq 0, \\
        \frac{1}{\gamma} \Id & \text{ if } q\in Q_{ijk}^\gamma\text{ with } |i|+|j|+|k| = 1,\\

        \frac{1}{\gamma}(j,i)^T(j,i)
        & \text{ if }q\in Q_{ijk}^\gamma\text{ with } |i| + |j| = 1,k\neq 0,\\

        \frac{1}{2\gamma}(i,j)^T(i,j) & \text{ if } q\in Q_{ij0}^\gamma\text{ with }i,j\neq 0. \\
    \end{cases}
\end{equation}

\subsection{General multibang control}\label{sec:multibang:general}

For more than two control dimensions or arbitrary sets $\calM$ an explicit calculation quickly becomes complicated.
However, note that $Q_{i_1\ldots i_k}$ is actually empty for most index collections $\{i_1,\ldots,i_k\}$,
so that the number of conditions to be checked in an algorithmic evaluation can be greatly reduced.
Indeed, by construction the sets $Q_{i_1\ldots i_k}$ are nothing else but the preimages under $g^*$ of the faces of the (polyhedral) graph of $g^*$.
Thus, finding all nonempty $Q_{i_1\ldots i_k}$ is equivalent to enumerating all faces of the graph of $g^*$,
which for given choices of $c$ and the $\bar u_i$ can be done by existing face enumeration algorithms such as \cite{FuLiMa97}.
Hence the general procedure outlined in the beginning of this section can be implemented
for given $\calM=\{\bar u_i:i\in I\}$ and $c$
as follows:
\begin{enumerate}
    \item
        Use the algorithm from \cite{FuLiMa97} to list all faces of
        \begin{equation*}
            \mathrm{epi}\,g^*=\{(q,t)\in\R^{m+1}:t\geq\langle \bar u_i,q\rangle-\alpha c(\bar u_i)\text{ for }i\in I\}.
        \end{equation*}
        where a face is identified by the indices $\{i_1,\ldots,i_k\}\subset I$ whose constraints are active on that face; these index collections are exactly those for which $Q_{i_1\ldots i_k}$ is nonempty.
    \item
        For each face $\mathfrak f=\{i_1,\ldots,i_k\}$, compute
        $A_{\mathfrak f}$, $b_{\mathfrak f}$ and $A_{\mathfrak f;l}$, $b_{\mathfrak f;l}$
        from \eqref{eqn:multibangGeneralSystem}.
    \item
        To evaluate $h_\gamma$ at a vector $q\in\R^m$, calculate
        $w_{\mathfrak f}=A_{\mathfrak f}q+b_{\mathfrak f}$ and
        $\lambda_{\mathfrak f;l}=A_{\mathfrak f;l}q+b_{\mathfrak f;l}$
        for each face $\mathfrak f$.
        Identify the (unique) face $\mathfrak f$ such that
        $0\leq\lambda_{\mathfrak f;l}\leq1$ for all $l$ and $g_{\bar u_{i_1}}(q)>g_{\bar u_i}(q)$ for $i_1\in\mathfrak f$ and all $i\notin\mathfrak f$.
        Then set
        \begin{equation*}
            h_\gamma(q)=\frac1\gamma(q-A_{\mathfrak f}q-b_{\mathfrak f})\qquad\text{and}\qquad D_Nh_\gamma=\frac1\gamma(\Id-A_{\mathfrak f}).
        \end{equation*}
\end{enumerate}
We point out that the computationally most expensive step -- enumerating the faces of $\mathrm{epi}\,g^*$, which requires solving a linear program -- is independent of $q$ and $\gamma$ and can thus be precomputed.

\section{State equation}\label{sec:stateEq}

In this section, we specify in more detail our model state operators and verify that assumptions \ref{enm:weakWeakContinuity}--\ref{enm:adjointConvergence} of \cref{sec:existence,sec:stability} are satisfied for our model problems.

\subsection{Bloch equation}\label{sec:BlochStateEq}

As our motivating model problem, we consider the Bloch equation in a rotating reference frame without relaxation
\begin{equation*}
    \frac{\dd}{\dd t}{\mathbf{M}}^{(\omega)}(t) = \mathbf{M}^{(\omega)}(t) \times \mathbf B^{(\omega)}(t),
    \qquad\mathbf{M}^{(\omega)}(0)=(0,0,1)^T,
\end{equation*}
which describes the temporally evolving magnetization $\mathbf M^{(\omega)}\in\R^3$ of an ensemble of spins rotating at the same resonance offset frequency $\omega$ (called \emph{isochromat}),
starting from a given equilibrium magnetization.
The time-varying effective magnetic field $\mathbf B^{(\omega)}(t)$ is of the form
\begin{equation*}
    \mathbf B^{(\omega)}(t)=(\omega_x(t),\omega_y(t),\omega)^T,
\end{equation*}
where $u(t):=(\omega_x(t),\omega_y(t))\in \R^2$ can be controlled.
The aim is to achieve a magnetization $\mathbf M^{(\omega)}(T)=\mathbf M_d$ within the time interval $\Omega=[0,T]$ for a list of offset frequencies $\omega_1,\ldots,\omega_J$.
In terms of our previous notation we thus set
\begin{equation}\label{eq:settingBloch}
    S:L^2(\Omega;\R^2) \to (\R^3)^J,\qquad u\mapsto\left[\mathbf M^{(\omega_1)}(T),\ldots,\mathbf M^{(\omega_J)}(T)\right].
\end{equation}
This choice of $S$ satisfies the assumptions \ref{enm:weakWeakContinuity}--\ref{enm:adjointConvergence}; see \cref{sec:BlochProperties}.

\subsection{Linear elasticity}\label{sec:elasticityStateEq}

In this case, $\Omega\subset\R^2$ represents an elastic body fixed at $\Gamma\subset\partial\Omega$ (with positive Hausdorff measure $\mathcal{H}^1(\Gamma)>0$),
where we assume $\Gamma$ and $\partial\Omega\setminus\Gamma$ to be smooth or $\Omega$ to be a convex polygon with $\Gamma$ being the union of some faces.
The elastic body is subject to a controlled body force $u:\Omega\to\R^2$.
The resulting displacement $y:\Omega\to\R^2$ is governed by the equations of linearized elasticity with Lam\'e parameters $\mu$ and $\lambda$,
\begin{equation}\label{eq:lame}
    \left\{\begin{aligned}
            -2\mu \divergence \epsilon(y) - \lambda \grad \divergence y &= u \text{ in } \Omega, \\
            y &= 0 \text{ on } \Gamma, \\
            (2\mu \epsilon(y) + \lambda \divergence y) n &= 0 \text{ on } \partial\Omega\setminus\Gamma,
    \end{aligned}\right.
\end{equation}
where $n$ denotes the unit outward normal, $Dy=[\nabla y_1|\nabla y_2]^T$ is the displacement gradient, and $\epsilon(y)=\frac{Dy+Dy^T}2$ is the symmetrized gradient.
Defining
\begin{equation*}
    H^1_\Gamma(\Omega) \defgl \set{v \in H^1(\Omega;\R^2)}{v = 0 \text{ on } \Gamma },
\end{equation*}
we may take
\begin{equation}\label{eq:settingElasticity}
    S: H_{\Gamma}^1(\Omega)^* \to H_{\Gamma}^1(\Omega),\qquad u\mapsto y\text{ solving \eqref{eq:lame}}.
\end{equation}

The solution operator $S$ of the linear elasticity problem is well known to be a bounded linear operator from $U=L^2(\Omega;\R^2)$ into $H^1_\Gamma(\Omega)\hookrightarrow L^2(\Omega;\R^2)=:Y$; see, e.g., \cite{Braess:2007}.
This immediately implies weak-to-weak continuity and Fr\'echet differentiability with $S'(u)=S$ for all $u\in U$.
Similarly, $S'(u)^*=S^*$ for all $u\in U$, and it is readily checked that $S$ is actually self-adjoint so that $S^*=S$.
As a consequence we have $\ran S'(u)^*=\ran S\hookrightarrow L^\infty(\Omega;\R^2)$.
Indeed, in the case of polygonal domains $\Omega$ this follows from $\ran S\subset H^{3/2}(\Omega;\R^2)$ by \cite[Thm.~2.3]{Ni92},
and in the case of piecewise smooth domains with smooth traction boundary it follows from $\ran S\subset H^{2}(\Omega;\R^2)$ by \cite[Thm.~8]{MaNi10}.
Summarizing, this choice of $S$ satisfies assumptions \ref{enm:weakWeakContinuity}--\ref{enm:linearity}.

\subsection{Multimaterial branched transport}

Here $\Omega$ represents a street or pipe network, i.e., a graph that is typically (but not necessarily) embedded in $\R^2$ or $\R^3$.
If we assume a constant material flux along the graph edges, the description simplifies to the following:
Let $G=(V,E)$ be a directed graph with vertex set $V$ and edge set $E$ representing a transport network.
We endow $E$ with the $\sigma$-algebra generated by the (finitely many) elements of $E$
and define $L^2(E;\R^m)$ as the space of measurable functions from $E$ to $\R^m$
with the inner product $(u,v)_{L^2}=\sum_{e\in E}\ell(e)u(e)\cdot v(e)$ (where $\ell(e)$ denotes the length of $e$).
The space $L^2(V;\R^m)$ is defined analogously.
The function $u\in L^2(E; \R^m)$ describes the flux of $m$ different materials along each edge,
and the operator $S$ is the graph divergence, i.e., the difference of outflux and influx at each vertex,
\begin{equation*}
    S:L^2(E;\R^m)\to L^2(V;\R^m),\quad Su(x)=\sum_{e\in E\text{ incident to }x}u(e)-\sum_{e\in E\text{ emanating from }x}u(e).
\end{equation*}
The target function $z\in L^2(V;\R^m)$ is zero except at the sources and sinks of each material $i$, where $z_i$ takes the value $m_i$ and $-m_i$, respectively.
For this finite-dimensional linear operator $S$, \ref{enm:weakWeakContinuity}--\ref{enm:linearity} are automatically fulfilled.
Note that an infinite-dimensional setting would likewise be possible;
in this case one could choose $\Omega$ as the graph embedding and $S:L^2(\Omega;\R^m)\to H^{-s}(\Omega;\R^m)$ with $s>1$ as the distributional divergence on $\Omega$.
Let us also mention that the regularization $\int_\Omega g(u(x))\,\dd x$ in the above setting reduces to $\sum_{e\in E}\ell(e)g(u(e))$,
which represents the total transport costs.

\section{Numerical solution}\label{sec:semiSmoothNewton}

We now discuss the numerical solution of the regularized system \eqref{eq:regsystem} via a semismooth Newton method.

\subsection{Bloch equation}

As is usual for time-dependent state equations, we avoid a full space-time discretization by following a reduced approach, i.e., we consider in place of \eqref{eq:regsystem} the equation
\begin{equation}\label{eq:my_sys_bloch}
    u_\gamma - H_\gamma(-\calF'(u_\gamma)) = 0.
\end{equation}
Recall that $H_\gamma$ is a superposition operator defined via
\begin{equation*}
    [H_\gamma(p)](x) = h_\gamma (p(x)) \qquad \text{for a.e. }x\in \Omega
\end{equation*}
with $h_\gamma = (\partial g^*)_\gamma$ given by \eqref{eq:my_bloch}.
By \cref{prop:blochAdjointContinuity}, we have $-\calF'(u_\gamma)=S'(u_\gamma)^*(z-S(u_\gamma))\in L^\infty(\Omega;\R^2)$ and, hence, we can consider $H_\gamma:L^r(\Omega;\R^2)\to L^2(\Omega;\R^2)$ for any $r>2$. Since $h_\gamma$ is Lipschitz continuous and piecewise differentiable, semismoothness of $H_\gamma$ follows from \cite[Thm.~3.49]{Ulbrich:2011} with a Newton derivative given by
\begin{equation*}
    [D_N H_\gamma(p) h](x) = D_Nh_\gamma(p(x))h(x) \qquad \text{for a.e. }x\in \Omega
\end{equation*}
and $D_Nh_\gamma$ defined in \eqref{eq:newton_bloch}.

Further, note that $S$ is twice continuously differentiable.
Indeed, this follows by an analogous argument as for Fr\'echet differentiability in the proof of \cref{thm:operatorProperties}:
Using the same notation, the second derivative applied to test directions $\varphi,\psi\in L^2(\Omega;\R^2)$ will be given by $S''(u)(\varphi,\psi)=\mathbf{W}(T)=(\mathbf{W}^1(T),\ldots,\mathbf{W}^J(T))$ with
\begin{equation*}
    \left\{\begin{aligned}
            \tfrac{\dd}{\dd t} \mathbf{W}^j(t)&=B_{u}^{\omega_j}(t)\mathbf{W}^j(t)+B_\varphi^0(t)\delta\mathbf M_\psi^{(\omega_j)}(t)+B_\psi^0(t)\delta\mathbf M_\varphi^{(\omega_j)}(t)\,,\qquad t\in[0,T],\\
            \mathbf{W}^j(0)&=0,
    \end{aligned}\right.
\end{equation*}
where $S'(u)(\varphi)=(\delta\mathbf M_\varphi^{(\omega_1)}(T),\ldots,\delta\mathbf M_\varphi^{(\omega_J)}(T))$ with $\delta\mathbf M_\varphi^{(\omega)}$ satisfying \eqref{eq:BlochFrechetEquation}.
This equation has exactly the same structure as \eqref{eq:BlochFrechetEquation}, and thus the argument for showing
\begin{equation*}
    |S'(\tilde u)(\varphi)-S'(u)(\varphi)-S''(u)(\tilde u-u,\varphi)|_2=\|\varphi\|_{L^2(\Omega;\R^2)}O(\|\tilde u-u\|_{L^2(\Omega;\R^2)}^2)
\end{equation*}
works analogously.
Since $S$ is twice continuously differentiable, we can apply the chain rule, e.g., from \cite[Thm.~3.69]{Ulbrich:2011}, to obtain
\begin{equation*}
    D_N (H_\gamma \circ (-\calF'))(u)\phi = - D_N H_\gamma(-\calF'(u))\calF''(u)\phi
\end{equation*}
for any $\phi\in L^2(\Omega;\R^2)$.
A semismooth Newton step is thus given by $u^{k+1} = u^k +\delta u$, where $\delta u$ is the solution to
\begin{equation}\label{eq:ssn_bloch}
    \left(\Id + D_N H_\gamma(-\calF'(u^k))\calF''(u^k)\right)\delta u = -u^k+ H_\gamma(-\calF'(u^k)),
\end{equation}
which can be obtained, e.g., using a matrix-free Krylov subspace method such as GMRES.

Recall that following \cref{prop:bloch_adjoint} and \cite{Aigner:2015}, $p=-\calF'(u)$ can be evaluated by solving the adjoint equations
\begin{equation}
    \label{eq:adjoint}
    \left\{\begin{aligned}
            -\tfrac{\dd}{dt}{\mathbf{P}}^{(\omega_j)}(t)&= B^\omega_u(t) \mathbf{P}^{(\omega_j)}(t),\qquad t\in[0,T],\\
            \mathbf{P}^{(\omega_j)}(T)&=\mathbf{M}^{(\omega_j)}_u(T)-(\mathbf{M}_d)_j
    \end{aligned}\right.
\end{equation}
for $j=1,\dots,J$ and setting
\begin{equation*}
    \label{eq:gradient}
    \begin{aligned}[t]
        p(t) &= \sum_{j=1}^J\begin{pmatrix}
            \big(\mathbf M_u^{(\omega_j)}(t)\big)_3{\mathbf{P}^{(\omega_j)}_2}(t) - \big(\mathbf M_u^{(\omega_j)}(t)\big)_2{\mathbf{P}^{(\omega_j)}_3}(t)\\
            \big(\mathbf M_u^{(\omega_j)}(t)\big)_3{\mathbf{P}^{(\omega_j)}_1}(t) - \big(\mathbf M_u^{(\omega_j)}(t)\big)_1{\mathbf{P}^{(\omega_j)}_3}(t)
        \end{pmatrix}\\
        &=\sum_{j=1}^J
        \begin{pmatrix}
            \mathbf{M}^{(\omega_j)}_u(t)^T{B_1} {\mathbf{P}^{(\omega_j)}}(t)\\
            \mathbf{M}^{(\omega_j)}_u(t)^T{B_2} {\mathbf{P}^{(\omega_j)}}(t)
        \end{pmatrix}
    \end{aligned}
\end{equation*}
for $t\in[0,T]$, where for the sake of brevity, we have set
\begin{equation*}
    {B_1} :=\begin{pmatrix} 0 & 0 & 0 \\ 0 & 0 & -1 \\ 0 & 1 &0\end{pmatrix},\qquad
    {B_2} :=\begin{pmatrix} 0 & 0 & -1 \\ 0 & 0 & 0 \\ 1 & 0 &0\end{pmatrix}.
\end{equation*}
Similarly, the application of $\calF''(u)\phi$ for given $u,\phi\in L^2(\Omega;\R^2)$ is given by
\begin{equation*}
    \calF''(u)\phi=\sum_{j=1}^J\begin{pmatrix}
        \delta \mathbf{M}_\phi^{(\omega_j)}(t)^T{B_1} {\mathbf{P}^{(\omega_j)}}(t) + \mathbf{M}_u^{(\omega_j)}(t)^T{B_1}\delta {\mathbf{P}^{(\omega_j)}}(t)\\
        \delta \mathbf{M}_\phi^{(\omega_j)}(t)^T{B_2} {\mathbf{P}^{(\omega_j)}}(t) + \mathbf{M}_u^{(\omega_j)}(t)^T{B_2}\delta {\mathbf{P}^{(\omega_j)}}(t)
    \end{pmatrix},
\end{equation*}
where $\delta \mathbf{M}_\phi^{(\omega)}$ (the directional derivative of $\mathbf{M}^{(\omega)}$ with respect to ${u}$) is given by the solution of the linearized state equation \eqref{eq:BlochFrechetEquation}
and $\delta {\mathbf{P}^{(\omega)}}$ (the directional derivative of ${\mathbf{P}^{(\omega)}}$ with respect to ${u}$) is given by the solution of the linearized adjoint equation
\begin{equation*}
    \label{eq:adjoint_lin}
    \left\{\begin{aligned}
            -\tfrac{\dd}{\dd t}{\delta {\mathbf{P}^{(\omega)}}}(t) &= {B_u^\omega}(t) {\delta {\mathbf{P}^{(\omega)}}}(t)+ B_\phi^0(t) {\mathbf{P}^{(\omega)}}(t),\qquad t\in[0,T],\\
            \delta {\mathbf{P}^{(\omega)}}(T) &= \delta \mathbf{M}_\phi^{(\omega)}(T).
    \end{aligned}\right.
\end{equation*}
This characterization can be derived using formal Lagrangian calculus and rigorously justified using the implicit function theorem; see, e.g., \cite[Chapter 1.6]{Hinze2009}.

\bigskip

Since the forward operator $S$ is nonlinear, the problem \eqref{eq:problem_reg} is nonconvex. Hence, convergence of the semismooth Newton method \eqref{eq:ssn_bloch} to a minimizer $u_\gamma$ requires a second-order sufficient (local quadratic growth) condition at $u_\gamma$ for $\gamma>0$ small, which is difficult to verify. Furthermore, we need to deal with the fact that Newton methods converge only locally, with the convergence region shrinking with $\gamma$. For this reason, we perform a continuation in $\gamma$, i.e., we solve \eqref{eq:regsystem} for a sequence $\gamma_1>\gamma_2>\cdots$ of regularization parameters, each time using the result for $\gamma_n$ as initialization for the iteration with $\gamma_{n+1}$. In addition, we include in each step of the semismooth Newton method a line search for $\delta u$ based on the residual norm of the reduced optimality condition \eqref{eq:my_sys_bloch}.
While globalization of nonsmooth Newton methods is a delicate issue that we do not want to address in this work, we remark that this heuristic approach seems to work well for this problem in practice.

\bigskip

We finally address the discretization of \eqref{eq:ssn_bloch}. The Bloch equation is discretized using a Crank--Nicolson method, where the states $\mathbf{M}^{(\omega)}$ are discretized as continuous piecewise linear functions with values $\mathbf{M}_m^{(\omega)}:=\mathbf{M}^{(\omega)}(t_m)$ for discrete time points $t_1,\ldots,t_{N_u}$, and the control ${u}$ is treated as a piecewise constant function, i.e., ${u}=\sum_{m=1}^{N_u} {u_m} \chi_{(t_{m-1},t_m]}(t)$, where $\chi_{(a,b]}$ is the characteristic function of the half-open interval~$(a,b]$. To obtain a consistent scheme, where discretization and optimization commute, the adjoint state $\mathbf{P}^{(\omega)}$ in \eqref{eq:adjoint} is discretized as piecewise constant using an appropriate time-stepping scheme \cite{BecMeiVex_07}, and the linearized state $\delta\mathbf{M}^{(\omega)}$ and the linearized adjoint state $\delta\mathbf{P}^{(\omega)}$ are discretized in the same way as the state and adjoint state, respectively; see~\cite{Aigner:2015}.

\subsection{Linearized elasticity}

For the case of linearized elasticity, we can proceed exactly as in \cite{CIK:2014,CK:2013}. First, note that due to the embedding $H_{\Gamma}^1(\Omega) \hookrightarrow L^p(\Omega;\R^2)$ for $p>2$, the superposition operator $H_{\gamma}$ (for $h_\gamma:=(\partial g^*)_\gamma$ now given by \eqref{eq:myreg_elast}) is again semismooth with Newton derivative $D_NH_\gamma$ (for $D_N h_\gamma$ now given by \eqref{eq:newton_elast}).

To obtain a symmetric Newton system, we reduce \eqref{eq:regsystem} to the state $y_\gamma=S(u_\gamma)$ and the dual variable $p_\gamma$. Since $S$ is a bounded linear operator, we have $S'(u) = S$ and therefore by the definition of $S$ obtain
\begin{equation*}
    \label{eq:redregsystem}
    \left\{\begin{aligned}
            A^*p_{\gamma} &= z - y_{\gamma}, \\
            Ay_{\gamma} &= H_{\gamma}(p_{\gamma}),
    \end{aligned} \right.
\end{equation*}
where $A$ denotes the elliptic linear differential operator arising from the system \eqref{eq:lame} of linearized elasticity.
Consequently, we consider
\begin{equation} \label{eq:Newtonfun}
    F(y,p) \defgl \begin{pmatrix} y - z + A^*p \\ Ay - H_{\gamma}(p) \end{pmatrix} = \begin{pmatrix} 0 \\ 0 \end{pmatrix},
\end{equation}
where $F:Y\times U^*\to Y\times U$. Since the regularized optimal state $y_\gamma$ and the adjoint state $p_\gamma$ are in $H^1_\Gamma(\Omega)$, we may consider $F:H^1_\Gamma(\Omega)\times H^1_\Gamma(\Omega)\to H^1_\Gamma(\Omega)^*\times H^1_\Gamma(\Omega)^*$.
For a semismooth Newton step, we obtain $(\delta y, \delta p)$ by solving
\begin{equation} \label{eq:Newtonstep}
    \begin{pmatrix}
        \Id & A^* \\
        A & -D_NH_{\gamma}(p^k)
    \end{pmatrix}
    \begin{pmatrix}
        \delta y \\
        \delta p
    \end{pmatrix}
    =
    \begin{pmatrix}
        z - y^k -A^*p^k \\
        -Ay^k + H_{\gamma}(p^k)
    \end{pmatrix}
\end{equation}
for given $(y^k,p^k)$, and we set $y^{k+1} = y^k + \delta y$ and $p^{k+1} = p^k + \delta p$.

\bigskip

Due to the linearity of the state equation (and hence convexity of the problem), the convergence of the semismooth Newton method for every $\gamma>0$ to a minimizer of \eqref{eq:problem_reg} can be shown exactly as in \cite{CIK:2014,CK:2013}. As in the case of the Bloch equation, we include a continuation in $\gamma$ as well as a line search based on the residual norm in \eqref{eq:Newtonfun}.

\bigskip

For the discretization, we consider \eqref{eq:Newtonfun} in its weak form
\begin{equation}\label{eq:elasticReducedWeakSystem}
    \begin{pmatrix}
        \int_\Omega2\mu\epsilon(p):\epsilon(\phi)+\lambda\divergence(p)\,\divergence\phi+(y-z)\phi\,\dd x\\
        \int_\Omega2\mu\epsilon(y):\epsilon(\psi)+\lambda\divergence y\,\divergence\psi-h_\gamma(p)\psi\,\dd x
    \end{pmatrix}
    =
    \begin{pmatrix} 0 \\ 0 \end{pmatrix}
\end{equation}
for all $\phi,\psi\in H^1_\Gamma(\Omega)$.
We now discretize the state $y$, the adjoint state $p$, and the test functions $\phi_h,\psi_h$ using piecewise linear finite element functions $y_h,p_h,\phi_h,\psi_h\in V_h$,
where $V_h\subset H^1_\Gamma(\Omega)$ denotes the space of piecewise linear, $\R^2$-valued functions on a uniform triangulation of $\Omega$.
Analogously to \cite{CIK:2014,CK:2013}, we employ exact quadrature for all terms
except for $\int_\Omega h_\gamma(p_h)\psi_h\,\dd x$, which we approximate by $\int_\Omega I_h(h_\gamma(p_h))\psi_h\,\dd x$ for the piecewise linear nodal interpolation operator $I_h$.
Thus, letting $\phi_1,\ldots,\phi_{N_h}$ denote a nodal basis of $V_h$ and introducing the mass and stiffness matrices
\begin{equation*}
    M_h=\left(\int_\Omega\phi_i\cdot \phi_j\,\dd x\right)_{ij},\qquad
    L_h=\left(\int_\Omega\epsilon(\phi_i):\epsilon(\phi_j)\,\dd x\right)_{ij},\qquad
    K_h=\left(\int_\Omega\divergence\phi_i\cdot \divergence\phi_j\,\dd x\right)_{ij},
\end{equation*}
as well as $A_h=2\mu L_h+\lambda K_h$ and the vector $Z_h=(\int_\Omega z\cdot\phi_1\,\dd x,\ldots,\int_\Omega z\cdot\phi_{N_h}\,\dd x)^T$, the discrete version of \eqref{eq:elasticReducedWeakSystem} reads
\begin{equation*}
    \begin{pmatrix}
        A_h^T\mathbf p+M_h\mathbf y-Z_h\\
        A_h\mathbf y-M_hh_\gamma(\mathbf p)
    \end{pmatrix}
    =
    \begin{pmatrix} 0 \\ 0 \end{pmatrix},
\end{equation*}
and \eqref{eq:Newtonstep} becomes
\begin{equation*}
    \begin{pmatrix}
        M_h & A_h^T \\
        A_h & -M_hD_Nh_{\gamma}(\mathbf p^k)
    \end{pmatrix}
    \begin{pmatrix}
        \delta\mathbf y \\
        \delta\mathbf p
    \end{pmatrix}
    =
    \begin{pmatrix}
        Z_h -\mathbf y^k -A_h^T\mathbf p^k \\
        -A_h\mathbf y^k + M_hh_{\gamma}(\mathbf p^k)
    \end{pmatrix},
\end{equation*}
where $\mathbf y=(y_i)_i$ and $\mathbf p=(p_i)_i$ are the nodal values of $y_h$ and $p_h$, and where $h_\gamma(\mathbf p)=(h_\gamma(p_i))_i$ and $D_Nh_\gamma(\mathbf p)=(D_Nh_\gamma(p_i)\delta_{ij})_{ij}$.

\subsection{Multimaterial branched transport}

Here we again follow the same approach as for the Bloch equation, that is, we solve the equation 
\begin{equation*}
    0=u_\gamma - H_\gamma(-\calF'(u_\gamma)) = u_\gamma-H_\gamma(S^*(z-Su_\gamma)).
\end{equation*}
Since $u_\gamma$ and $S^*(z-Su_\gamma)$ are finite-dimensional, the equation is Newton-differentiable with Newton step $u^{k+1}=u^k+\delta u$ for $\delta u$ the solution of
\begin{equation*}
    \left(\Id + D_N H_\gamma(S^*(z-Su^k))S^*S\right)\delta u = -u^k+ H_\gamma(S^*(z-Su^k)).
\end{equation*}
Due to the discrete nature of the domain, this is a more challenging problem than for the Bloch equation and therefore requires a more involved path-following. Specifically, during the outer iteration we adapt the reduction factor for $\gamma$ to keep convergence of the semismooth Newton method within a small number of steps, increasing it if the method requires too many steps (or does not converge at all) and reducing it if the method converges very quickly. In addition, we again employ a line search in $\delta u$.

\section{Numerical examples}\label{sec:examples}

We illustrate the proposed approach for the two model problems described in \cref{sec:stateEq} and the two specific multibang penalties described in \cref{sec:penalty}.
The MATLAB code used to generate these examples is available online \cite{Code}.

\subsection{Bloch equation}

The first example is based on the optimal excitation of isochromats in nuclear magnetic resonance imaging \cite{Spincontrol:15}, where the aim is to shift the magnetization vector $\mathbf M$ at time $T$ from initial alignment with a strong external magnetic field, i.e., $\mathbf M(0) = (0,0,1)^T$, to the saturated state $\mathbf M_d=(1,0,0)^T$ using a radiofrequency pulse $u(t)=(\omega_x(t),\omega_y(t))^T$. To follow the physical setup, we scale the controls as $u(t) = \bar \gamma B_1 \tilde u(t)$, where $\bar \gamma \approx 267.51$ is the gyromagnetic ratio (in MHz per Tesla) and $B_1 = 10^{-2}$ is the strength of the modulated magnetic field (in milliTesla); the figures always show the unscaled control $\tilde u$. The control cost parameter (which in this setting can be interpreted as a penalty on the specific absorption rate of the radio energy) is set to $\alpha=10^{-1}$.
In all examples, the Bloch equation is discretized with $N_u=1000$ time intervals; the implementation of the discrete (linearized) Bloch and adjoint equations is taken from \cite{rfcontrol}.
The semismooth Newton iteration is then applied and terminated if the relative or absolute norm of the residual in the optimality condition drops below $10^{-7}$ or if $500$ iterations are exceeded. The Newton step is solved via GMRES without restarts and without preconditioning, which is terminated if the relative residual drops below $10^{-10}$ or if $1000$ iterations are exceeded.
The continuation in the Moreau--Yosida regularization is started with $\gamma_0=10^2$ and reduced by a factor of $1/2$ until $\gamma_{\min}=10^{-10}$ is reached or the semismooth Newton iteration fails to converge.
We remark that in a practical implementation, these strict fixed tolerances should be replaced as in inexact Newton methods by adaptive criteria based on residuals in the outer loops.

We begin with a single isochromat with $\omega = 10^{-2}\bar\gamma$. \Cref{fig:Bloch1} shows the resulting optimal control $\tilde u$ and magnetization evolution $\mathbf{M}^{(\omega)}(t)$ for $M=3$ equally spaced radially distributed desired control values with magnitude $\omega_0=1$ and phases
$\theta_1=-\pi$, $\theta_2=-\pi/3$, $\theta_3=\pi/3$, which are marked by colored dashed lines.
At any time $t\in[0,T]$, the optimal control $\tilde u(t)=(\omega_x(t),\omega_y(t))$ can be seen to only take values from $\calM$ as desired.
(For easier visual comprehension, $\tilde u(t)$ is plotted as a continuous curve so that a jump from one value in $\calM$ to another is shown as a connecting line.)
Indeed, most of the time we have $\tilde u=\bar u_0=0$, periodically intermitted by short time intervals where $\tilde u$ takes the values $\bar u_1,\bar u_2,\bar u_3\in\calM$ in a periodically rotating order.
Each of these time intervals coincides in the state trajectory with a change in $M_z$, while the $M_z$ component of $\mathbf M^{(\omega)}$ stays constant during $\tilde u=0$.
The final magnetization $\mathbf M^{(\omega)}(T)$ shows a very close attainment of the target $\mathbf M_d$. The situation is very similar for $M=6$ with $\omega_0=1$ and $\theta \in \{-\pi,-2\pi/3,-\pi/3,0,\pi/3,2\pi/3\}$; see \Cref{fig:Bloch2}. In both cases, all nonzero desired control values are made use of equally.

\definecolor{mycolor1}{rgb}{0.00000,0.44700,0.74100}%
\definecolor{mycolor2}{rgb}{0.85000,0.32500,0.09800}%
\definecolor{mycolor3}{rgb}{0.92900,0.69400,0.12500}%
\definecolor{mycolor4}{rgb}{0.49400,0.18400,0.55600}%
\definecolor{mycolor5}{rgb}{0.46600,0.67400,0.18800}%
\definecolor{mycolor6}{rgb}{0.30100,0.74500,0.93300}%
\definecolor{mycolor7}{rgb}{0.63500,0.07800,0.18400}%

\begin{figure}
    \centering
    \begin{subfigure}[t]{0.45\linewidth}
        \centering
        \begin{tikzpicture}[trim axis left,scale=0.9]

\begin{axis}[%
width=\linewidth,
scale only axis,
xmin=0,
xmax=8,
tick align=outside,
xlabel={$t$},
ymin=-1,
ymax=1,
ylabel={$u_1$},
zmin=-1,
zmax=1,
zlabel={$u_2$},
view={-37.5}{30},
axis background/.style={fill=white},
axis x line*=bottom,
axis y line*=left,
axis z line*=left,
xmajorgrids,
ymajorgrids,
zmajorgrids
]
\addplot3 [color=mycolor1, line width=1pt]
 table[row sep=crcr] {%
0	-0	0\\
0.504	-0	0\\
0.511	-1	-0\\
0.567	-1	-0\\
0.574	0	0\\
1.288	-0	0\\
1.295	0.5	-0.866025403784438\\
1.351	0.5	-0.866025403784438\\
1.358	0	-0\\
2.065	-0	0\\
2.072	0.389203148246079	0.67411962722797\\
2.079	0.5	0.866025403784438\\
2.135	0.5	0.866025403784438\\
2.142	-0	0\\
2.849	-0	-0\\
2.856	-1	-0\\
2.912	-1	-0\\
2.926	-0	0\\
3.633	-0	-0\\
3.64	0.5	-0.866025403784438\\
3.696	0.5	-0.866025403784438\\
3.703	-0	0\\
4.417	0	-0\\
4.424	0.5	0.866025403784438\\
4.48	0.5	0.866025403784438\\
4.487	-0	0\\
5.201	0	-0\\
5.208	-1	-0\\
5.264	-1	-0\\
5.271	0	0\\
5.985	-0	0\\
5.992	0.5	-0.866025403784438\\
6.048	0.5	-0.866025403784438\\
6.055	-0	0\\
6.762	-0	-0\\
6.769	0.439094307431502	0.760533649785632\\
6.776	0.5	0.866025403784438\\
6.832	0.500000000000001	0.866025403784441\\
6.839	-0	-0\\
6.993	0	0\\
};
 \addplot3 [color=mycolor2, dashed, line width=1.5pt]
 table[row sep=crcr] {%
0	-1	-0\\
6.993	-1	-0\\
};
 \addplot3 [color=mycolor3, dashed, line width=1.5pt]
 table[row sep=crcr] {%
0	0.5	-0.866025403784438\\
6.993	0.5	-0.866025403784438\\
};
 \addplot3 [color=mycolor4, dashed, line width=1.5pt]
 table[row sep=crcr] {%
0	0.5	0.866025403784438\\
6.993	0.5	0.866025403784438\\
};
 \end{axis}
\end{tikzpicture}%
        \caption{control $\tilde u(t)$}\label{fig:Bloch1:u}
    \end{subfigure}
    \hfill
    \begin{subfigure}[t]{0.45\linewidth}
        \centering
        \input{bloch_d3_state.tikz}
        \caption{state $\mathbf{M}_u^{(\omega)}(t)$}\label{fig:Bloch1:M}
    \end{subfigure}
    \caption{Control and state for the Bloch model problem: $M=3$}\label{fig:Bloch1}
\end{figure} 
\begin{figure}
    \centering
    \begin{subfigure}[t]{0.45\linewidth}
        \centering
        \begin{tikzpicture}[trim axis left,scale=0.9]

\begin{axis}[%
width=\linewidth,
scale only axis,
xmin=0,
xmax=8,
tick align=outside,
xlabel={$t$},
ymin=-1,
ymax=1,
ylabel={$u_1$},
zmin=-1,
zmax=1,
zlabel={$u_2$},
view={-37.5}{30},
axis background/.style={fill=white},
axis x line*=bottom,
axis y line*=left,
axis z line*=left,
xmajorgrids,
ymajorgrids,
zmajorgrids
]
\addplot3 [color=mycolor1, line width=1pt]
 table[row sep=crcr] {%
0	-8.88178419700125e-16	0\\
0.126	0	0\\
0.133	-0.5	0.866025403784438\\
0.161	-0.5	0.866025403784436\\
0.168	-0	-0\\
0.518	-0	0\\
0.525	-1	-0\\
0.546	-1	-0\\
0.553	-0.968339124108316	-0\\
0.56	0	-0\\
0.91	-0	-0\\
0.917	-0.5	-0.866025403784438\\
0.938	-0.5	-0.866025403784438\\
0.952	-0	-0\\
1.302	-0	-0\\
1.309	0.5	-0.866025403784438\\
1.33	0.5	-0.866025403784438\\
1.337	-0	0\\
1.694	-0	0\\
1.701	1	0\\
1.722	1	0\\
1.729	0	0\\
2.086	0	0\\
2.093	0.499999999999999	0.866025403784439\\
2.114	0.499999999999999	0.866025403784439\\
2.121	0	0\\
2.478	-0	1.77635683940025e-15\\
2.485	-0.5	0.866025403784438\\
2.506	-0.5	0.866025403784438\\
2.513	0	-0\\
2.863	-0	-0\\
2.87	-0.731267125527537	-0\\
2.877	-1	-0\\
2.898	-1	-0\\
2.905	-0	0\\
3.255	-0	-0\\
3.262	-0.500000000000001	-0.86602540378443\\
3.29	-0.5	-0.866025403784438\\
3.297	0	-0\\
3.647	-0	0\\
3.654	0.499999999999999	-0.866025403784437\\
3.682	0.5	-0.866025403784438\\
3.689	-0	0\\
4.039	0	0\\
4.046	1.00000000000001	0\\
4.074	1	0\\
4.081	0	0\\
4.431	0	0\\
4.438	0.499999999999999	0.866025403784439\\
4.466	0.499999999999999	0.86602540378444\\
4.473	0	-0\\
4.823	0	-0\\
4.83	-0.5	0.866025403784438\\
4.858	-0.5	0.866025403784436\\
4.865	0	-0\\
5.215	-0	-0\\
5.222	-1	-0\\
5.25	-0.999999999999999	-0\\
5.257	-0	0\\
5.607	-0	0\\
5.614	-0.5	-0.866025403784438\\
5.642	-0.499999999999997	-0.866025403784436\\
5.649	0	0\\
5.999	0	-0\\
6.006	0.5	-0.866025403784438\\
6.027	0.5	-0.866025403784438\\
6.034	0.324766587893095	-0.562512230831794\\
6.041	0	0\\
6.391	-0	0\\
6.398	1	0\\
6.419	1	0\\
6.426	-0	0\\
6.783	-0	-0\\
6.79	0.499999999999999	0.866025403784439\\
6.811	0.499999999999999	0.866025403784439\\
6.818	-0	0\\
6.993	0	0\\
};
 \addplot3 [color=mycolor2, dashed, line width=1.5pt]
 table[row sep=crcr] {%
0	-1	-0\\
6.993	-1	-0\\
};
 \addplot3 [color=mycolor3, dashed, line width=1.5pt]
 table[row sep=crcr] {%
0	-0.5	-0.866025403784438\\
6.993	-0.5	-0.866025403784438\\
};
 \addplot3 [color=mycolor4, dashed, line width=1.5pt]
 table[row sep=crcr] {%
0	0.5	-0.866025403784438\\
6.993	0.5	-0.866025403784438\\
};
 \addplot3 [color=mycolor5, dashed, line width=1.5pt]
 table[row sep=crcr] {%
0	1	0\\
6.993	1	0\\
};
 \addplot3 [color=mycolor6, dashed, line width=1.5pt]
 table[row sep=crcr] {%
0	0.499999999999999	0.866025403784439\\
6.993	0.499999999999999	0.866025403784439\\
};
 \addplot3 [color=mycolor7, dashed, line width=1.5pt]
 table[row sep=crcr] {%
0	-0.5	0.866025403784438\\
6.993	-0.5	0.866025403784438\\
};
 \end{axis}
\end{tikzpicture}%
        \caption{control $\tilde u(t)$}\label{fig:Bloch2:u}
    \end{subfigure}
    \hfill
    \begin{subfigure}[t]{0.45\linewidth}
        \centering
        \input{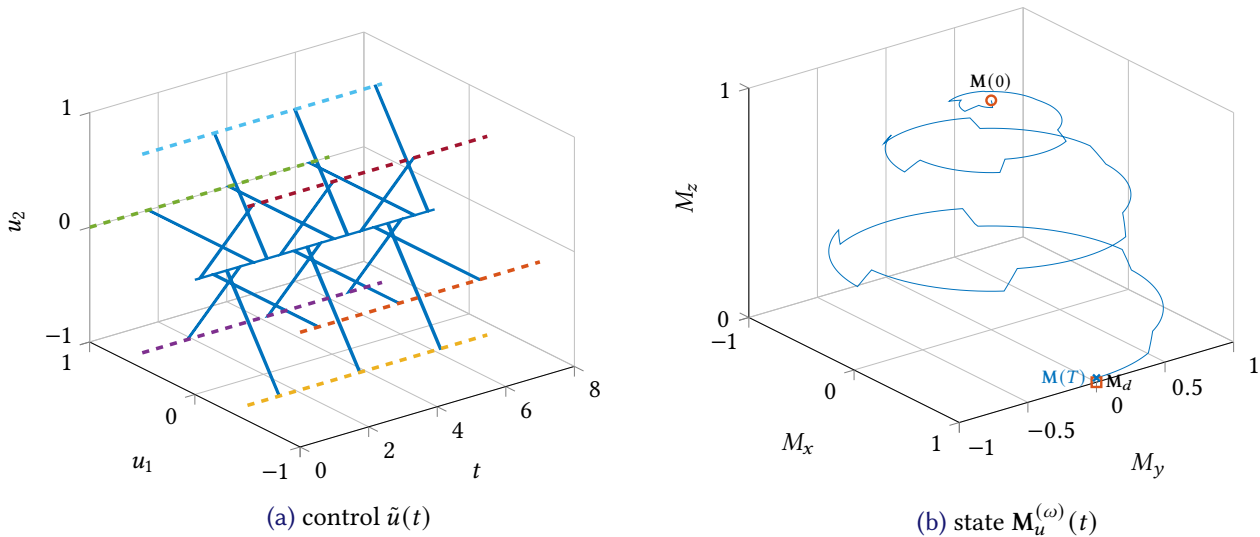}
        \caption{state $\mathbf{M}_u^{(\omega)}(t)$}\label{fig:Bloch2:M}
    \end{subfigure}
    \caption{Control and state for the Bloch model problem: $M=6$}\label{fig:Bloch2}
\end{figure} 

\Cref{tab:iterBloch} summarizes the convergence behavior for the case $M=3$. For a representative selection of values of $\gamma$, it shows the number of semismooth Newton iterations,
the average number of GMRES iterations needed to solve a Newton step, the number of times a step of length less than $1$ was taken, and the number of nodes $t_m$ for which $u_\gamma(t_m)\notin\calM$.
For moderate values of $\gamma$ (approximately $\gamma>10^{-6}$ in this case), very few iterations of both the semismooth Newton method and the inner GMRES method are required to reach the solution.
If $\gamma$ is decreased further, however, the problem starts becoming significantly more difficult, requiring an increasing number of Newton iterations that, in addition, require a damping to lead
to a decrease of the residual. These damped steps typically are taken after a few initial full steps and continue until the region of superlinear convergence is reached, after which the iteration
terminates after a small number of full steps. The average number of GMRES steps, however, remains small. For $\gamma < 9.313\cdot10^{-8}$, the maximal number of semismooth Newton iterations is
no longer sufficient to reach the given tolerance. However, the final row of the table demonstrates that already for $\gamma \approx 10^{-5}$ (where the convergence is still fast), the control is
already almost perfectly multibang.

\begin{table}[b]
    \caption{Convergence behavior for the example in \cref{fig:Bloch1}: number of semi-smooth Newton steps, average number of GMRES iterations to solve a Newton step, number of times a line search was required, and number of nodes $t_m$ with $u_\gamma(t_m)\notin\calM$}
    \label{tab:iterBloch}
    \centering
    \resizebox{\linewidth}{!}{%
        \begin{tabular}{lrrrrrrrrr}
            \toprule
            $\gamma$       & \num[round-precision=0]{1e2}  & \num[round-precision=0]{1.6e0} & \num[round-precision=0]{2.0e-1} & \num[round-precision=0]{1.2e-2} & \num[round-precision=0]{1.5e-3} & \num[round-precision=0]{1.9e-4} & \num[round-precision=0]{1.2e-5} & \num[round-precision=0]{1.5e-6} & \num[round-precision=0]{9.3e-8}\\
            \midrule
            \# SSN         & \num{3}    & \num{3}     & \num{4}      & \num{5}      & \num{5}      & \num{5}      & \num{4}      & \num{100}    & \num{101}\\
            avg. \# GMRES  & \num{3}    & \num{7}     & \num{7.5}    & \num{7.4}    & \num{7.8}    & \num{8.2}    & \num{3.75}   & \num{3.14}   & \num{4.3}  \\
            \# line search & \num{0}    & \num{0}     & \num{0}      & \num{0}      & \num{0}      & \num{0}      & \num{0}      & \num{98}     & \num{99} \\
            \# not MB     & \num{1000} & \num{1000}  & \num{862}    & \num{376}    & \num{191}    & \num{44}     & \num{3}      & \num{3}      & \num{3} \\
            \bottomrule
        \end{tabular}
    }
\end{table}

We now consider the simultaneous control of $J=4$ isochromats with $\omega = 10^{-2}\bar\gamma\cdot(1,2,3,4)$. \Cref{fig:Bloch3} shows the result if the same target $\mathbf M_d=(1,0,0)^T$ is specified for all isochromats. Again, we have a very close attainment of the target,
and again the control is zero most of the time, intermitted by regularly spaced intervals in which nonzero control values from $\calM$ are used.
This time, not all nonzero values from $\calM$ occur, but just $\bar u_2$ and $\bar u_3$ (indicated by the red and turquoise dashed lines).
In addition there are five time points at which control values outside $\calM$ are adopted, visible in the graph as short spikes emanating from $\tilde u=0$.
(Note, though, that these values still show the desired angles, merely at smaller than desired magnitudes.)
This may be due to the fact that in this example, the Newton method has failed to converge already for $\gamma<2\cdot 10^{-6}$.
In the more realistic case where only a single isochromat -- in this case $j=3$ -- is supposed to be excited (i.e., $\mathbf M_d=(1,0,0)^T$ for $\mathbf M^{(\omega_3)}$ and $\mathbf M_d=(0,0,1)^T$ otherwise), we again obtain a pure multibang control (see \cref{fig:Bloch4}).

\begin{figure}
    \centering
    \begin{subfigure}[t]{0.45\linewidth}
        \centering
        \begin{tikzpicture}[trim axis left,scale=0.9]

\begin{axis}[%
width=\linewidth,
scale only axis,
xmin=0,
xmax=8,
tick align=outside,
xlabel={$t$},
ymin=-1,
ymax=1,
ylabel={$u_1$},
zmin=-1,
zmax=1,
zlabel={$u_2$},
view={-37.5}{30},
axis background/.style={fill=white},
axis x line*=bottom,
axis y line*=left,
axis z line*=left,
xmajorgrids,
ymajorgrids,
zmajorgrids
]
\addplot3 [color=mycolor1, line width=1pt]
 table[row sep=crcr] {%
0	-0.5	0.866025403784438\\
0.0490000000000004	-0.5	0.866025403784438\\
0.056	-0.390135461161352	0.675734440565775\\
0.0629999999999997	-0	-0\\
0.651	-0	0\\
0.665	-0.237138774803105	-0\\
0.672	-0	-0\\
0.77	0	0\\
0.777	-0.10146386341886	-0.175740566573695\\
0.784	-0.164113039195898	-0.284252122071837\\
0.791	0	-0\\
1.428	-0	0\\
1.435	0.155726530747316	-0.269726263340789\\
1.442	0.124406663474223	-0.215478661937476\\
1.449	-0	0\\
1.547	-0	0\\
1.554	0.2517482849423	0\\
1.561	0.0485040870096149	0\\
1.568	0	0\\
2.17	0	0\\
2.177	0.00902189733162384	0.0156263845590425\\
2.184	0.499999999999999	0.866025403784439\\
2.289	0.499999999999999	0.866025403784439\\
2.296	-0	-0\\
2.303	-0	-0\\
2.31	-0.5	0.866025403784438\\
2.401	-0.5	0.866025403784438\\
2.408	-0.335528775544908	0.581152886645155\\
2.415	0	-0\\
3.129	0	0\\
3.136	-0.0084788176348729	-0.0146857429317109\\
3.143	0	0\\
4.522	0	-0\\
4.536	0.499999999999999	0.866025403784439\\
4.641	0.499999999999999	0.866025403784439\\
4.648	-0	-0\\
4.655	-0.5	0.866025403784438\\
4.753	-0.5	0.866025403784438\\
4.76	-0.349794162256198	0.605861261218727\\
4.767	0	-0\\
6.874	-0	0\\
6.881	0.358708208735582	0.62130084262205\\
6.888	0.499999999999999	0.866025403784439\\
6.986	0.499999999999999	0.866025403784439\\
6.993	-0	-0\\
};
 \addplot3 [color=mycolor2, dashed, line width=1.5pt]
 table[row sep=crcr] {%
0	-1	-0\\
6.993	-1	-0\\
};
 \addplot3 [color=mycolor3, dashed, line width=1.5pt]
 table[row sep=crcr] {%
0	-0.5	-0.866025403784438\\
6.993	-0.5	-0.866025403784438\\
};
 \addplot3 [color=mycolor4, dashed, line width=1.5pt]
 table[row sep=crcr] {%
0	0.5	-0.866025403784438\\
6.993	0.5	-0.866025403784438\\
};
 \addplot3 [color=mycolor5, dashed, line width=1.5pt]
 table[row sep=crcr] {%
0	1	0\\
6.993	1	0\\
};
 \addplot3 [color=mycolor6, dashed, line width=1.5pt]
 table[row sep=crcr] {%
0	0.499999999999999	0.866025403784439\\
6.993	0.499999999999999	0.866025403784439\\
};
 \addplot3 [color=mycolor7, dashed, line width=1.5pt]
 table[row sep=crcr] {%
0	-0.5	0.866025403784438\\
6.993	-0.5	0.866025403784438\\
};
 \end{axis}
\end{tikzpicture}%
        \caption{control $\tilde u(t)$}\label{fig:Bloch3:u}
    \end{subfigure}
    \hfill
    \begin{subfigure}[t]{0.45\linewidth}
        \centering
        \input{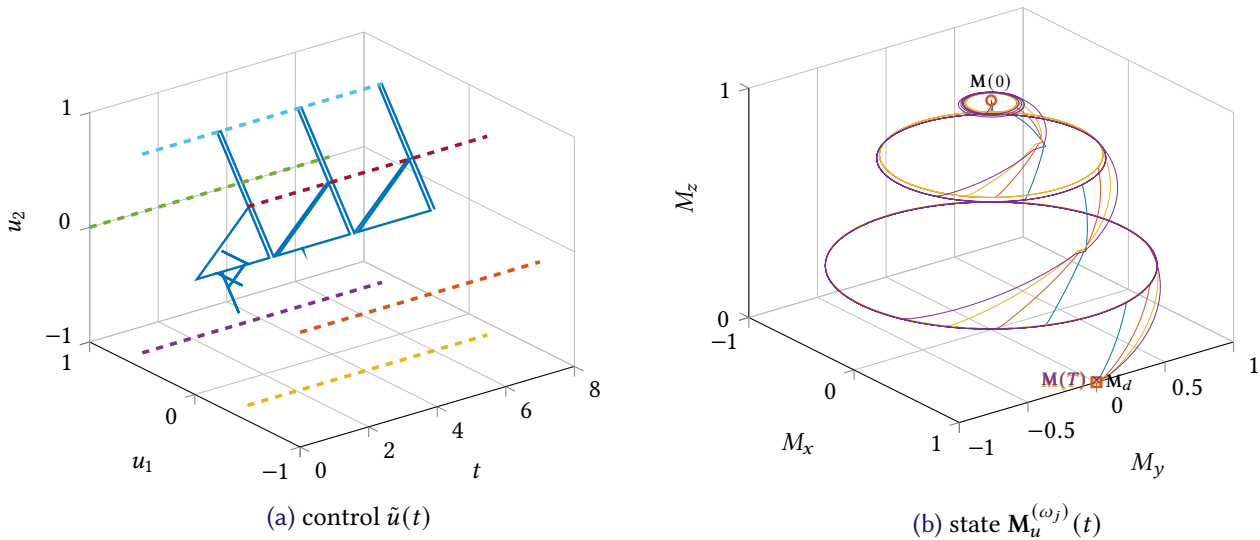}
        \caption{state $\mathbf{M}_u^{(\omega_j)}(t)$}\label{fig:Bloch3:M}
    \end{subfigure}
    \caption{Control and state for the Bloch model problem: $M=6$, $J=4$}\label{fig:Bloch3}
\end{figure}
\begin{figure}
    \centering
    \begin{subfigure}[t]{0.45\linewidth}
        \centering
        \begin{tikzpicture}[trim axis left,scale=0.9]

\begin{axis}[%
width=\linewidth,
scale only axis,
xmin=0,
xmax=8,
tick align=outside,
xlabel={$t$},
ymin=-1,
ymax=1,
ylabel={$u_1$},
zmin=-1,
zmax=1,
zlabel={$u_2$},
view={-37.5}{30},
axis background/.style={fill=white},
axis x line*=bottom,
axis y line*=left,
axis z line*=left,
xmajorgrids,
ymajorgrids,
zmajorgrids
]
\addplot3 [color=mycolor1, line width=1pt]
 table[row sep=crcr] {%
0	-0	0\\
0.00699999999999967	0	0\\
0.0140000000000002	-0.5	0.866025403784438\\
0.0209999999999999	0	0\\
0.133	0	0\\
0.14	-1	-0\\
0.147	-1	-0\\
0.154	-0	0\\
0.266	0	-0\\
0.273	-0.5	-0.866025403784438\\
0.28	-0.301760939368978	-0.522665278726782\\
0.287	-0	0\\
0.399	0	-0\\
0.406	0.5	-0.866025403784438\\
0.413	0	-0\\
0.525	-0	0\\
0.532	1	0\\
0.539	1	0\\
0.546	-0	0\\
0.658	-0	-0\\
0.665	0.499999999999999	0.866025403784439\\
0.679	0	0\\
0.784	-0	0\\
0.791	-0.124654847034217	0.215908528472991\\
0.798	-0.5	0.866025403784438\\
0.805	-0	-0\\
0.917	-0	-0\\
0.924	-1	-0\\
0.931	-1	-0\\
0.938	-0	0\\
1.05	-0	-0\\
1.057	-0.5	-0.866025403784438\\
1.064	-0.0110260488915124	-0.0190976768868376\\
1.071	-0	-0\\
1.176	0	0\\
1.183	0.194960546499996	-0.337681572009386\\
1.19	0.5	-0.866025403784438\\
1.197	-0	-0\\
1.309	0	0\\
1.316	1	0\\
1.323	1	0\\
1.33	-0	0\\
1.442	0	0\\
1.449	0.499999999999999	0.866025403784439\\
1.456	-0	-0\\
1.568	0	0\\
1.575	-0.353729988054552	0.612678311471217\\
1.582	-0.5	0.866025403784438\\
1.589	-0	-0\\
1.701	-0	-0\\
1.708	-1	-0\\
1.715	-0.8746925557469	-0\\
1.722	-0	0\\
1.834	0	0\\
1.841	-0.5	-0.866025403784438\\
1.848	-0	-0\\
1.96	-0	-0\\
1.967	0.5	-0.866025403784438\\
1.974	0.5	-0.866025403784438\\
1.981	0	-0\\
2.093	0	0\\
2.1	1	0\\
2.107	0.820583989450242	0\\
2.114	0	0\\
2.219	-0	-0\\
2.226	0.020826304043859	0.0360722167378409\\
2.233	0.499999999999999	0.866025403784439\\
2.24	-0	0\\
2.352	-0	-0\\
2.359	-0.5	0.866025403784438\\
2.366	-0.5	0.866025403784438\\
2.373	-0	0\\
2.485	-0	-0\\
2.492	-1	-0\\
2.506	-0	-0\\
2.611	-0	0\\
2.618	-0.107193807473524	-0.185665120800901\\
2.625	-0.5	-0.866025403784438\\
2.632	-0	-0\\
2.744	0	0\\
2.751	0.5	-0.866025403784438\\
2.758	0.5	-0.866025403784438\\
2.765	0	-0\\
2.877	0	0\\
2.884	1	0\\
2.891	0.199063844241419	0\\
2.898	-0	0\\
3.003	-0	-0\\
3.01	0.210361500514878	0.364356806848197\\
3.017	0.499999999999999	0.866025403784439\\
3.024	-0	0\\
3.136	0	0\\
3.143	-0.5	0.866025403784438\\
3.15	-0.5	0.866025403784438\\
3.157	-0	0\\
3.269	-0	-0\\
3.276	-1	-0\\
3.283	-0	-0\\
3.395	-0	-0\\
3.402	-0.322207954766399	-0.558080548258258\\
3.409	-0.5	-0.866025403784438\\
3.416	-0	0\\
3.528	-0	0\\
3.535	0.5	-0.866025403784438\\
3.542	0.5	-0.866025403784438\\
3.549	0	-0\\
3.661	0	0\\
3.668	1	0\\
3.675	-0	0\\
3.787	0	0\\
3.794	0.367232102962029	0.636064660500599\\
3.801	0.499999999999999	0.866025403784439\\
3.808	0	-0\\
3.92	0	0\\
3.927	-0.5	0.866025403784438\\
3.934	-0.396769649352417	0.687225191579675\\
3.941	-0	-0\\
4.053	0	0\\
4.06	-1	-0\\
4.067	-0	-0\\
4.179	-0	0\\
4.186	-0.5	-0.866025403784438\\
4.193	-0.5	-0.866025403784438\\
4.2	-0	0\\
4.312	-0	-0\\
4.319	0.5	-0.866025403784438\\
4.326	0.308492440046711	-0.534324579911799\\
4.333	-0	0\\
4.438	0	0\\
4.445	0.099686261469107	0\\
4.452	1	0\\
4.459	0	0\\
4.571	-0	0\\
4.578	0.499999999999999	0.866025403784439\\
4.585	0.499999999999999	0.866025403784439\\
4.592	0	0\\
4.704	-0	0\\
4.711	-0.5	0.866025403784438\\
4.725	0	0\\
4.83	0	0\\
4.844	-1	-0\\
4.851	-0	-0\\
4.963	0	-0\\
4.97	-0.5	-0.866025403784438\\
4.977	-0.5	-0.866025403784438\\
4.984	-0	-0\\
5.096	0	-0\\
5.103	0.5	-0.866025403784438\\
5.11	0	-0\\
5.222	0	0\\
5.236	1	0\\
5.243	-0	0\\
5.355	0	-0\\
5.362	0.499999999999999	0.866025403784439\\
5.369	0.499999999999999	0.866025403784439\\
5.376	0	-0\\
5.488	-0	0\\
5.495	-0.5	0.866025403784438\\
5.502	-0	0\\
5.614	-0	-0\\
5.621	-0.847035175587634	-0\\
5.628	-1	-0\\
5.635	0	0\\
5.747	-0	-0\\
5.754	-0.5	-0.866025403784438\\
5.761	-0.5	-0.866025403784438\\
5.768	-0	-0\\
5.88	-0	0\\
5.887	0.5	-0.866025403784438\\
5.894	-0	0\\
6.006	0	0\\
6.013	0.872181644123008	0\\
6.02	1	0\\
6.027	0	0\\
6.139	-0	0\\
6.146	0.499999999999999	0.866025403784439\\
6.153	0.396574145549109	0.686886569059273\\
6.16	-0	0\\
6.272	0	-0\\
6.279	-0.5	0.866025403784438\\
6.286	-0	0\\
6.398	0	0\\
6.405	-1	-0\\
6.412	-1	-0\\
6.419	-0	0\\
6.531	-0	-0\\
6.538	-0.5	-0.866025403784438\\
6.545	-0.17042063016699	-0.295177190107132\\
6.552	0	-0\\
6.657	0	0\\
6.664	0.0155491830246897	-0.0269319750149499\\
6.671	0.5	-0.866025403784438\\
6.678	0	-0\\
6.79	0	0\\
6.797	1	0\\
6.804	1	0\\
6.811	-0	0\\
6.923	0	0\\
6.93	0.499999999999999	0.866025403784439\\
6.937	0.0852671501867031	0.147687036339977\\
6.944	-0	0\\
6.993	0	0\\
};
 \addplot3 [color=mycolor2, dashed, line width=1.5pt]
 table[row sep=crcr] {%
0	-1	-0\\
6.993	-1	-0\\
};
 \addplot3 [color=mycolor3, dashed, line width=1.5pt]
 table[row sep=crcr] {%
0	-0.5	-0.866025403784438\\
6.993	-0.5	-0.866025403784438\\
};
 \addplot3 [color=mycolor4, dashed, line width=1.5pt]
 table[row sep=crcr] {%
0	0.5	-0.866025403784438\\
6.993	0.5	-0.866025403784438\\
};
 \addplot3 [color=mycolor5, dashed, line width=1.5pt]
 table[row sep=crcr] {%
0	1	0\\
6.993	1	0\\
};
 \addplot3 [color=mycolor6, dashed, line width=1.5pt]
 table[row sep=crcr] {%
0	0.499999999999999	0.866025403784439\\
6.993	0.499999999999999	0.866025403784439\\
};
 \addplot3 [color=mycolor7, dashed, line width=1.5pt]
 table[row sep=crcr] {%
0	-0.5	0.866025403784438\\
6.993	-0.5	0.866025403784438\\
};
 \end{axis}
\end{tikzpicture}%
        \caption{control $\tilde u(t)$}\label{fig:Bloch4:u}
    \end{subfigure}
    \hfill
    \begin{subfigure}[t]{0.45\linewidth}
        \centering
        \input{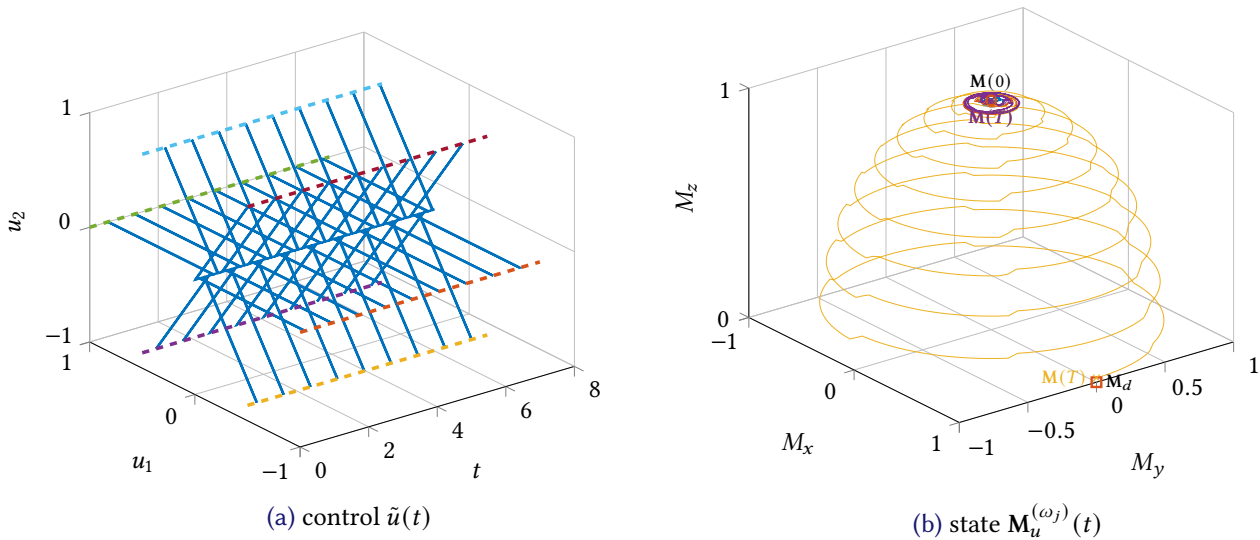}
        \caption{state $\mathbf{M}_u^{(\omega_j)}(t)$}\label{fig:Bloch4:M}
    \end{subfigure}
    \caption{Control and state for the Bloch model problem: $M=6$, $J=4$, $\mathbf M_d = (0,0,1)$ for $j=3$ and $\mathbf M_0$ otherwise}\label{fig:Bloch4}
\end{figure} 

\subsection{Linearized elasticity}

We now address the behavior in the context of optimal control of elliptic partial differential equations for the model equations of two-dimensional linearized elasticity.
Here, we choose $\Omega=[0,1]\times[0,2]$ and $\Gamma=[0,1]\times\{0\}$, which models an elastic beam clamped at the bottom. The
Lam\'e parameters are set to $\mu = \frac{E}{2(1+\nu)}$ and $\lambda = \frac{E\nu}{(1+\nu)(1-2\nu)}$ for the elastic modulus $E=20$ and the Poisson ratio $\nu =0.3$. We use a uniform structured mesh with $129$ vertices in each direction.
Since the state equation is linear, we use a direct solver for the Newton step. The Newton iteration is terminated if the active sets (i.e., the case distinctions in the definition of the Moreau--Yosida regularization) for each node coincide for two consecutive iterations, or if $50$ iterations are exceeded. The continuation in the regularization parameter $\gamma$ is performed as for the Bloch equation.

\Cref{fig:elasticExamples} shows the results for six different choices of target, multibang penalty, and control cost parameter.
In \crefrange{fig:elastic:1}{fig:elastic:4}, the target displacement $z(x)=R(x-(\frac12,1)^T)-x$ corresponds to a rotation $R\in SO(2)$ of the solid around its center.
\Cref{fig:elastic:1,fig:elastic:2} use the penalty from \cref{sec:multibang:bloch} for $\alpha=10^{-3}$, while \cref{fig:elastic:3,fig:elastic:4} use the penalty from \cref{sec:multibang:elast} for $\alpha=10^{-5}$ and $\alpha=10^{-3}$, respectively.
In all cases, the obtained control makes use of all control values in $\calM$ and aligns them with the rotation.
Furthermore, the center of the force vortex always lies slightly to the top right of the rotation center of the target state;
this allows a stronger overall rightward force in the lower part of the solid to compensate for the clamping at the bottom.
Note that unlike the case of (additional) gradient regularization of the control, small patches or sharp corners of the domains with homogeneous force are allowed.
\begin{figure}[t] 
    \begin{subfigure}[t]{0.15\linewidth}
        \centering
        \begin{tikzpicture}[x=\linewidth,y=\linewidth]
\path[use as bounding box] (-.5,-.55) rectangle (.5,.5);
\begin{axis}[%
    width=0.66\linewidth,
    height=0.66\linewidth,
    at={(-.33\linewidth,-.33\linewidth)},
    scale only axis,
    axis on top,
    xmin=-128,
    xmax=128,
    ymin=-128,
    ymax=128,
    axis line style={draw=none},
    ticks=none
    ]
    \addplot  graphics [xmin=-128,xmax=128,ymin=-128,ymax=128] {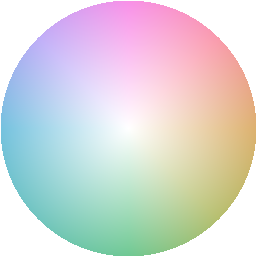};
\end{axis}
\begin{polaraxis}[%
    width=0.66\linewidth,
    height=0.66\linewidth,
    at={(-.33\linewidth,-.33\linewidth)},
    ymin=0,
    ymax=256,
    scale only axis,
    xtick={60,180,300},
    ytick={0,256},
    xticklabels={$\frac{\pi}3$,$\pi$,$-\frac{\pi}3$},
    yticklabels={$0$,$\sqrt{8}$},
    yticklabel style = {font=\tiny}
    ]
\end{polaraxis}
\end{tikzpicture}%
\\
        \begin{tikzpicture}[x=\linewidth,y=\linewidth]
\path[use as bounding box] (0,0) -- (1,0) -- (1,2.1) -- (0,2.1) -- cycle;    

\begin{axis}[%
width=\linewidth,
height=2\linewidth,
scale only axis,
axis on top,
clip=false,
xmin=0,
xmax=1,
ymin=0,
ymax=2,
axis background/.style={fill=white},
ticks=none
]
\addplot [forget plot] graphics [xmin=0, xmax=1, ymin=0, ymax=2] {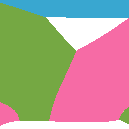};
\addplot[-Latex, 
color=black, 
point meta={sqrt((\thisrow{u})^2+(\thisrow{v})^2)}, 
point meta min=0, 
quiver={u=1.5*\thisrow{u}, v=1.5*\thisrow{v}, 
    every arrow/.append style={-{Latex[scale={1/1000*\pgfplotspointmetatransformed}]}}}]
 table[row sep=crcr] {%
x	y	u	v\\
0.0546875	0.109375	0.05	0.0866025403784439\\
0.0546875	0.359375	0.05	0.0866025403784439\\
0.0546875	0.609375	0.05	-0.0866025403784439\\
0.0546875	0.859375	0.05	-0.0866025403784439\\
0.0546875	1.109375	0.05	-0.0866025403784439\\
0.0546875	1.359375	0.05	-0.0866025403784439\\
0.0546875	1.609375	0.05	-0.0866025403784439\\
0.0546875	1.859375	0.05	-0.0866025403784439\\
0.1796875	0.359375	0.05	-0.0866025403784439\\
0.1796875	0.609375	0.05	-0.0866025403784439\\
0.1796875	0.859375	0.05	-0.0866025403784439\\
0.1796875	1.109375	0.05	-0.0866025403784439\\
0.1796875	1.359375	0.05	-0.0866025403784439\\
0.1796875	1.609375	0.05	-0.0866025403784439\\
0.1796875	1.859375	-0.1	-1.22464679914735e-17\\
0.3046875	0.359375	0.05	-0.0866025403784439\\
0.3046875	0.609375	0.05	-0.0866025403784439\\
0.3046875	0.859375	0.05	-0.0866025403784439\\
0.3046875	1.109375	0.05	-0.0866025403784439\\
0.3046875	1.359375	0.05	-0.0866025403784439\\
0.3046875	1.609375	0.05	-0.0866025403784439\\
0.3046875	1.859375	-0.1	-1.22464679914735e-17\\
0.4296875	0.109375	0.05	0.0866025403784439\\
0.4296875	0.359375	0.05	0.0866025403784439\\
0.4296875	0.609375	0.05	-0.0866025403784439\\
0.4296875	0.859375	0.05	-0.0866025403784439\\
0.4296875	1.109375	0.05	-0.0866025403784439\\
0.4296875	1.359375	0.05	-0.0866025403784439\\
0.4296875	1.859375	-0.1	-1.22464679914735e-17\\
0.5546875	0.109375	0.05	0.0866025403784439\\
0.5546875	0.359375	0.05	0.0866025403784439\\
0.5546875	0.609375	0.05	0.0866025403784439\\
0.5546875	0.859375	0.05	0.0866025403784439\\
0.5546875	1.109375	0.05	-0.0866025403784439\\
0.5546875	1.859375	-0.1	-1.22464679914735e-17\\
0.6796875	0.109375	0.05	0.0866025403784439\\
0.6796875	0.359375	0.05	0.0866025403784439\\
0.6796875	0.609375	0.05	0.0866025403784439\\
0.6796875	0.859375	0.05	0.0866025403784439\\
0.6796875	1.109375	0.05	0.0866025403784439\\
0.6796875	1.859375	-0.1	-1.22464679914735e-17\\
0.8046875	0.109375	0.05	0.0866025403784439\\
0.8046875	0.359375	0.05	0.0866025403784439\\
0.8046875	0.609375	0.05	0.0866025403784439\\
0.8046875	0.859375	0.05	0.0866025403784439\\
0.8046875	1.109375	0.05	0.0866025403784439\\
0.8046875	1.359375	0.05	0.0866025403784439\\
0.8046875	1.859375	-0.1	-1.22464679914735e-17\\
0.9296875	0.359375	0.05	0.0866025403784439\\
0.9296875	0.609375	0.05	0.0866025403784439\\
0.9296875	0.859375	0.05	0.0866025403784439\\
0.9296875	1.109375	0.05	0.0866025403784439\\
0.9296875	1.359375	0.05	0.0866025403784439\\
0.9296875	1.609375	0.05	0.0866025403784439\\
0.9296875	1.859375	-0.1	-1.22464679914735e-17\\
};
\end{axis}
\end{tikzpicture}%
\\
        \input{elast_rad3_deform.tikz}
        \caption{radial, $d=3$, $\alpha = 10^{-3}$}\label{fig:elastic:1}
    \end{subfigure}
    \hfill
    \begin{subfigure}[t]{0.15\linewidth}
        \centering
        \begin{tikzpicture}[x=\linewidth,y=\linewidth]
\path[use as bounding box] (-.5,-.55) rectangle (.5,.5);
\begin{axis}[%
    width=0.66\linewidth,
    height=0.66\linewidth,
    at={(-.33\linewidth,-.33\linewidth)},
    scale only axis,
    axis on top,
    xmin=-128,
    xmax=128,
    ymin=-128,
    ymax=128,
    axis line style={draw=none},
    ticks=none
    ]
    \addplot  graphics [xmin=-128,xmax=128,ymin=-128,ymax=128] {wheel.png};
\end{axis}
\begin{polaraxis}[%
    width=0.66\linewidth,
    height=0.66\linewidth,
    at={(-.33\linewidth,-.33\linewidth)},
    ymin=0,
    ymax=256,
    scale only axis,
    xtick={36,108,180,252,324},
    ytick={0,256},
    xticklabels={$\frac{\pi}5$,$\frac{2\pi}5$,$\pi$,$-\frac{2\pi}5$,$-\frac{\pi}5$},
    yticklabels={$0$,$\sqrt{8}$},
    yticklabel style = {font=\tiny}
    ]
\end{polaraxis}
\end{tikzpicture}%
\\
        \begin{tikzpicture}[x=\linewidth,y=\linewidth]
\path[use as bounding box] (0,0) -- (1,0) -- (1,2.1) -- (0,2.1) -- cycle;    

\begin{axis}[%
width=\linewidth,
height=2\linewidth,
scale only axis,
axis on top,
clip=false,
xmin=0,
xmax=1,
ymin=0,
ymax=2,
axis background/.style={fill=white},
ticks=none
]
\addplot [forget plot] graphics [xmin=0, xmax=1, ymin=0, ymax=2] {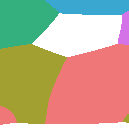};
\addplot[-Latex, 
color=black, 
point meta={sqrt((\thisrow{u})^2+(\thisrow{v})^2)}, 
point meta min=0, 
quiver={u=1.5*\thisrow{u}, v=1.5*\thisrow{v}, 
    every arrow/.append style={-{Latex[scale={1/1000*\pgfplotspointmetatransformed}]}}}]
 table[row sep=crcr] {%
x	y	u	v\\
0.0546875	0.109375	0.0809016994374947	0.0587785252292473\\
0.0546875	0.359375	0.0809016994374947	-0.0587785252292473\\
0.0546875	0.609375	0.0809016994374947	-0.0587785252292473\\
0.0546875	0.859375	0.0809016994374947	-0.0587785252292473\\
0.0546875	1.109375	0.0809016994374947	-0.0587785252292473\\
0.0546875	1.359375	-0.0309016994374947	-0.0951056516295154\\
0.0546875	1.609375	-0.0309016994374947	-0.0951056516295154\\
0.0546875	1.859375	-0.0309016994374947	-0.0951056516295154\\
0.1796875	0.109375	0.0809016994374947	-0.0587785252292473\\
0.1796875	0.359375	0.0809016994374947	-0.0587785252292473\\
0.1796875	0.609375	0.0809016994374947	-0.0587785252292473\\
0.1796875	0.859375	0.0809016994374947	-0.0587785252292473\\
0.1796875	1.109375	0.0809016994374947	-0.0587785252292473\\
0.1796875	1.359375	-0.0309016994374947	-0.0951056516295154\\
0.1796875	1.609375	-0.0309016994374947	-0.0951056516295154\\
0.1796875	1.859375	-0.0309016994374947	-0.0951056516295154\\
0.3046875	0.109375	0.0809016994374947	-0.0587785252292473\\
0.3046875	0.359375	0.0809016994374947	-0.0587785252292473\\
0.3046875	0.609375	0.0809016994374947	-0.0587785252292473\\
0.3046875	0.859375	0.0809016994374947	-0.0587785252292473\\
0.3046875	1.109375	0.0809016994374947	-0.0587785252292473\\
0.3046875	1.609375	-0.0309016994374947	-0.0951056516295154\\
0.3046875	1.859375	-0.0309016994374947	-0.0951056516295154\\
0.4296875	0.109375	0.0809016994374947	0.0587785252292473\\
0.4296875	0.359375	0.0809016994374947	0.0587785252292473\\
0.4296875	0.609375	0.0809016994374947	0.0587785252292473\\
0.4296875	0.859375	0.0809016994374947	-0.0587785252292473\\
0.4296875	1.109375	0.0809016994374947	-0.0587785252292473\\
0.4296875	1.859375	-0.0309016994374947	-0.0951056516295154\\
0.5546875	0.109375	0.0809016994374947	0.0587785252292473\\
0.5546875	0.359375	0.0809016994374947	0.0587785252292473\\
0.5546875	0.609375	0.0809016994374947	0.0587785252292473\\
0.5546875	0.859375	0.0809016994374947	0.0587785252292473\\
0.5546875	1.109375	0.0809016994374947	0.0587785252292473\\
0.5546875	1.859375	-0.1	-1.22464679914735e-17\\
0.6796875	0.109375	0.0809016994374947	0.0587785252292473\\
0.6796875	0.359375	0.0809016994374947	0.0587785252292473\\
0.6796875	0.609375	0.0809016994374947	0.0587785252292473\\
0.6796875	0.859375	0.0809016994374947	0.0587785252292473\\
0.6796875	1.109375	0.0809016994374947	0.0587785252292473\\
0.6796875	1.859375	-0.1	-1.22464679914735e-17\\
0.8046875	0.109375	0.0809016994374947	0.0587785252292473\\
0.8046875	0.359375	0.0809016994374947	0.0587785252292473\\
0.8046875	0.609375	0.0809016994374947	0.0587785252292473\\
0.8046875	0.859375	0.0809016994374947	0.0587785252292473\\
0.8046875	1.109375	0.0809016994374947	0.0587785252292473\\
0.8046875	1.859375	-0.1	-1.22464679914735e-17\\
0.9296875	0.109375	0.0809016994374947	0.0587785252292473\\
0.9296875	0.359375	0.0809016994374947	0.0587785252292473\\
0.9296875	0.609375	0.0809016994374947	0.0587785252292473\\
0.9296875	0.859375	0.0809016994374947	0.0587785252292473\\
0.9296875	1.109375	0.0809016994374947	0.0587785252292473\\
0.9296875	1.359375	-0.0309016994374947	0.0951056516295154\\
0.9296875	1.859375	-0.1	-1.22464679914735e-17\\
};
\end{axis}
\end{tikzpicture}%
\\
        \input{elast_rad5_deform.tikz}
        \caption{radial, $d=5$, $\alpha = 10^{-3}$}\label{fig:elastic:2}
    \end{subfigure}
    \hfill
    \begin{subfigure}[t]{0.15\linewidth}
        \centering
        \begin{tikzpicture}[x=\linewidth,y=\linewidth]
\path[use as bounding box] (-.5,-.55) rectangle (.5,.5);
\begin{axis}[%
    width=0.66\linewidth,
    height=0.66\linewidth,
    at={(-.33\linewidth,-.33\linewidth)},
    scale only axis,
    axis on top,
    xmin=-128,
    xmax=128,
    ymin=-128,
    ymax=128,
    axis line style={draw=none},
    ticks=none
    ]
    \addplot  graphics [xmin=-128,xmax=128,ymin=-128,ymax=128] {wheel.png};
\end{axis}
\begin{polaraxis}[%
    width=0.66\linewidth,
    height=0.66\linewidth,
    at={(-.33\linewidth,-.33\linewidth)},
    ymin=0,
    ymax=256,
    scale only axis,
    xtick={45,135,225,315},
    ytick={128,256},
    xticklabels={$\frac{\pi}4$,$\frac{3\pi}{4}\!$,$\frac{-3\pi}4\!$,$\frac{-\pi}4$},
    yticklabels={$\sqrt{2}$,$\sqrt{8}$},
    yticklabel style = {font=\tiny}
    ]
\end{polaraxis}
\end{tikzpicture}%
\\
        \begin{tikzpicture}[x=\linewidth,y=\linewidth]
\path[use as bounding box] (0,0) -- (1,0) -- (1,2.1) -- (0,2.1) -- cycle;    

\begin{axis}[%
width=\linewidth,
height=2\linewidth,
scale only axis,
axis on top,
clip=false,
xmin=0,
xmax=1,
ymin=0,
ymax=2,
axis background/.style={fill=white},
ticks=none
]
\addplot [forget plot] graphics [xmin=0, xmax=1, ymin=0, ymax=2] {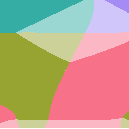};
\addplot[-Latex, 
color=black, 
point meta={sqrt((\thisrow{u})^2+(\thisrow{v})^2)}, 
point meta min=0, 
quiver={u=1.5*\thisrow{u}, v=1.5*\thisrow{v}, 
    every arrow/.append style={-{Latex[scale={1/1000*\pgfplotspointmetatransformed}]}}}]
 table[row sep=crcr] {%
x	y	u	v\\
0.0546875	0.109375	0.0353553390593274	0.0353553390593274\\
0.0546875	0.359375	0.0707106781186548	-0.0707106781186548\\
0.0546875	0.609375	0.0707106781186548	-0.0707106781186548\\
0.0546875	0.859375	0.0707106781186548	-0.0707106781186548\\
0.0546875	1.109375	0.0707106781186548	-0.0707106781186548\\
0.0546875	1.359375	0.0707106781186548	-0.0707106781186548\\
0.0546875	1.609375	-0.0707106781186548	-0.0707106781186548\\
0.0546875	1.859375	-0.0707106781186548	-0.0707106781186548\\
0.1796875	0.109375	0.0353553390593274	-0.0353553390593274\\
0.1796875	0.359375	0.0707106781186548	-0.0707106781186548\\
0.1796875	0.609375	0.0707106781186548	-0.0707106781186548\\
0.1796875	0.859375	0.0707106781186548	-0.0707106781186548\\
0.1796875	1.109375	0.0707106781186548	-0.0707106781186548\\
0.1796875	1.359375	0.0707106781186548	-0.0707106781186548\\
0.1796875	1.609375	-0.0707106781186548	-0.0707106781186548\\
0.1796875	1.859375	-0.0707106781186548	-0.0707106781186548\\
0.3046875	0.109375	0.0353553390593274	-0.0353553390593274\\
0.3046875	0.359375	0.0707106781186548	-0.0707106781186548\\
0.3046875	0.609375	0.0707106781186548	-0.0707106781186548\\
0.3046875	0.859375	0.0707106781186548	-0.0707106781186548\\
0.3046875	1.109375	0.0707106781186548	-0.0707106781186548\\
0.3046875	1.359375	0.0353553390593274	-0.0353553390593274\\
0.3046875	1.609375	-0.0353553390593274	-0.0353553390593274\\
0.3046875	1.859375	-0.0707106781186548	-0.0707106781186548\\
0.4296875	0.109375	0.0353553390593274	0.0353553390593274\\
0.4296875	0.359375	0.0707106781186548	0.0707106781186548\\
0.4296875	0.609375	0.0707106781186548	-0.0707106781186548\\
0.4296875	0.859375	0.0707106781186548	-0.0707106781186548\\
0.4296875	1.109375	0.0707106781186548	-0.0707106781186548\\
0.4296875	1.359375	0.0353553390593274	-0.0353553390593274\\
0.4296875	1.609375	-0.0353553390593274	-0.0353553390593274\\
0.4296875	1.859375	-0.0707106781186548	-0.0707106781186548\\
0.5546875	0.109375	0.0353553390593274	0.0353553390593274\\
0.5546875	0.359375	0.0707106781186548	0.0707106781186548\\
0.5546875	0.609375	0.0707106781186548	0.0707106781186548\\
0.5546875	0.859375	0.0707106781186548	0.0707106781186548\\
0.5546875	1.109375	0.0353553390593274	-0.0353553390593274\\
0.5546875	1.359375	0.0353553390593274	-0.0353553390593274\\
0.5546875	1.609375	-0.0353553390593274	-0.0353553390593274\\
0.5546875	1.859375	-0.0353553390593274	-0.0353553390593274\\
0.6796875	0.109375	0.0353553390593274	0.0353553390593274\\
0.6796875	0.359375	0.0707106781186548	0.0707106781186548\\
0.6796875	0.609375	0.0707106781186548	0.0707106781186548\\
0.6796875	0.859375	0.0707106781186548	0.0707106781186548\\
0.6796875	1.109375	0.0707106781186548	0.0707106781186548\\
0.6796875	1.359375	0.0353553390593274	0.0353553390593274\\
0.6796875	1.609375	-0.0353553390593274	-0.0353553390593274\\
0.6796875	1.859375	-0.0353553390593274	-0.0353553390593274\\
0.8046875	0.109375	0.0353553390593274	0.0353553390593274\\
0.8046875	0.359375	0.0707106781186548	0.0707106781186548\\
0.8046875	0.609375	0.0707106781186548	0.0707106781186548\\
0.8046875	0.859375	0.0707106781186548	0.0707106781186548\\
0.8046875	1.109375	0.0707106781186548	0.0707106781186548\\
0.8046875	1.359375	0.0353553390593274	0.0353553390593274\\
0.8046875	1.609375	-0.0353553390593274	0.0353553390593274\\
0.8046875	1.859375	-0.0353553390593274	0.0353553390593274\\
0.9296875	0.109375	0.0353553390593274	0.0353553390593274\\
0.9296875	0.359375	0.0707106781186548	0.0707106781186548\\
0.9296875	0.609375	0.0707106781186548	0.0707106781186548\\
0.9296875	0.859375	0.0707106781186548	0.0707106781186548\\
0.9296875	1.109375	0.0707106781186548	0.0707106781186548\\
0.9296875	1.359375	0.0353553390593274	0.0353553390593274\\
0.9296875	1.609375	-0.0353553390593274	0.0353553390593274\\
0.9296875	1.859375	-0.0353553390593274	0.0353553390593274\\
};
\end{axis}
\end{tikzpicture}%
\\
        \input{elast_box-3_deform.tikz}
        \caption{concentric, $\alpha = 10^{-3}$}\label{fig:elastic:3}
    \end{subfigure}
    \hfill
    \begin{subfigure}[t]{0.15\linewidth}
        \centering
        \begin{tikzpicture}[x=\linewidth,y=\linewidth]
\path[use as bounding box] (-.5,-.55) rectangle (.5,.5);
\begin{axis}[%
    width=0.66\linewidth,
    height=0.66\linewidth,
    at={(-.33\linewidth,-.33\linewidth)},
    scale only axis,
    axis on top,
    xmin=-128,
    xmax=128,
    ymin=-128,
    ymax=128,
    axis line style={draw=none},
    ticks=none
    ]
    \addplot  graphics [xmin=-128,xmax=128,ymin=-128,ymax=128] {wheel.png};
\end{axis}
\begin{polaraxis}[%
    width=0.66\linewidth,
    height=0.66\linewidth,
    at={(-.33\linewidth,-.33\linewidth)},
    ymin=0,
    ymax=256,
    scale only axis,
    xtick={45,135,225,315},
    ytick={128,256},
    xticklabels={$\frac{\pi}4$,$\frac{3\pi}{4}\!$,$\frac{-3\pi}4\!$,$\frac{-\pi}4$},
    yticklabels={$\sqrt{2}$,$\sqrt{8}$},
    yticklabel style = {font=\tiny}
    ]
\end{polaraxis}
\end{tikzpicture}%
\\
        \begin{tikzpicture}[x=\linewidth,y=\linewidth]
\path[use as bounding box] (0,0) -- (1,0) -- (1,2.1) -- (0,2.1) -- cycle;    

\begin{axis}[%
width=\linewidth,
height=2\linewidth,
scale only axis,
axis on top,
clip=false,
xmin=0,
xmax=1,
ymin=0,
ymax=2,
axis background/.style={fill=white},
ticks=none
]
\addplot [forget plot] graphics [xmin=0, xmax=1, ymin=0, ymax=2] {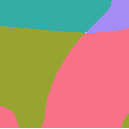};
\addplot[-Latex, 
color=black, 
point meta={sqrt((\thisrow{u})^2+(\thisrow{v})^2)}, 
point meta min=0, 
quiver={u=1.5*\thisrow{u}, v=1.5*\thisrow{v}, 
    every arrow/.append style={-{Latex[scale={1/1000*\pgfplotspointmetatransformed}]}}}]
 table[row sep=crcr] {%
x	y	u	v\\
0.0546875	0.109375	0.0707106781186548	0.0707106781186548\\
0.0546875	0.359375	0.0707106781186548	-0.0707106781186548\\
0.0546875	0.609375	0.0707106781186548	-0.0707106781186548\\
0.0546875	0.859375	0.0707106781186548	-0.0707106781186548\\
0.0546875	1.109375	0.0707106781186548	-0.0707106781186548\\
0.0546875	1.359375	0.0707106781186548	-0.0707106781186548\\
0.0546875	1.609375	-0.0707106781186548	-0.0707106781186548\\
0.0546875	1.859375	-0.0707106781186548	-0.0707106781186548\\
0.1796875	0.109375	0.0707106781186548	-0.0707106781186548\\
0.1796875	0.359375	0.0707106781186548	-0.0707106781186548\\
0.1796875	0.609375	0.0707106781186548	-0.0707106781186548\\
0.1796875	0.859375	0.0707106781186548	-0.0707106781186548\\
0.1796875	1.109375	0.0707106781186548	-0.0707106781186548\\
0.1796875	1.359375	0.0707106781186548	-0.0707106781186548\\
0.1796875	1.609375	-0.0707106781186548	-0.0707106781186548\\
0.1796875	1.859375	-0.0707106781186548	-0.0707106781186548\\
0.3046875	0.109375	0.0707106781186548	-0.0707106781186548\\
0.3046875	0.359375	0.0707106781186548	-0.0707106781186548\\
0.3046875	0.609375	0.0707106781186548	-0.0707106781186548\\
0.3046875	0.859375	0.0707106781186548	-0.0707106781186548\\
0.3046875	1.109375	0.0707106781186548	-0.0707106781186548\\
0.3046875	1.359375	0.0707106781186548	-0.0707106781186548\\
0.3046875	1.609375	-0.0707106781186548	-0.0707106781186548\\
0.3046875	1.859375	-0.0707106781186548	-0.0707106781186548\\
0.4296875	0.109375	0.0707106781186548	0.0707106781186548\\
0.4296875	0.359375	0.0707106781186548	0.0707106781186548\\
0.4296875	0.609375	0.0707106781186548	0.0707106781186548\\
0.4296875	0.859375	0.0707106781186548	-0.0707106781186548\\
0.4296875	1.109375	0.0707106781186548	-0.0707106781186548\\
0.4296875	1.359375	0.0707106781186548	-0.0707106781186548\\
0.4296875	1.609375	-0.0707106781186548	-0.0707106781186548\\
0.4296875	1.859375	-0.0707106781186548	-0.0707106781186548\\
0.5546875	0.109375	0.0707106781186548	0.0707106781186548\\
0.5546875	0.359375	0.0707106781186548	0.0707106781186548\\
0.5546875	0.609375	0.0707106781186548	0.0707106781186548\\
0.5546875	0.859375	0.0707106781186548	0.0707106781186548\\
0.5546875	1.109375	0.0707106781186548	0.0707106781186548\\
0.5546875	1.359375	0.0707106781186548	-0.0707106781186548\\
0.5546875	1.609375	-0.0707106781186548	-0.0707106781186548\\
0.5546875	1.859375	-0.0707106781186548	-0.0707106781186548\\
0.6796875	0.109375	0.0707106781186548	0.0707106781186548\\
0.6796875	0.359375	0.0707106781186548	0.0707106781186548\\
0.6796875	0.609375	0.0707106781186548	0.0707106781186548\\
0.6796875	0.859375	0.0707106781186548	0.0707106781186548\\
0.6796875	1.109375	0.0707106781186548	0.0707106781186548\\
0.6796875	1.359375	0.0707106781186548	0.0707106781186548\\
0.6796875	1.609375	-0.0707106781186548	-0.0707106781186548\\
0.6796875	1.859375	-0.0707106781186548	-0.0707106781186548\\
0.8046875	0.109375	0.0707106781186548	0.0707106781186548\\
0.8046875	0.359375	0.0707106781186548	0.0707106781186548\\
0.8046875	0.609375	0.0707106781186548	0.0707106781186548\\
0.8046875	0.859375	0.0707106781186548	0.0707106781186548\\
0.8046875	1.109375	0.0707106781186548	0.0707106781186548\\
0.8046875	1.359375	0.0707106781186548	0.0707106781186548\\
0.8046875	1.609375	-0.0707106781186548	0.0707106781186548\\
0.8046875	1.859375	-0.0707106781186548	-0.0707106781186548\\
0.9296875	0.109375	0.0707106781186548	0.0707106781186548\\
0.9296875	0.359375	0.0707106781186548	0.0707106781186548\\
0.9296875	0.609375	0.0707106781186548	0.0707106781186548\\
0.9296875	0.859375	0.0707106781186548	0.0707106781186548\\
0.9296875	1.109375	0.0707106781186548	0.0707106781186548\\
0.9296875	1.359375	0.0707106781186548	0.0707106781186548\\
0.9296875	1.609375	-0.0707106781186548	0.0707106781186548\\
0.9296875	1.859375	-0.0707106781186548	0.0707106781186548\\
};
\end{axis}
\end{tikzpicture}%
\\
        \input{elast_box-5_deform.tikz}
        \caption{concentric, $\alpha = 10^{-5}$}\label{fig:elastic:4}
    \end{subfigure}
    \hfill
    \begin{subfigure}[t]{0.15\linewidth}
        \centering
        \begin{tikzpicture}[x=\linewidth,y=\linewidth]
\path[use as bounding box] (-.5,-.55) rectangle (.5,.5);
\begin{axis}[%
    width=0.66\linewidth,
    height=0.66\linewidth,
    at={(-.33\linewidth,-.33\linewidth)},
    scale only axis,
    axis on top,
    xmin=-128,
    xmax=128,
    ymin=-128,
    ymax=128,
    axis line style={draw=none},
    ticks=none
    ]
    \addplot  graphics [xmin=-128,xmax=128,ymin=-128,ymax=128] {wheel.png};
\end{axis}
\begin{polaraxis}[%
    width=0.66\linewidth,
    height=0.66\linewidth,
    at={(-.33\linewidth,-.33\linewidth)},
    ymin=0,
    ymax=256,
    scale only axis,
    xtick={45,135,225,315},
    ytick={128,256},
    xticklabels={$\frac{\pi}4$,$\frac{3\pi}{4}\!$,$\frac{-3\pi}4\!$,$\frac{-\pi}4$},
    yticklabels={$\sqrt{2}$,$\sqrt{8}$},
    yticklabel style = {font=\tiny}
    ]
\end{polaraxis}
\end{tikzpicture}%
\\
        \begin{tikzpicture}[x=\linewidth,y=\linewidth]
\path[use as bounding box] (0,0) -- (1,0) -- (1,2.1) -- (0,2.1) -- cycle;    

\begin{axis}[%
width=\linewidth,
height=2\linewidth,
scale only axis,
axis on top,
clip=false,
xmin=0,
xmax=1,
ymin=0,
ymax=2,
axis background/.style={fill=white},
ticks=none
]
\addplot [forget plot] graphics [xmin=0, xmax=1, ymin=0, ymax=2] {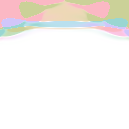};
\addplot[-Latex, 
color=black, 
point meta={sqrt((\thisrow{u})^2+(\thisrow{v})^2)}, 
point meta min=0, 
point meta max=0.1,
quiver={u=1.5*\thisrow{u}, v=1.5*\thisrow{v}, 
    every arrow/.append style={-{Latex[scale={1/1000*\pgfplotspointmetatransformed}]}}}]
 table[row sep=crcr] {%
x	y	u	v\\
0.0546875	1.359375	-0.000283327586604607	0.000995558151701052\\
0.0546875	1.609375	-0.0353553390593274	0.0353553390593274\\
0.0546875	1.859375	0.0353553390593274	0.0353553390593274\\
0.1796875	1.359375	0.000263743367003969	-0.000410714927635264\\
0.1796875	1.609375	-0.0353553390593274	-0.0263140159852521\\
0.1796875	1.859375	0.0353553390593274	-0.0353553390593274\\
0.3046875	1.609375	-0.0353553390593274	-0.0273054273281385\\
0.3046875	1.859375	0.0353553390593274	-0.0353553390593274\\
0.4296875	1.609375	-0.0353553390593274	-0.00560682635689849\\
0.4296875	1.859375	0.0353553390593274	-0.00167050569451843\\
0.5546875	1.609375	-0.0353553390593274	0.00406141696236012\\
0.5546875	1.859375	0.0353553390593274	0.00274745639862095\\
0.6796875	1.609375	-0.0353553390593274	0.0230661354414778\\
0.6796875	1.859375	0.0353553390593274	0.0353553390593274\\
0.8046875	1.359375	0.000175031022078319	0.000271420370308644\\
0.8046875	1.609375	-0.0353553390593274	0.00715909474315652\\
0.8046875	1.859375	0.0353553390593274	0.0353553390593274\\
0.9296875	1.359375	-0.000186682252276792	-0.000701918731637707\\
0.9296875	1.609375	-0.0353553390593274	-0.0353553390593274\\
0.9296875	1.859375	0.0353553390593274	-0.0353553390593274\\
};
\end{axis}
\end{tikzpicture}%
\\
        \input{elast_box_c2_deform.tikz}
        \caption{concentric, $\alpha = 10^{-5}$}\label{fig:elastic:5}
    \end{subfigure}
    \hfill
    \begin{subfigure}[t]{0.15\linewidth}
        \centering
        \begin{tikzpicture}[x=\linewidth,y=\linewidth]
\path[use as bounding box] (-.5,-.55) rectangle (.5,.5);
\begin{axis}[%
    width=0.66\linewidth,
    height=0.66\linewidth,
    at={(-.33\linewidth,-.33\linewidth)},
    scale only axis,
    axis on top,
    xmin=-128,
    xmax=128,
    ymin=-128,
    ymax=128,
    axis line style={draw=none},
    ticks=none
    ]
    \addplot  graphics [xmin=-128,xmax=128,ymin=-128,ymax=128] {wheel.png};
\end{axis}
\begin{polaraxis}[%
    width=0.66\linewidth,
    height=0.66\linewidth,
    at={(-.33\linewidth,-.33\linewidth)},
    ymin=0,
    ymax=256,
    scale only axis,
    xtick={45,135,225,315},
    ytick={128,256},
    xticklabels={$\frac{\pi}4$,$\frac{3\pi}{4}\!$,$\frac{-3\pi}4\!$,$\frac{-\pi}4$},
    yticklabels={$\sqrt{2}$,$\sqrt{8}$},
    yticklabel style = {font=\tiny}
    ]
\end{polaraxis}
\end{tikzpicture}%
\\
        \begin{tikzpicture}[x=\linewidth,y=\linewidth]
\path[use as bounding box] (0,0) -- (1,0) -- (1,2.1) -- (0,2.1) -- cycle;    

\begin{axis}[%
width=\linewidth,
height=2\linewidth,
scale only axis,
axis on top,
clip=false,
xmin=0,
xmax=1,
ymin=0,
ymax=2,
axis background/.style={fill=white},
ticks=none
]
\addplot [forget plot] graphics [xmin=0, xmax=1, ymin=0, ymax=2] {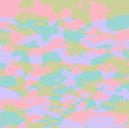};
\addplot[-Latex, 
color=black, 
point meta={sqrt((\thisrow{u})^2+(\thisrow{v})^2)}, 
point meta min=0, 
point meta max=0.1,
quiver={u=1.5*\thisrow{u}, v=1.5*\thisrow{v}, 
    every arrow/.append style={-{Latex[scale={1/1000*\pgfplotspointmetatransformed}]}}}]
 table[row sep=crcr] {%
x	y	u	v\\
0.0546875	0.109375	-0.0353553390593274	-0.0353553390593274\\
0.0546875	0.359375	0.0353553390593274	-0.0353553390593274\\
0.0546875	0.609375	-0.0353553390593274	0.0353553390593274\\
0.0546875	0.859375	0.0353553390593274	0.0353553390593274\\
0.0546875	1.109375	-0.0353553390593274	0.0353553390593274\\
0.0546875	1.359375	0.0353553390593274	-0.0353553390593274\\
0.0546875	1.609375	0.0353553390593274	0.0353553390593274\\
0.0546875	1.859375	0.0353553390593274	0.0353553390593274\\
0.1796875	0.109375	-0.0353553390593274	-0.0353553390593274\\
0.1796875	0.359375	0.0353553390593274	0.0353553390593274\\
0.1796875	0.609375	0.0353553390593274	-0.0353553390593274\\
0.1796875	0.859375	-0.0353553390593274	0.0353553390593274\\
0.1796875	1.109375	-0.0353553390593274	-0.0353553390593274\\
0.1796875	1.359375	-0.0353553390593274	0.0353553390593274\\
0.1796875	1.609375	-0.0353553390593274	-0.0353553390593274\\
0.1796875	1.859375	0.0353553390593274	0.0353553390593274\\
0.3046875	0.109375	-0.0353553390593274	-0.0353553390593274\\
0.3046875	0.359375	-0.0353553390593274	0.0353553390593274\\
0.3046875	0.609375	-0.0353553390593274	-0.0353553390593274\\
0.3046875	0.859375	0.0353553390593274	0.0353553390593274\\
0.3046875	1.109375	-0.0353553390593274	0.0353553390593274\\
0.3046875	1.359375	-0.0353553390593274	0.0353553390593274\\
0.3046875	1.609375	0.0353553390593274	-0.0353553390593274\\
0.3046875	1.859375	0.0353553390593274	0.0353553390593274\\
0.4296875	0.109375	0.0353553390593274	-0.0353553390593274\\
0.4296875	0.359375	0.0353553390593274	0.0353553390593274\\
0.4296875	0.609375	-0.0353553390593274	-0.00132554437269138\\
0.4296875	0.859375	-0.0353553390593274	-0.0353553390593274\\
0.4296875	1.109375	-0.0353553390593274	-0.0353553390593274\\
0.4296875	1.359375	0.0353553390593274	0.0353553390593274\\
0.4296875	1.609375	-0.0353553390593274	-0.0353553390593274\\
0.4296875	1.859375	0.0353553390593274	-0.0353553390593274\\
0.5546875	0.109375	-0.0353553390593274	0.0353553390593274\\
0.5546875	0.359375	0.0353553390593274	-0.0353553390593274\\
0.5546875	0.609375	-0.0353553390593274	-0.0353553390593274\\
0.5546875	0.859375	-0.0353553390593274	0.0353553390593274\\
0.5546875	1.109375	-0.0353553390593274	0.0353553390593274\\
0.5546875	1.359375	0.0353553390593274	-0.0353553390593274\\
0.5546875	1.609375	0.0353553390593274	0.0353553390593274\\
0.5546875	1.859375	0.0353553390593274	-0.0353553390593274\\
0.6796875	0.109375	0.0353553390593274	0.0353553390593274\\
0.6796875	0.359375	-0.0353553390593274	-0.0353553390593274\\
0.6796875	0.609375	0.0353553390593274	-0.0353553390593274\\
0.6796875	0.859375	-0.0353553390593274	-0.0353553390593274\\
0.6796875	1.109375	-0.0353553390593274	0.0353553390593274\\
0.6796875	1.359375	-0.0353553390593274	0.0353553390593274\\
0.6796875	1.609375	0.0353553390593274	-0.0353553390593274\\
0.6796875	1.859375	0.0353553390593274	-0.0353553390593274\\
0.8046875	0.109375	-0.0353553390593274	0.0353553390593274\\
0.8046875	0.359375	-0.0353553390593274	-0.0353553390593274\\
0.8046875	0.609375	-0.0353553390593274	0.0353553390593274\\
0.8046875	0.859375	-0.0353553390593274	0.0353553390593274\\
0.8046875	1.109375	-0.0353553390593274	-0.0353553390593274\\
0.8046875	1.359375	-0.0353553390593274	0.0353553390593274\\
0.8046875	1.609375	-0.0353553390593274	0.0353553390593274\\
0.8046875	1.859375	0.0353553390593274	0.0353553390593274\\
0.9296875	0.109375	0.0353553390593274	0.0353553390593274\\
0.9296875	0.359375	0.0353553390593274	-0.0353553390593274\\
0.9296875	0.609375	-0.0353553390593274	-0.0353553390593274\\
0.9296875	0.859375	-0.0353553390593274	-0.0353553390593274\\
0.9296875	1.109375	0.0353553390593274	0.0353553390593274\\
0.9296875	1.359375	0.0353553390593274	0.0353553390593274\\
0.9296875	1.609375	-0.0353553390593274	-0.0353553390593274\\
0.9296875	1.859375	0.0353553390593274	-0.0353553390593274\\
};
\end{axis}
\end{tikzpicture}%
\\
        \input{elast_box_c3_deform.tikz}
        \caption{concentric, $\alpha = 10^{-5}$}\label{fig:elastic:6}
    \end{subfigure}
    \caption{Control (top rows: phase and magnitude color coded as shown in color wheel with values in $\calM$ indicated, additionally indicated by arrows) and state (bottom row: target deformation in gray, achieved deformation in red) for the elasticity model problem}
    \label{fig:elasticExamples}
\end{figure}

\Cref{fig:elastic:5} shows that the control is not guaranteed to take values in $\calM$; here, the target displacement $z$ is the displacement induced by a deadload to the left applied at the top domain boundary.
Since the target was induced by a forcing with zero load throughout the bulk material, the optimal control mainly takes the nonpreferred value of zero.
However, a slight random perturbation of $z$ again leads to a pure multibang control, as shown in \cref{fig:elastic:6}.

We again show the convergence behavior for the example in \cref{fig:elastic:3} in \cref{tab:iterElastic}. Since this example is linear, only a few Newton iterations ($2$ to $6$) are required for all values of $\gamma$, and correspondingly only a few line searches are carried out for $\gamma<10^{-5}$. As before, the multibang structure is already strongly promoted for $\gamma\approx 10^{-6}$. (Let us point out that the elastic body is fixed at the bottom boundary so that the control has to be $0$ there, which for this example does not lie in $\calM$.)
\begin{table}[b]
    \caption{Convergence behavior for the example in \cref{fig:elastic:3}: number of semi-smooth Newton steps, number of times a line search was required, and number of nodes with $u_\gamma(x)\notin\calM$}
    \label{tab:iterElastic}
    \centering
    \resizebox{\linewidth}{!}{%
        \begin{tabular}{lrrrrrrrrrrrrr}
            \toprule
            $\gamma$       & \num[round-precision = 0]{1.953e-1} & \num[round-precision = 0]{1.221e-2} & \num[round-precision = 0]{1.526e-3} & \num[round-precision = 0]{1.907e-4} & \num[round-precision = 0]{1.192e-5} & \num[round-precision = 0]{1.490e-6} & \num[round-precision = 0]{1.863e-7} & \num[round-precision = 0]{1.164e-8} & \num[round-precision = 0]{1.445e-9} & \num[round-precision = 0]{1.819e-10}\\
            \midrule
            \# SSN         & \num{2}    & \num{4}    & \num{5}    & \num{5}    & \num{4}   & \num{6}  & \num{4}  & \num{4}  & \num{5}  & \num{6} \\
            \# line search & \num{0}    & \num{0}    & \num{0}    & \num{0}    & \num{0}   & \num{2}  & \num{1}  & \num{2}  & \num{3}  & \num{4}\\
            \# not MB      & \num{4225} & \num{4210} & \num{3747} & \num{1245} & \num{179} & \num{84} & \num{71} & \num{68} & \num{68} & \num{68}  \\
            \bottomrule	
        \end{tabular}
    }
\end{table}

\subsection{Multimaterial branched transport}

To illustrate the behavior of our approach for multimaterial branched transport, we fix a random network
obtained by a random perturbation of vertices of a regular $10\times 10$ square grid and then performing a Delaunay triangulation (subsequently removing very long edges).
In our experiments we fix the vertices that will act as material sources or sinks
and assign them different materials, resulting in different optimal transportation schemes; see \cref{fig:btransport}.
The amount of each material in our example calculations is simply taken as $m_i=1$ for all $i$.
In contrast to the previous examples, we here pick $c(v)=|v|_2$ rather than the squared norm,
which leads to a preference for combined flows of multiple materials. We fix the control costs at $\alpha=10^{-3}$.

For the numerical solution, we start with a zero flux and a Moreau--Yosida parameter $\gamma_0=20$. The semismooth Newton systems are again solved iteratively using GMRES without restarts at a tolerance of $10^{-11}$; we include a backtracking line search with a minimal step size $\tau_{\min}=10^{-5}$. Starting with a reduction factor $q=0.5$, we adapt $\gamma$ and $q$ as follows: If the Newton method for $\gamma_n$ did not converge within $20$ iterations or the minimal step size did not lead to a reduction of the residual norm, we discard the iterate, set $q=q^{0.25}$, and restart with $u^{n-1}$ and $\gamma_{n-1}q$. If the Newton iteration converged (with a residual norm smaller than $\min\{\gamma_n,10^{-6}\}$ or a relative residual norm smaller than $10^{-9}$, whichever occurs first) within $15$ steps, we accept the iterate, set $q = \min\{q^{0.75},1-10^{-4}\}$, and compute $u^{n+1}$ with $u^n$ as starting value and $\gamma_{n+1} = \gamma_n q$. If the Newton iteration even converged within $5$ steps, we continue similarly but with reduced $q = \min\{1-10^{-3},\max\{q^{1.25},0.5\}\}$. Otherwise we continue with $q = \min\{1-10^{-3},q\}$. We terminate the path-following at $\gamma<10^{-7}$. Again, this is a heuristic procedure that worked well for this example; in all reported cases, the deviations from the desired discrete control values are less than $0.006$ on each edge.

The results are shown in \cref{fig:btransport} for different configurations of sources and sinks.
For three sources at the bottom and three sinks in reversed order at the top, all mass flows converge along the optimal transportation path (\cref{fig:btransport:a}).
If the flow of $m_3$ is reversed by swapping its source and sink, there is no longer a payoff by joint transport so that $m_3$ is transported independently of the other materials (\cref{fig:btransport:b}).
If instead the order of the sinks is reversed it becomes more economic for $m_3$ to take the direct route than to create a flow with all materials (\cref{fig:btransport:c}).
Finally, the network with an additional fourth material becomes more complicated (\cref{fig:btransport:d}).

\begin{figure}
    \centering
    \setlength\unitlength{.24\linewidth}
    \begin{subfigure}[t]{0.24\linewidth}
        \begin{picture}(1,1)
            \put(0,0){\includegraphics[width=\unitlength]{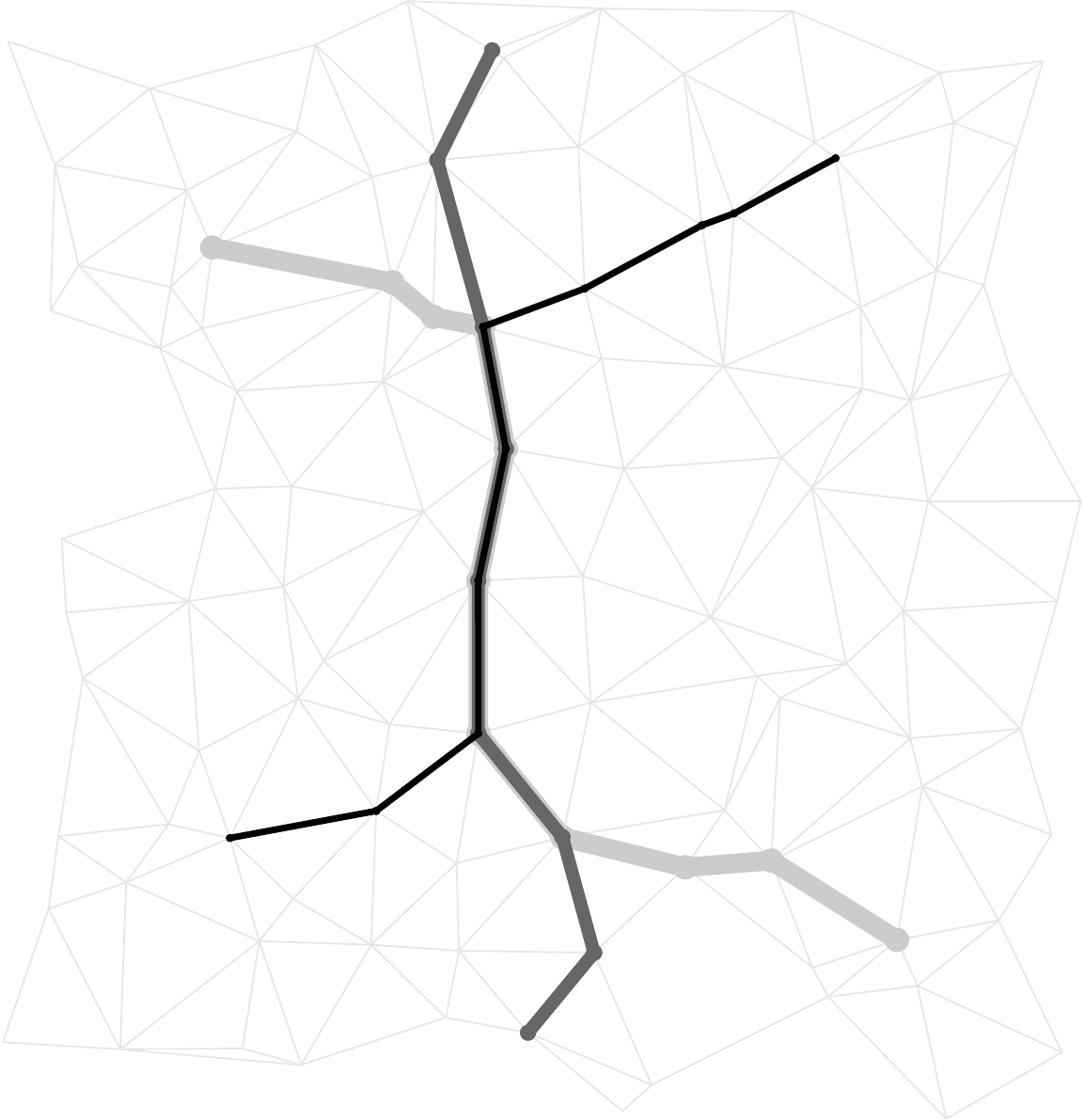}}
            \put(.0,.2){\small$+m_1$}
            \put(.4,.0){\small$+m_2$}
            \put(.75,.1){\small$+m_3$}
            \put(.65,.9){\small$-m_1$}
            \put(.42,.98){\small$-m_2$}
            \put(.0,.85){\small$-m_3$}
        \end{picture}%
        \caption{}\label{fig:btransport:a}
    \end{subfigure}
    \hfill
    \begin{subfigure}[t]{0.24\linewidth}
        \begin{picture}(1,1)
            \put(0,0){\includegraphics[width=\unitlength]{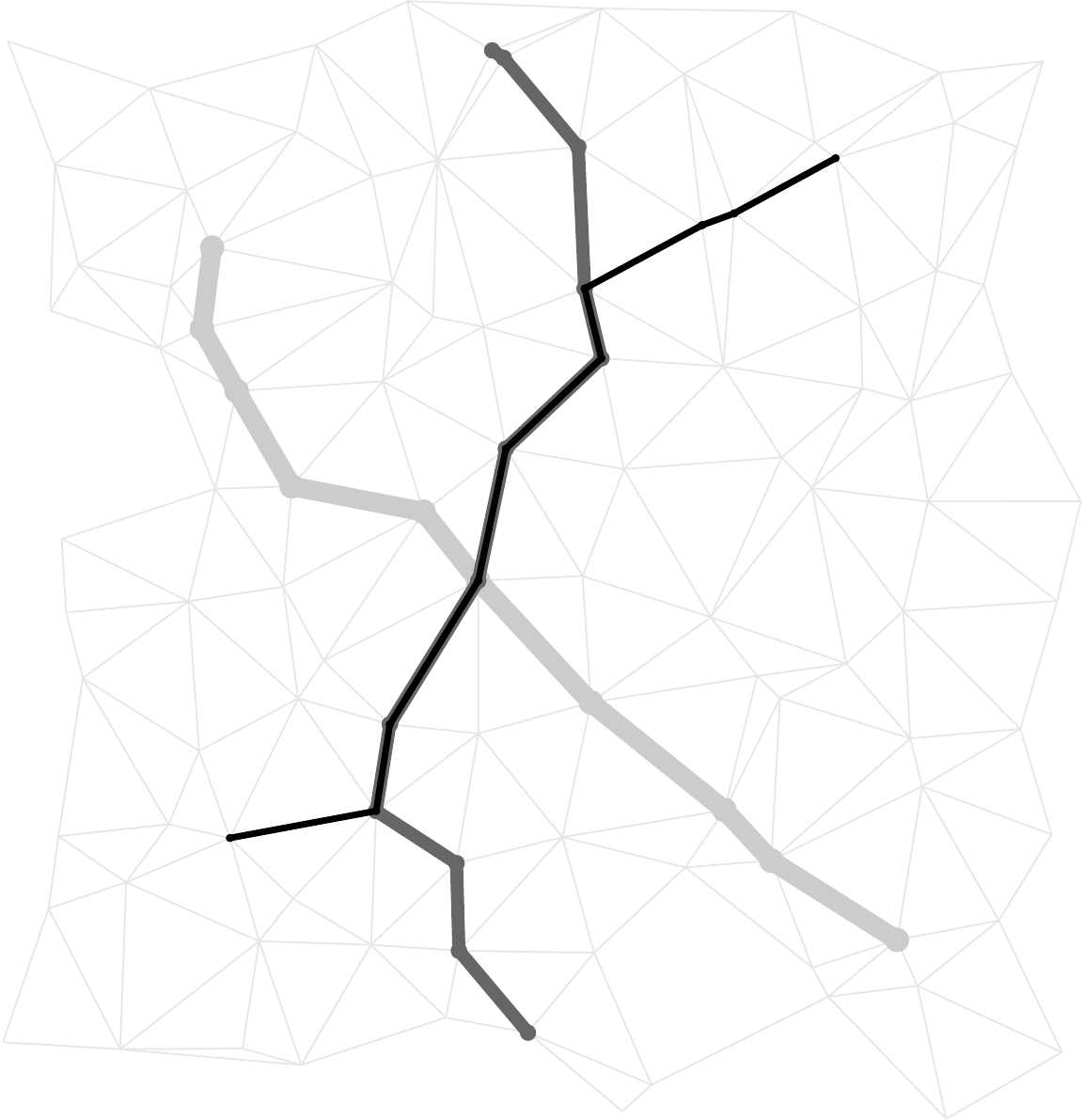}}
            \put(.0,.2){\small$+m_1$}
            \put(.4,.0){\small$+m_2$}
            \put(.75,.1){\small$-m_3$}
            \put(.65,.9){\small$-m_1$}
            \put(.42,.98){\small$-m_2$}
            \put(.0,.85){\small$+m_3$}
        \end{picture}%
        \caption{}\label{fig:btransport:b}
    \end{subfigure}
    \hfill
    \begin{subfigure}[t]{0.24\linewidth}
        \begin{picture}(1,1)
            \put(0,0){\includegraphics[width=\unitlength]{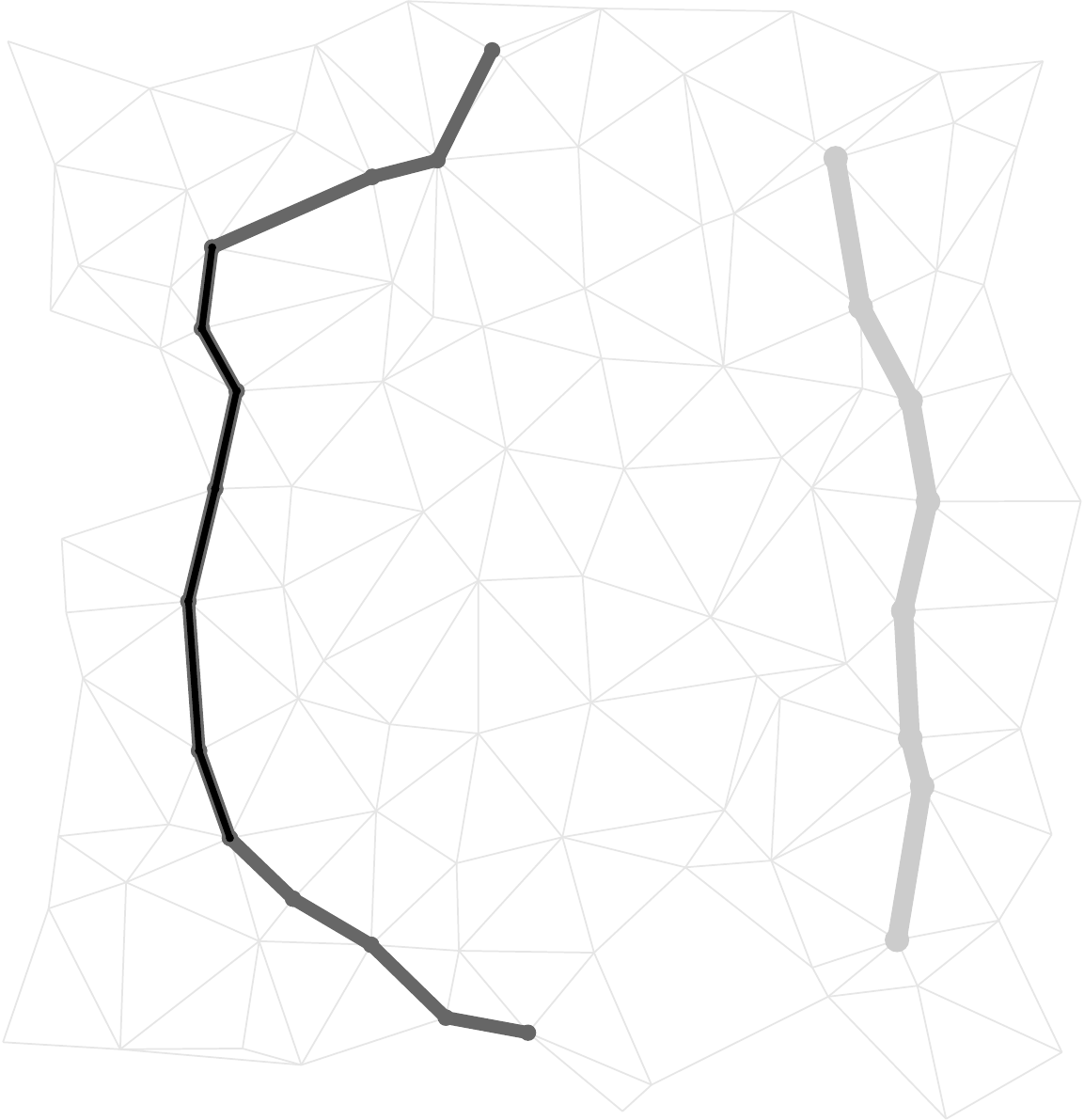}}
            \put(.0,.2){\small$+m_1$}
            \put(.4,.0){\small$+m_2$}
            \put(.75,.1){\small$+m_3$}
            \put(.65,.9){\small$-m_3$}
            \put(.42,.98){\small$-m_2$}
            \put(.0,.85){\small$-m_1$}
        \end{picture}%
        \caption{}\label{fig:btransport:c}
    \end{subfigure}
    \hfill
    \begin{subfigure}[t]{0.24\linewidth}
        \begin{picture}(1,1)
            \put(0,0){\includegraphics[width=\unitlength]{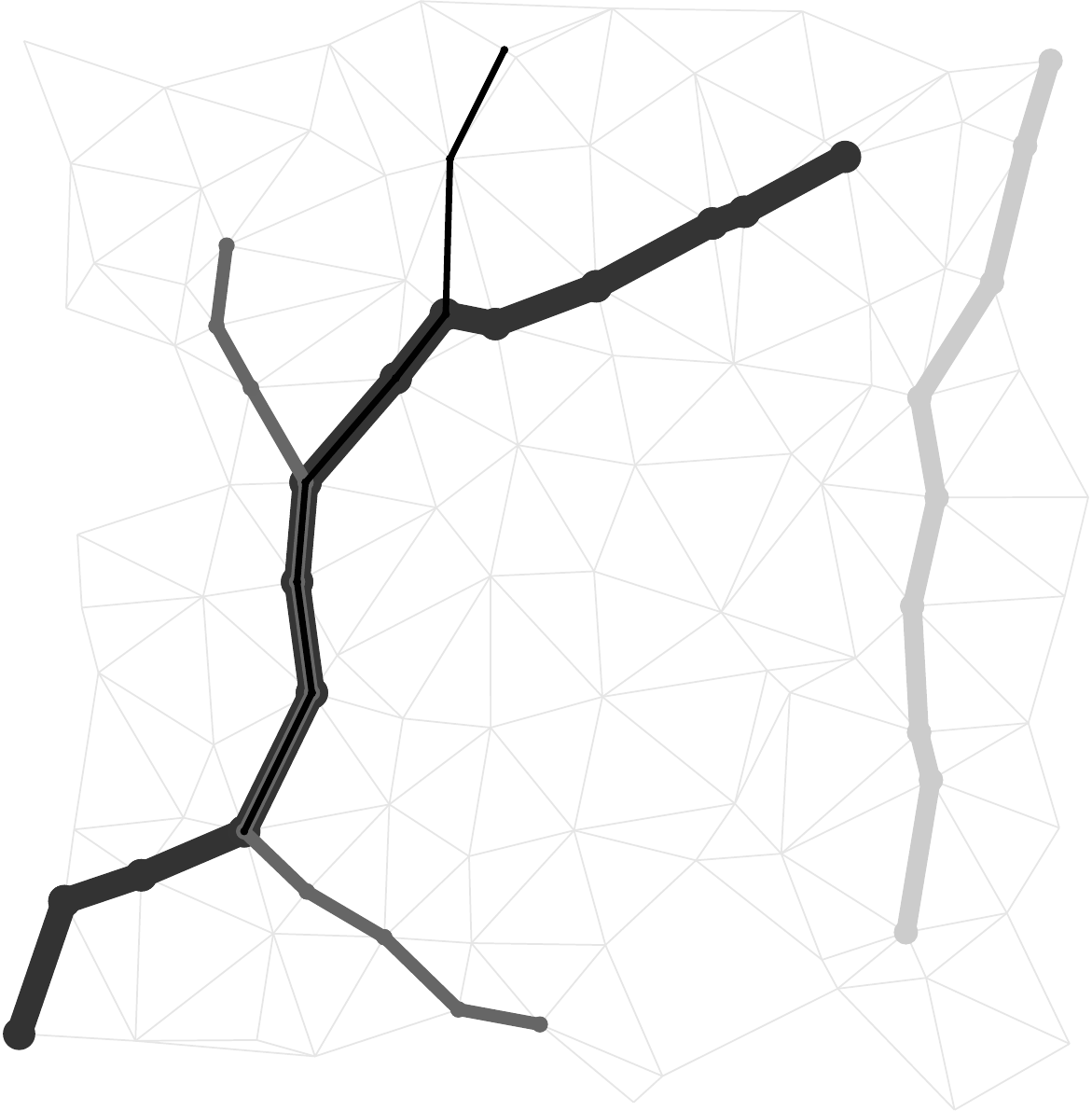}}
            \put(.25,.23){\small$+m_1$}
            \put(.4,.0){\small$+m_2$}
            \put(.75,.1){\small$+m_3$}
            \put(.0,.0){\small$+m_4$}
            \put(.65,.9){\small$-m_4$}
            \put(.42,.98){\small$-m_1$}
            \put(.0,.85){\small$-m_2$}
            \put(.85,.98){\small$-m_3$}
        \end{picture}%
        \caption{}\label{fig:btransport:d}
    \end{subfigure}
    \caption{Optimal material flows on a fixed random network with fixed locations but different permutations of sources ($+m_i$) and sinks ($-m_i$); flows corresponding to different materials are indicated by lines of different widths and gray values.}
    \label{fig:btransport}
\end{figure}
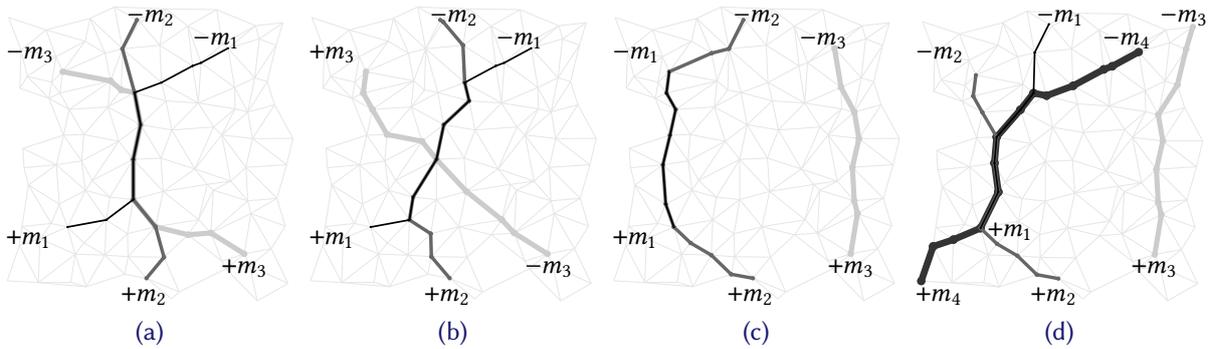

\section{Conclusion}

A preference for a small number of predefined discrete control values can be achieved by a piecewise affine pointwise regularization term whose corners lie at the preferred values.
In contrast to the case of scalar controls treated in \cite{CK:2013,CK:2015,CKK:2017,CKT:2019}, the case of vector-valued controls allows giving multiple control values equal preference,
and numerical experiments for optimal control of the Bloch equation, optimal control of elastic deformation, and multimaterial branched transport show that this feature is indeed exhibited by the optimal control in a broad range of practically relevant scenarios.
Furthermore, the optimal control problems leading to admissible controls turn out to be dense among a family of control problems.
A more precise characterization of control problems with admissible solutions would be desirable and should be further investigated.
For instance, for certain control problems such as the elasticity-based example, one might conjecture that targets leading to nonmultibang controls are nowhere dense.
In the context of nonsmooth optimization, investigating rigorous globalization of the semismooth Newton method by the path-following approach followed for the branched transport example would be worthwhile. Finally, an interesting topic for follow-up work would be combining the vector-valued results presented in this work with the techniques from \cite{CKK:2017} for topology optimization of elastic composite materials.

\appendix

\section{Properties of the Bloch equation}\label{sec:BlochProperties}
Here we verify that the state operator \eqref{eq:settingBloch} satisfies the required assumptions \ref{enm:weakWeakContinuity}--\ref{enm:adjointConvergence}.
In the following, a subscript to $\mathbf M^{(\omega)}$ and $\mathbf B^{(\omega)}$ always refers to the chosen control $u$.

\begin{proposition}\label{thm:operatorProperties}
    The operator $S$ as defined in \eqref{eq:settingBloch} is well-defined and satisfies {\ref{enm:weakWeakContinuity}--\ref{enm:compactness}.}
\end{proposition}
\begin{proof}
    Introducing the skew-symmetric matrix
    \begin{equation*}
        B_u^\omega(t)=\begin{pmatrix}0&\omega&-(u(t))_2\\-\omega&0&(u(t))_1\\(u(t))_2&-(u(t))_1&0\end{pmatrix},
    \end{equation*}
    the homogeneous linear Bloch equation $\frac{\dd}{\dd t}{\mathbf{M}}^{(\omega)}_u(t) = B_u^\omega(t)\mathbf{M}^{(\omega)}_u(t)$ for a control $u(t)\in\R^2$ has a solution $\mathbf{M}^{(\omega)}_u(t)$ by  Carath\'eodory's existence theorem.
    Furthermore,
    \begin{equation*}
        \frac{\dd}{\dd t}|\mathbf{M}^{(\omega)}_u(t)|_2^2=2\mathbf{M}^{(\omega)}_u(t)\cdot\frac{\dd}{\dd t}{\mathbf{M}}^{(\omega)}_u(t)=0,
    \end{equation*}
    and thus $|\mathbf{M}^{(\omega)}_u(t)|_2=1$ for all $t$.
    Now let $u_i\rightharpoonup u$ weakly in $L^2(\Omega;\R^2)$. Then
    \begin{equation*}
        \left\{\begin{aligned}
                \frac{\dd}{\dd t}{\mathbf{M}}^{(\omega)}_{u_i}(t)-\frac{\dd}{\dd t}{\mathbf{M}}^{(\omega)}_{u}(t)
                &=B_{u_i}^\omega(t)\left(\mathbf{M}^{(\omega)}_{u_i}(t)-\mathbf{M}^{(\omega)}_{u}(t)\right)
                +(B_{u_i}^\omega(t)-B_u^\omega(t))\mathbf{M}^{(\omega)}_{u}(t), \quad t\in[0,T],\\
                \mathbf{M}^{(\omega)}_{u_i}(0) &= \mathbf{M}^{(\omega)}_{u}(0).
        \end{aligned}\right.
    \end{equation*}
    Upon abbreviating $\Delta M_i=\mathbf{M}^{(\omega)}_{u_i}-\mathbf{M}^{(\omega)}_{u}$ and $\Delta B_i=(B_{u_i}^\omega-B_u^\omega)\mathbf{M}^{(\omega)}_{u}$ and integrating from $0$ to $t$, we arrive at
    \begin{equation*}
        \begin{aligned}[t]
            |\Delta M_i(t)|_2
            &=\left|\int_0^tB_{u_i}^\omega(s)\Delta M_i(s)\,\dd s+\int_0^t\Delta B_i(s)\,\dd s\,\right|_2\\
            &\leq\int_0^t|B_{u_i}^\omega(s)|_2|\Delta M_i(s)|_2\,\dd s+\left|\int_0^t\Delta B_i(s)\,\dd s\, \right|_2.
        \end{aligned}
    \end{equation*}
    Gronwall's inequality now implies that
    \begin{equation}\label{eq:Gronwall}
        \left|\Delta M_i(t)\right|_2
        \leq\left|\int_0^t\Delta B_i(s)\,\dd s\,\right|_2
        +\int_0^t\left|\int_0^r\Delta B_i(s)\,\dd s\,\right|_2|B_{u_i}^\omega(r)|_2\exp\left(\int_r^t|B_{u_i}^\omega(s)|_2\,\dd s\right)\,\dd r.
    \end{equation}
    The first term converges to zero due to $\Delta B_i\rightharpoonup0$ in $L^2(\Omega;\R^{3})$ (since $\mathbf{M}^{(\omega)}_{u}\in L^\infty(\Omega;\R^3)$).
    Additionally, the exponential is bounded by $\exp(\sqrt T\|B_{u_i}^\omega\|_{L^2(\Omega;\R^{3\times3})})\leq C\in\R$ independent of $i$.
    Thus, the right-hand side converges to zero if
    \begin{equation*}
        f_i\to0\text{ in }L^2(\Omega;\R)
        \qquad\text{for}\qquad
        f_i:\Omega\to\R,\quad r\mapsto\int_0^r\Delta B_i(s)\,\dd s.
    \end{equation*}
    This is indeed the case since
    \begin{equation*}
        \|f_i\|_{L^2(\Omega)}^2=\int_{\{s\in(0,T)^3:s_1,s_2\leq s_3\}}\Delta B_i(s_1)\cdot\Delta B_i(s_2)\,\dd s
    \end{equation*}
    and $s\mapsto \Delta B_i(s_1)\cdot\Delta B_i(s_2)$ converges weakly to zero in $L^2((0,T)^3;\R)$.
    Thus $\mathbf M^{(\omega_j)}_{u_i}(T)$ converges for all $j$, and therefore $S(u_i)\to S(u)$.
    This argument also implies uniqueness of the solution.

    Moreover, $S$ is Fr\'echet differentiable, and its derivative at $u\in L^2(\Omega;\R^2)$ is given by
    \begin{equation*}
        S'(u):U\to Y,\quad
        \phi\mapsto \delta\mathbf M_\phi(T)=(\delta\mathbf M_\phi^{(\omega_1)}(T),\ldots,\delta\mathbf M_\phi^{(\omega_J)}(T))
        \quad
    \end{equation*}
    with $\delta\mathbf M_\phi^{(\omega)}$ solving the linearized state equation (note $\partial_u(B_u^\omega)(\phi)=B_\phi^0$)
    \begin{equation}\label{eq:BlochFrechetEquation}
        \left\{\begin{aligned}
                \tfrac{\dd}{\dd t} \delta\mathbf M_\phi^{(\omega)}(t)&=B_{u}^{\omega}(t)\delta\mathbf M_\phi^{(\omega)}(t)+B_\phi^0(t)\mathbf M^{(\omega)}_{u}(t)\,,\qquad t\in[0,T],\\
                \delta\mathbf M_\phi^{(\omega)}(0)&=(0,0,0)^T.
        \end{aligned}\right.
    \end{equation}
    Indeed, $\delta\mathbf M_\phi(T)$ is obviously linear in $\phi$, and the unique solvability follows just like for $\mathbf M_u^{(\omega)}$.
    Furthermore, for any $\tilde u\in U$ with $\|\tilde u-u\|_{U}\leq1$ and $\phi=\tilde u-u$ we have
    \begin{equation*}
        \tfrac\dd{\dd t}(\mathbf{M}^{(\omega)}_{\tilde u}-\mathbf{M}^{(\omega)}_{u}-\delta\mathbf M_\phi^{(\omega)})
        =B_{\tilde u}^{\omega}(\mathbf{M}^{(\omega)}_{\tilde u}-\mathbf{M}^{(\omega)}_{u}-\delta\mathbf M_\phi^{(\omega)})
        +(B_{\tilde u}^{\omega}-B_{u}^{\omega})\delta\mathbf M_\phi^{(\omega)}
    \end{equation*}
    with zero initial condition.
    Gronwall estimates analogous to \eqref{eq:Gronwall} (now for $\delta\mathbf M_\phi^{(\omega)}$ and $\mathbf{M}^{(\omega)}_{\tilde u}-\mathbf{M}^{(\omega)}_{u}-\delta\mathbf M_\phi^{(\omega)}$, exploiting that $|B_{\tilde u}^\omega(r)|_2\exp(\int_r^t|B_{\tilde u}^\omega(s)|_2\,\dd s)$ is bounded by a constant only depending on $\|u\|_{U}$) imply that
    \begin{equation*}
        |\delta\mathbf M_\phi^{(\omega)}(t)|_2\leq\tilde C\|B_{\tilde u}^{\omega}-B_{u}^{\omega}\|_{L^2(\Omega;\R^{3\times 3})}\leq2\tilde C\|\tilde u-u\|_{U}
    \end{equation*}
    for a constant $\tilde C>0$ and all $t\in\Omega$ as well as
    \begin{equation*}
        \begin{aligned}[t]
            |\mathbf{M}^{(\omega)}_{\tilde u}(T)-\mathbf{M}^{(\omega)}_{u}(T)-\delta\mathbf M_\phi^{(\omega)}(T)|_2
            &\leq C\sup_{t\in\Omega}\left|\int_0^t(B_{\tilde u}^{\omega}(s)-B_{u}^{\omega}(s))\delta\mathbf M_\phi^{(\omega)}(s)\,\dd s\,\right|_2\\
            &\leq C\|B_{\tilde u}^{\omega}-B_{u}^{\omega}\|_{L^1(\Omega;\R^{3\times 3})}\|\delta\mathbf M_\phi^{(\omega)}\|_{L^\infty(\Omega;\R^3)}\\
            &\leq C\|\tilde u-u\|_{U}^2,
        \end{aligned}
    \end{equation*}
    where $C$ denotes a positive constant (not necessarily the same in all inequalities).
    We thus have
    \begin{equation*}
        |S(\tilde u)-S(u)-S'(u)(\tilde u-u)|_2\leq C\|\tilde u-u\|_{U}^2
    \end{equation*}
    as required.
    The compactness follows from the finite dimensionality of $\ran S$.
\end{proof}

We will also require some regularity results for the adjoint operator $S'(u)^*$.

\begin{proposition}\label{prop:bloch_adjoint}
    For $S$ from \eqref{eq:settingBloch} and $u\in U$ we have
    \begin{align*}
        &S'(u)^*:Y\to U,\\
        &(S'(u)^*y)(t)=\sum_{j=1}^J\left(\begin{smallmatrix}0&\left(\mathbf M_u^{(\omega_j)}(t)\right)_3&-\left(\mathbf M_u^{(\omega_j)}(t)\right)_2\\-\left(\mathbf M_u^{(\omega_j)}(t)\right)_3&0&\left(\mathbf M_u^{(\omega_j)}(t)\right)_1\end{smallmatrix}\right)\mathbf{\Psi}_{u,j}(t),
    \end{align*}
    where $\mathbf{\Psi}_{u,j}$ solves the adjoint equation
    \begin{equation*}
        \frac{\dd}{\dd t}\mathbf{\Psi}_{u,j}(t)=\mathbf{\Psi}_{u,j}(t)\times\mathbf B_u^{(\omega_j)}(t),\quad
        \mathbf{\Psi}_{u,j}(T)=y_j,\quad
        j=1,\ldots,J.
    \end{equation*}
\end{proposition}
\begin{proof}
    From \cref{thm:operatorProperties} we have $S'(u)\phi=(\delta\mathbf M_\phi^{(\omega_1)}(T),\ldots,\delta\mathbf M_\phi^{(\omega_J)}(T))$ for any $u,\phi\in U$ with $\delta\mathbf M_\phi^{(\omega)}$ solving \eqref{eq:BlochFrechetEquation}.
    Thus we obtain for $y\in (\R^3)^J$ that
    \begin{equation*}
        \begin{aligned}[t]
            \int_\Omega \phi(t)\cdot(S'(u)^* y)(t)\,\dd t
            &=\scalprod{ y,S'(u)\phi}
            =\sum_{j=1}^J y_j^T\delta\mathbf M_\phi^{(\omega_j)}(T)\\
            &=\sum_{j=1}^J\mathbf{\Psi}_{u,j}(T)^T\delta\mathbf M_\phi^{(\omega_j)}(T)\\
            &=\sum_{j=1}^J\int_\Omega\mathbf{\Psi}_{u,j}(t)^T\frac{\dd}{\dd t} \delta\mathbf M_\phi^{(\omega_j)}(t)+\frac{\dd}{\dd t}\mathbf{\Psi}_{u,j}(t)^T\delta\mathbf M_\phi^{(\omega_j)}(t)\,\dd t\\
            &=\sum_{j=1}^J\int_\Omega\mathbf{\Psi}_{u,j}(t)^T\left[\frac{\dd}{\dd t} \delta\mathbf M_\phi^{(\omega_j)}(t)-\delta\mathbf M_\phi^{(\omega_j)}(t)\times\mathbf B_u^{\omega_j}(t)\right]\,\dd t\\
            & =\sum_{j=1}^J\int_\Omega\mathbf{\Psi}_{u,j}(t)^T\left[B_\phi^0(t)\mathbf M_u^{(\omega_j)}(t)\right]\,\dd t
        \end{aligned}
    \end{equation*}
    from which the result follows.
\end{proof}

\begin{proposition}\label{prop:blochAdjointContinuity}
    For any $u\in U$, we have $\ran S'(u)^*\hookrightarrow L^\infty(\Omega;\R^2)$.
    Moreover, $u\mapsto S'(u)^*$ is continuous in $L^\infty(\Omega;\R^2)$ under weak convergence of $u$ in $U$ and thus it satisfies \ref{enm:adjointConvergence} with $V=L^1(\Omega;\R^2)$.
\end{proposition}
\begin{proof}
    By the formula for $S'(u)^*$ from \cref{prop:bloch_adjoint}, it is enough to show that $\mathbf M_{u_i}^{(\omega_j)}$ and $\mathbf{\Psi}_{u_i,j}$ converge in $L^\infty(\Omega;\R^3)$ as $u_i\rightharpoonup u$ in $U$.
    It suffices to consider $\mathbf M_{u_i}^{(\omega_j)}$, since the adjoint variable $\mathbf{\Psi}_{u,j}$ satisfies the same differential equation.
    Thus, we only have to show that the right-hand side in \eqref{eq:Gronwall} converges to zero uniformly in $t$.
    Note that the second integral is bounded above by the one for $t=T$ which has already been shown to converge to zero. Hence it suffices to show $\int_0^t\Delta B_i(s)\,\dd s\to0$ uniformly in $t$ as $i\to\infty$.
    Since $\Delta B_i\rightharpoonup0$ in $L^2(\Omega;\R^3)$, we also have weak convergence in $L^1(\Omega;\R^3)$ so that by the Dunford--Pettis criterion the $\Delta B_i$ are equi-integrable.
    Now let $t_i\in[0,T]$ be such that $|\int_0^{t_i}\Delta B_i(s)\,\dd s\,|_2\geq\sup_{t\in[0,T]}|\int_0^{t}\Delta B_i(s)\,\dd s\,|_2-\frac1i$,
    and assume that for a subsequence (still indexed by $i$) we have $|\int_0^{t_i}\Delta B_i(s)\,\dd s\,|_2\geq C>0$ for all $i$.
    Upon taking another subsequence, we can further assume that $t_i\to\hat t\in[0,T]$.
    Due to the equi-integrability, there is a $\Delta t>0$ such that $\int_{\hat t-\Delta t}^{\hat t+\Delta t}|\Delta B_i(s)|_2\,\dd s<C/2$;
    thus for $i$ large enough we have $|\int_0^{\hat t}\Delta B_i(s)\,\dd s\,|_2\geq C/2$.
    However, this contradicts the weak convergence of $\Delta B_i$ to $0$ so that indeed $\int_0^t\Delta B_i(s)\,\dd s\to0$ uniformly in $t$ as\break $i\to\infty$.
\end{proof}

\bibliographystyle{jnsao}
\bibliography{multibang}

\end{document}